\documentclass[10pt]{amsart}
\usepackage{amsmath,amssymb,latexsym,cancel,rotating}
\usepackage{graphicx,amssymb,mathrsfs,amsmath,color,fancyhdr,amsthm}
\usepackage[all]{xy}

\newtheorem{theorem}{Theorem}[section]
\newtheorem{lemma}[theorem]{Lemma}
\newtheorem{corollary}[theorem]{Corollary}
\newtheorem{definition}[theorem]{Definition}
\newtheorem{proposition}[theorem]{Proposition}

\newtheorem{remark}[theorem]{Remark}

\begin{document}
\title[Lusztig sheaves and  integrable highest weight modules]
{Lusztig sheaves and  integrable highest weight modules}
\author[Fang,Lan,Xiao]{Jiepeng Fang,Yixin Lan,Jie Xiao}
\address{School of mathematical secience, Peking University, Beijing 100871, P. R. China}
\email{fangjp@math.pku.edu.cn (J.Fang)}

\address{Academy of Mathematics and Systems Science, Chinese Academy of Sciences, Beijing 100190, P.R.China}
 \email{lanyixin@amss.ac.cn (Y.Lan)}

\address{School of mathematical seciences, Beijing Normal University, Beijing 100875, P. R. China}
\email{jxiao@bnu.edu.cn (J.Xiao)}

\begin{abstract}
	We consider the localization $\mathcal{Q}_{\mathbf{V},\mathbf{W}}/\mathcal{N}_{\mathbf{V}}$ of Lusztig's sheaves for framed quivers, and define functors  $E^{(n)}_{i},F^{(n)}_{i},K^{\pm}_{i},n\in \mathbb{N},i \in I$ between the localizations. With these functors, the Grothendieck group of localizations realizes the irreducible integrable highest weight modules $L(\Lambda)$ of quantum groups. Moreover, the nonzero simple perverse sheaves in localizations form the canonical bases of $L(\Lambda)$. We also compare our realization (at $v \rightarrow 1$) with Nakajima's realization via quiver varieties and prove that  the transition matrix between canonical bases and fundamental classes is upper triangular with diagonal entries all equal to $\pm 1$.
\end{abstract}



\date{\today}

\bibliographystyle{abbrv}

\maketitle

\setcounter{tocdepth}{1}\tableofcontents

\section{Introduction}\label{sec:intro}
Given a symmetric Cartan datum $(I,(-,-))$, we can associate an acyclic quiver $Q$ and define its Kac-Moody Lie algebra $\mathfrak{g}$ and the quantized enveloping algebra (or quantum group) $\mathbf{U}=\mathbf{U}_{v}(\mathfrak{g})$. In \cite{MR1088333},\cite{MR1227098} and \cite{MR1653038}, G.Lusztig  considered the moduli space $\mathbf{E}_{\mathbf{V},\Omega}$ of quiver representations and introduced his category $\mathcal{Q}_{\mathbf{V}}$  of semisimple perverse sheaves (complexes).  Perverse sheaves in $\mathcal{Q}_{\mathbf{V}}$ are called Lusztig sheaves. By using Grothendieck's six operators, the induction functor  $\mathbf{Ind}^{\mathbf{V}}_{\mathbf{V}',\mathbf{V}''}$ and the restriction functor $\mathbf{Res}^{\mathbf{V}}_{\mathbf{T},\mathbf{W}}$ have been defined. Together with the induction and restriction functors, the Grothendieck group $\mathcal{K}$ of the category $\coprod \limits_{\mathbf{V}}\mathcal{Q}_{\mathbf{V}}$ becomes a bialgebra, which is canonically isomorphic to the integral form ${_{\mathcal{A}}\mathbf{U}^{+}}$ (or ${_{\mathcal{A}}\mathbf{U}^{-}}$) of the positive (or negative) part of the quantum group. (Here $\mathcal{A}=\mathbb{Z}[v,v^{-1}]$.) Moreover, the set $\mathcal{P}$ of simple Lusztig sheaves forms a basis of ${_{\mathcal{A}}\mathbf{U}^{+}}$,  which is called the canonical basis by Lusztig. The canonical basis has many remarkable properties, such as integral property and positive property.

Given a dominant weight $\Lambda$, one can define the irreducible highest weight module $L(\Lambda)$. Even though Lusztig didn't provide a categorification of $L(\Lambda)$, he constructed the canonical basis of $L(\Lambda)$. In fact, if we identify $\mathcal{K}$ with ${_{\mathcal{A}}\mathbf{U}}^{-}$ (viewed as the Verma module) and consider the canonical map $$\pi: {_{\mathcal{A}}\mathbf{U}}^{-} \rightarrow {_{\mathcal{A}}\mathbf{U}}^{-}/ \sum\limits_{i \in I } {_{\mathcal{A}}\mathbf{U}}^{-} f_{i}^{\langle \Lambda, \alpha_{i}^{\vee} \rangle +1} \cong {_{\mathcal{A}}L}(\Lambda) $$  
then $\{ \pi([L])\neq 0| L\in \mathcal{P} \}$ forms a basis of ${_{\mathcal{A}}L}(\Lambda)$. However, a categorical realization  of ${_{\mathcal{A}}L}(\Lambda)$ and its canonical basis is still expected.

H.Zheng took a breakthrough in his work \cite{MR3200442}. He categorified the irreducible integrable highest weight modules and their tensor products by using classes of micro-local perverse sheaves $\mathfrak{D}_{\overrightarrow{\omega}}$ on moduli stacks of framed quivers. Later, Y.Li in \cite{MR3177922}  pointed out that there is a vector bundle $\pi_{\mathbf{W}}$ which implies the relation between Zheng's construction and Lusztig's canonical basis.   \nocite{MR4379282}

Compared with Lusztig's theory, another approach to categorify the quantum group is to use the projective representations of quiver Hecke algebras \cite{MR2525917} ,\cite{MR2763732}, \cite{MR2837011}. S-J.Kang and M.Kashiwara considered the cyclotomic quiver Hecke algebras $R^{\Lambda}$ in \cite{MR2995184}, which is a quotient of quiver Hecke algebras. They defined functors $F^{\Lambda}_{i},E^{\Lambda}_{i},i \in I$ for the category of modules of cyclotomic quiver Hecke algebras, which satisfy the following relation
\begin{equation*}
	q_{i}^{-2}F^{\Lambda}_{i}E^{\Lambda}_{i} \oplus \bigoplus \limits_{k \geqslant 0}^{ \langle h_{i},\Lambda \rangle -1} q_{i}^{2k} Id \cong E^{\Lambda}_{i} F^{\Lambda}_{i},  \langle h_{i},\Lambda \rangle \geqslant 0,
\end{equation*}
\begin{equation*}
	q_{i}^{-2}F^{\Lambda}_{i}E^{\Lambda}_{i}  \cong E^{\Lambda}_{i} F^{\Lambda}_{i} \oplus \bigoplus \limits_{k \geqslant 0}^{ -\langle h_{i},\Lambda \rangle -1} q_{i}^{2k} Id,  \langle h_{i},\Lambda \rangle \leqslant 0.
\end{equation*}
Then the Grothendieck group of projective modules becomes a $\mathbf{U}$-module and there is an isomorphism $[Proj(R^{\Lambda})] \cong {_{\mathcal{A}}L(\Lambda)}.$ The key ingredient of their construction is the following exact sequence (See \cite[Theorem 4.7]{MR2995184} for details.) 
\begin{equation*}
	0 \rightarrow \bar{F}_{i}M \rightarrow F_{i}M \rightarrow F^{\Lambda}_{i}M \rightarrow 0,
\end{equation*}
which categorifies the following equation 
\begin{equation*}
	[e_{i}, P]= \frac{K^{-1}_{i}e'_{i}(P)-K_{i}e''_{i}(P) }{q_{i}^{-1}-q_{i}}.
\end{equation*}
This provides a successful model to categorify $L(\Lambda)$ \nocite{MR3084241}. 

Inspired by H. Zheng's work \nocite{zheng2007geometric} \cite{MR3200442} and with our puzzles about his proof, we go back to Lusztig's theory and give a self-contained categorical realization of $L(\Lambda)$ in the present paper. We define  a certain localization $\mathcal{L}_{\mathbf{V}}(\Lambda)=\mathcal{Q}_{\mathbf{V},\mathbf{W}}/\mathcal{N}_{\mathbf{V}}$ of Lusztig's category on the moduli space of framed quivers. Here, $\mathcal{N}_{\mathbf{V}}$ is a  thick subcategory such that objects in $\mathcal{Q}_{\mathbf{V},\mathbf{W}} \cap \mathcal{N}_{\mathbf{V}}$ correspond to elements in the left ideal $\sum\limits_{i \in I } \mathbf{U}^{-} f_{i}^{\langle \Lambda, \alpha_{i}^{\vee} \rangle +1}$, hence the Grothendieck group of $\mathcal{Q}_{\mathbf{V},\mathbf{W}}/\mathcal{N}_{\mathbf{V}}$ corresponds to the weight space of $L(\Lambda)$ via Lusztig's categorification of $\mathbf{U}^{-}$. 

We also need to define functors $E^{(n)}_{i},F^{(n)}_{i}$ and $K_{i}^{\pm}$, which realize the operators $E^{(n)}_{i},F^{(n)}_{i}$ and $K_{i}^{\pm}$ on $L(\Lambda)$ respectively. Recall that  ${_{\mathcal{A}}\mathbf{U}}^{-}/ \sum\limits_{i \in I } {_{\mathcal{A}}\mathbf{U}}^{-} f_{i}^{\langle \Lambda, \alpha_{i}^{\vee} \rangle +1}$ is naturally an ${_{\mathcal{A}}\mathbf{U}}^{-}$-module, hence to make ${_{\mathcal{A}}\mathbf{U}}^{-}/ \sum\limits_{i \in I } {_{\mathcal{A}}\mathbf{U}}^{-} f_{i}^{\langle \Lambda, \alpha_{i}^{\vee} \rangle +1}$ an ${_{\mathcal{A}}\mathbf{U}}$-module, the space $\sum\limits_{i \in I } {_{\mathcal{A}}\mathbf{U}}^{-} f_{i}^{\langle \Lambda, \alpha_{i}^{\vee} \rangle +1}$ needs to be an ${_{\mathcal{A}}\mathbf{U}}^{+}$-submodule, which can be proved by  some detailed calculation at algebraic level.  This problem is more difficult at categorical level. In order to overcome this difficulty, we study the commutative relations between some new functors $\mathcal{F}^{(n),\vee}_{j}$ and $\mathcal{E}_{i,r,n}$ in Lemma \ref{commute4} and generalize Lusztig's key lemma to non-semisimple objects in Lemma \ref{key0}. After that, we prove our Proposition \ref{keypro}, which tells us that the functor $E^{(n)}_{i}$ sends objects of $\mathcal{N}_{\mathbf{V}}$ to those of $\mathcal{N}_{\mathbf{V}'}$. In particular, $E^{(n)}_{i}$ is a well-defined functor on the localization. Our first main theorem is the following:
\begin{theorem}
	With the action of linear operators induced by functors $E^{(n)}_{i},F^{(n)}_{i},K^{\pm}_{i}$ for $n\in \mathbb{N},i \in I$, the Grothendieck group $\mathcal{K}_{0}(\Lambda)$ of $\coprod\limits_{\mathbf{V}}\mathcal{Q}_{\mathbf{V},\mathbf{W}}/\mathcal{N}_{\mathbf{V}}$  becomes a $_{\mathcal{A}}\mathbf{U}$-module, and there exists a canonical isomorphism of $_{\mathcal{A}}\mathbf{U}$-modules
	\begin{equation*}
		\varsigma^{\Lambda}:\mathcal{K}_{0}(\Lambda) \rightarrow {_{\mathcal{A}}L(\Lambda)}.
	\end{equation*}
	The morphism $\varsigma^{\Lambda}$ sends the image of constant sheaf $[\overline{\mathbb{Q}}_{l}]$ on
	$\mathbf{E}_{0,\mathbf{W},\hat{\Omega}}$
	to the highest weight vector $v_{\Lambda}$ in  ${_{\mathcal{A}}L(\Lambda)}$. The Euler form induces a contravariant bilinear form on $\mathcal{K}_{0}(\Lambda)$. Moreover, the set
	$\{\varsigma^{\Lambda}([L])|L$ is a simple perverse sheaf in $\mathcal{L}_{\mathbf{V}}(\Lambda)\}$ form a bar-invariant and almost orthogonal $\mathcal{A}$-basis of ${_{\mathcal{A}}L_{\mathbf{V}}(\Lambda)}$, which is exactly the canonical basis of $_{\mathcal{A}}L(\Lambda)$.
\end{theorem}

In particular, since the functor $F^{(n)}_{i}$ is defined by Lusztig's induction functor, the $\mathbf{U}^{-}$-module structure of $\mathcal{K}_{0}(\Lambda)$ is compatible with Lusztig's algebraic construction of $L(\Lambda)$. In particular, the vector bundle $\pi_{\mathbf{W}}$ allows us to compare our construction with Lusztig's construction of canonical basis directly, without relying on the uniqueness of the canonical basis.

Another progress is that we compare our functor $E_{i}$ with Lusztig's restriction functors. More precisely,  we  define category $\hat{\mathcal{Q}}_{\mathbf{V},\mathbf{W}}$ and functors $\hat{\mathcal{R}}^{\Lambda}_{i}$ and $ {_{i}\hat{\mathcal{R}}^{\Lambda}}$, which categorify linear operators $\frac{1  }{v-v^{-1}}\bar{r}_{i}$ and $\frac{ v^{(i,|\mathbf{V}'\oplus \mathbf{W}|) }}{v-v^{-1}} {_{i}\bar{r}}: {_{\mathcal{A}}\mathbf{U}^{-}  } \rightarrow \mathbf{U}^{-} \rightarrow L(\Lambda)$ respectively, then we obtain a split exact sequence as follows.
\begin{theorem}
	For any object $L$ of $\mathcal{Q}_{\mathbf{V},\mathbf{W}}$, there is a split exact sequence in $\hat{\mathcal{Q}}_{\mathbf{V}',\mathbf{W}}/ \mathcal{N}_{\mathbf{V}'}$ 
	\begin{equation*}
		0 \rightarrow {_{i}\hat{\mathcal{R}}^{\Lambda}}(L)  \rightarrow {\hat{\mathcal{R}}^{\Lambda}_{i}}(L)  \rightarrow E_{i}(L) \rightarrow 0.
	\end{equation*}
\end{theorem}
In fact, the theorem above is analogous to S.-J.Kang and M.Kashiwara's exact sequence in \cite[Theorem 4.7]{MR2995184}   and categorifies the equation $$E_{i}(x \cdot  v_{\Lambda} )=( v^{(i,|x|-i)- \langle \Lambda,\alpha^{\vee}_{i} \rangle}{_{i} \bar{r}}(x)\cdot  v_{\Lambda} - v^{\langle \Lambda,\alpha^{\vee}_{i} \rangle} \bar{r}_{i}(x)\cdot  v_{\Lambda} ) /(v^{-1}-v ).  $$

Recall that H.Nakajima provided a construction of irreducible highest weight $\mathfrak{g}$-modules (denoted by $L_{0}(\Lambda)$) via Borel-Moore cohomology groups of quiver varieties $\mathfrak{m}(\nu,\omega)$ and $\mathfrak{L}(\nu,\omega)$ in \cite{MR1302318} and \cite{MR1604167}. He defined operators $E_{i},F_{i},i \in I$ by using Hecke correspondences and proved that with operators $E_{i},F_{i}$, $\bigoplus \limits_{\nu}{\rm{H}}_{top}( \mathfrak{L}(\nu,\omega))$ becomes a $\mathbf{U}(\mathfrak{g})$-module, which is isomorphic to the irreducible integrable highest weight $\mathfrak{g}$-module $L(\Lambda)$. ( We denote this isomorphism by $\varkappa^{\Lambda}$.) Moreover, the fundamental classes of irreducible components of  $\mathfrak{L}(\nu,\omega)$ form a basis of $L_{0}(\Lambda)$. The last main result is the comparison between our sheaf realization (taking $v \rightarrow 1$) and Nakajima's construction via  quiver varieties. We provide a correspondence $\Phi^{\Lambda}$ between  the set of simple objects in localizations and the set of  irreducible components of  $\mathfrak{L}(\nu,\omega)$. Indeed, Lusztig's key lemma induces a left graph of the canonical basis of $L(\Lambda)$, which is isomorphic to the left graph of Nakajima's quiver varieties. As a corollary, these two different construction share the same monomial basis. After defining  orders $\preceq$ and $\preceq'$ on these bases, we can state our last main theorem:
\begin{theorem}
	The transition matrix from the canonical basis to the fundamental classes is upper triangular and with diagonal entries all equal to $\pm 1$.
	More precisely, if $X$ is an irreducible component of $\mathfrak{L}(\nu,\omega)$  and $[L]=\Phi^{\Lambda}(X)$, then we have 
	\begin{equation*}
		\varsigma^{\Lambda}([L])=\varkappa^{\Lambda}(sgn(X)[X]) +\sum\limits_{X \preceq X' } c_{X'}\varkappa^{\Lambda} (sgn(X')[X'])
	\end{equation*}
	where $c_{X'} \in \mathbb{Q}$ are constants.
\end{theorem}

In section 2, we recall the construction of Lusztig's sheaves and  the  categorification theory of the positive part of quantum group. In section 3, we define the localization $\mathcal{Q}_{\mathbf{V},\mathbf{W}}/\mathcal{N}_{\mathbf{V}}$ and functors $E^{(n)}_{i},F^{(n)}_{i},K^{\pm}_{i}$ to categorify $L(\Lambda)$. We also compare the functor $E^{(n)}_{i}$ for $n=1$ with Lusztig's restriction functors, which categorify derivative operators, and prove the split exact sequence in Theorem 1.2. In section 4, we compare our construction at $v \rightarrow 1$ with Nakajima's construction and prove the last main theorem. 

Obviously, as what has been done in  \cite{MR1865400},  \cite{MR3077693} and \cite{MR3200442},  a realization of tensor products of integrable highest weight modules via Lusztig sheaves is expected. We have done this work in the preprint \cite{fang2023tensor}. 

\section{Lusztig sheaves and quantum groups}
In this section, we recall Lusztig's theory of semisimple perverse sheaves and refer \cite{MR1227098} for details.
\subsection{Induction functor and restriction functor}
Assume $(I,(-,-))$ is a given Cartan datum, let $\mathbf{\Gamma}$ be the finite graph without loops associated to $(I,(-,-))$, where $I$ is the set of vertices and $H$ is the set of pairs consisting of edges with an orientation. More precisely, to give an edge $h$ with an orientation is equivalent to give $h',h'' \in I$ and we adapt the notation $h' \xrightarrow{h} h''$. Let $-:h \mapsto \bar{h}$ be the involution of $H$ such that $\bar{h}'=h'',\bar{h}''=h'$ and $\bar{h} \neq h$. An orientation of the graph $\Gamma$ is a subset $\Omega \subset H$ such that $\Omega \cap \bar{\Omega} =\emptyset$ and $\Omega \cup \bar{\Omega} = H$. We denote the oriented graph $(I,H,\Omega)$ by $Q$.

Let $\mathbf{k}=\overline{\mathbb{F}}_q$ be the algebraic closure of the finite field $\mathbb{F}_q$. Given $\nu \in \mathbb{N}[I]$, a subset $\tilde{H} \subseteq H$ and an $I$-graded $\mathbf{k}$-vector space $\mathbf{V}$ of dimension vector $|\mathbf{V}|=\nu$,  we take
\begin{center}
	$\mathbf{E}_{\mathbf{V}}= \bigoplus\limits_{h \in H} \mathbf{Hom}(\mathbf{V}_{h'},\mathbf{V}_{h''})$, \\
	$\mathbf{E}_{\mathbf{V}, \tilde{H}}= \bigoplus\limits_{h \in \tilde{H}} \mathbf{Hom}(\mathbf{V}_{h'},\mathbf{V}_{h''}).$
\end{center} 
In particular, for an orientation $\Omega$, we have 
\begin{center}
	$\mathbf{E}_{\mathbf{V}, \Omega}= \bigoplus\limits_{h \in \Omega} \mathbf{Hom}(\mathbf{V}_{h'},\mathbf{V}_{h''}).$
\end{center}

The algebraic group $G_{\mathbf{V}}= \prod\limits_{i \in I} \mathbf{GL}(\mathbf{V}_{i})$ acts on $\mathbf{E}_{\mathbf{V}},\mathbf{E}_{\mathbf{V},\tilde{H}}$ and $\mathbf{E}_{\mathbf{V}, \Omega}$ by $$(g \cdot x)_{h} =g_{h''} x_{h} g_{h'}^{-1}. $$ 

Let $\mathcal{D}^{b}_{G_{\mathbf{V}}}(\mathbf{E}_{\mathbf{V},\Omega})$ be the  $G_{\mathbf{V}}$-equivariant derived category of constructible sheaves on $\mathbf{E}_{\mathbf{V},\Omega}$. For any $n \in \mathbb{Z}$, we denote the $n$-times shift functors by $[n]$. We say complexes $A$ and $B$ are isomorphic up to shifts, if  $A$ and $B[n]$ are isomorphic for some $n \in \mathbb{Z}$. 

Given $\nu'+\nu''=\nu \in \mathbb{N}[I] $ and graded vector spaces $\mathbf{V},\mathbf{V}',\mathbf{V}''$ of dimension vectors $\nu, \nu', \nu''$ respectively, let $\mathbf{E}'_{\Omega}$ be the variety consisting of $(x,\tilde{\mathbf{W}}, \rho_{1}, \rho_{2})$, where $x \in \mathbf{E}_{\mathbf{V},\Omega}$,$\tilde{\mathbf{W}}$ is an $I$-graded $x$-stable  subspace of $\mathbf{V}$ of dimension vector $\nu''$, $ \rho_{1}: \mathbf{V}/\tilde{\mathbf{W}} \simeq \mathbf{V}', \rho_{2}:\tilde{\mathbf{W}} \simeq \mathbf{V}''$ are linear isomorphisms. Here we say $\tilde{\mathbf{W}}$ is $x$-stable if and only if $x_{h}(\tilde{\mathbf{W}}_{h'}) \subset \tilde{\mathbf{W}}_{h''}$ for any $h \in \Omega$.  Let $\mathbf{E}''_{\Omega}$ be the variety consisting of $(x,\tilde{\mathbf{W}})$ as above.
Consider the following diagram
\begin{center}
	$\mathbf{E}_{\mathbf{V}',\Omega} \times \mathbf{E}_{\mathbf{V}'',\Omega} \xleftarrow{p_{1}} \mathbf{E}'_{\Omega} \xrightarrow{p_{2}} \mathbf{E}''_{\Omega} \xrightarrow{p_{3}} \mathbf{E}_{\mathbf{V},\Omega}$
\end{center}
where
\begin{equation}\label{Lind}
	\begin{split}
		&p_{1}(x,\tilde{\mathbf{W}},\rho_{1},\rho_{2})=(\rho_{1,\ast}(\bar{x}|_{\mathbf{V}/\tilde{\mathbf{W}}}),\rho_{2,\ast}(x|_{\tilde{\mathbf{W}}})),\\
		&p_{2}(x,\tilde{\mathbf{W}},\rho_{1},\rho_{2}) =(x, \tilde{\mathbf{W}}),\\
		&p_{3}(x,\tilde{\mathbf{W}})=x.
	\end{split}
\end{equation}
Here $\bar{x}|_{\mathbf{V}/\tilde{\mathbf{W}}}$ is the natural linear map induced by $x$ on the quotient space $\mathbf{V}/\tilde{\mathbf{W}}$ and $x|_{\tilde{\mathbf{W}}}$ is the restriction of $x$ on the subspace $\tilde{\mathbf{W}}$, then $ \rho_{1,\ast}(\bar{x}|_{\mathbf{V}/\tilde{\mathbf{W}}})= \rho_{1} (\bar{x}|_{\mathbf{V}/\tilde{\mathbf{W}}}) \rho_{1}^{-1}\in \mathbf{E}_{\mathbf{V}'}$ and $\rho_{2,\ast}(x|_{\tilde{\mathbf{W}}})=\rho_{2}(x|_{\tilde{\mathbf{W}}}) \rho_{2}^{-1}\in \mathbf{E}_{\mathbf{V}''}$. Notice that  $p_{1}$ is smooth with connected fibers, $p_{2}$ is a principle $G_{\mathbf{V}'} \times G_{\mathbf{V}''}$-bundle and $p_{3}$ is proper.  Let $d_{1}$ be the dimension of the fibers of $p_{1}$ and $d_{2}$ be the dimension of the fibers of $p_{2}$. 

Lusztig's induction functor is defined by
\begin{align*}
	&\mathbf{Ind}^{\mathbf{V}}_{\mathbf{V'},\mathbf{V''}}: \mathcal{D}^{b}_{G_{\mathbf{V}'}}(\mathbf{E}_{\mathbf{V}',\Omega}) \times \mathcal{D}^{b}_{G_{\mathbf{V}''}}(\mathbf{E}_{\mathbf{V}'',\Omega}) \rightarrow \mathcal{D}^{b}_{G_{\mathbf{V}}}(\mathbf{E}_{\mathbf{V},\Omega})\\
	&\mathbf{Ind}^{\mathbf{V}}_{\mathbf{V'},\mathbf{V''}}(A\boxtimes B)= (p_{3})_{!}(p_{2})_{\flat}(p_{1})^{\ast}(A\boxtimes B)[d_{1}-d_{2}].
\end{align*}
Here for principle $G$-bundle  $f:Y \rightarrow X$, the functor $f_{\flat}$ is defined to be the quasi-inverse of $f^{\ast}$, which gives an equivalence $\mathcal{D}^{b}(X) \cong \mathcal{D}^{b}_{G}(Y)$. See  \cite[Theorem 6.5.9]{Achar}.

We fix a decomposition $\mathbf{T} \oplus \mathbf{W} =\mathbf{V}$ of graded vector space such that $|\mathbf{T}|=\nu'$ and $|\mathbf{W}|=\nu''$, let $F_{\Omega}$ be the closed subvariety of $\mathbf{E}_{\mathbf{V},\Omega}$ consisting of $x$ such that $\mathbf{W}$ is $x$-stable. 
Consider the following diagram
\begin{center}
	$\mathbf{E}_{\mathbf{T},\Omega} \times \mathbf{E}_{\mathbf{W},\Omega} \xleftarrow{\kappa_{\Omega} } F_{\Omega} \xrightarrow{\iota_{\Omega}} \mathbf{E}_{\mathbf{V},\Omega}$
\end{center} 
where $\iota_{\Omega}$ is the natural embedding and $\kappa_{\Omega}(x)=(\overline{x}|_{\mathbf{T}},x|_{\mathbf{W}}) \in \mathbf{E}_{\mathbf{T},\Omega} \times \mathbf{E}_{\mathbf{W},\Omega} $ for any $x \in F_{\Omega}$. Notice that $\kappa_{\Omega}$ is a vector bundle. 

Lusztig's restriction functor is defined by
\begin{align*}
	&\mathbf{Res}^{\mathbf{V}}_{\mathbf{T},\mathbf{W}}: \mathcal{D}^{b}_{G_{\mathbf{V}}}(\mathbf{E}_{\mathbf{V},\Omega}) \rightarrow \mathcal{D}^{b}_{G_{\mathbf{T}} \times G_{\mathbf{W}}}(\mathbf{E}_{\mathbf{T},\Omega}\times \mathbf{E}_{\mathbf{W},\Omega})\\
	&\mathbf{Res}^{\mathbf{V}}_{\mathbf{T},\mathbf{W}}(C)=(\kappa_{\Omega})_{!} (\iota_{\Omega})^{\ast}(C)[-\langle\nu',\nu''\rangle],
\end{align*}
where $\langle \nu',\nu''\rangle=\sum_{i\in I}\nu'_i\nu''_i-\sum_{h\in \Omega}\nu'_{h'}\nu''_{h''}$ is the  Euler form associated to $Q$.

We denote by $\mathcal{S}_{|\mathbf{V}|}$ the set of sequences $\boldsymbol{\nu}=(\nu^{1},\nu^{2},\cdots, \nu^{m})$ such that each $\nu^{l}$ is of the form $(i_{l})^{a_{l}} \in \mathbb{N}[I]$ for some $i_l\in I,a_l\in \mathbb{N}$ and $\sum_{l=1}^m\nu^l=|\mathbf{V}|$. (Here we denote $a_{l}i_{l} \in \mathbb{N}[I]$ by $(i_{l})^{a_{l}}$ in order to be compatible with the multiplication of $\mathbf{U}$.) We call an element $\boldsymbol{\nu} \in \mathcal{S}_{|\mathbf{V}|}$ a flag type of $\mathbf{V}$.  For any flag type $\boldsymbol{\nu}=(\nu^{1},\nu^{2},\cdots, \nu^{m}) \in \mathcal{S}_{|\mathbf{V}|}$, the flag variety $\mathcal{F}_{\boldsymbol{\nu},\Omega}$ is the variety consisting of $(x,f)$, where $x \in \mathbf{E}_{\mathbf{V},\Omega}$ and $f=(0=\mathbf{V}^{m} \subseteq \mathbf{V}^{m-1} \subseteq \cdots \subseteq \mathbf{V}^{0}=\mathbf{V})$  is a flag of $\mathbf{V}$ such that  $x(\mathbf{V}^{k}) \subseteq \mathbf{V}^{k}$ and $|\mathbf{V}^{k-1}/\mathbf{V}^{k}|=\nu^{k}$ for every $1\leqslant k \leqslant m$. 

The flag variety $\mathcal{F}_{\boldsymbol{\nu},\Omega}$ is smooth and the natural projection map $\pi_{\boldsymbol{\nu},\Omega}:\mathcal{F}_{\boldsymbol{\nu},\Omega} \rightarrow \mathbf{E}_{\mathbf{V},\Omega}$ is proper. Then by the decomposition theorem in \cite{BBD}, the complex $L_{\boldsymbol{\nu}}= (\pi_{\boldsymbol{\nu},\Omega})_{!} \bar{\mathbb{Q}}_{l}[\dim \mathcal{F}_{\boldsymbol{\nu},\Omega}]$ is a semisimple complex on $\mathbf{E}_{\mathbf{V},\Omega}$, where $\bar{\mathbb{Q}}_{l}$ is the constant sheaf on $\mathcal{F}_{\boldsymbol{\nu},\Omega}$. 

Let $\mathcal{P}_{\mathbf{V},\Omega}$ be the full subcategory of $\mathcal{D}^{b}_{G_{\mathbf{V}}}(\mathbf{E}_{\mathbf{V},\Omega})$ consisting of simple perverse sheaves appearing as direct summands of some $L_{\boldsymbol{\nu}}, \boldsymbol{\nu}\in \mathcal{S}_{|\mathbf{V}|}$ up to $[n]$ shifts, and let $\mathcal{Q}_{\mathbf{V},\Omega}$ be the full subcategory of $\mathcal{D}^{b}_{G_{\mathbf{V}}}(\mathbf{E}_{\mathbf{V},\Omega})$ consisting of direct sums of  shifts of objects in $\mathcal{P}_{\mathbf{V},\Omega}$. Objects in $\mathcal{Q}_{\mathbf{V},\Omega}$ are called Lusztig sheaves in \cite{MR3202708}.

\begin{proposition}\cite[Lemma 3.2, Proposition 4.2]{MR1088333} \label{indres formula}
	For any $\nu'+\nu''=\nu$ and flag types $\boldsymbol{\nu}'\in \mathcal{S}_{|\mathbf{V}'|},\boldsymbol{\nu}''\in \mathcal{S}_{|\mathbf{V}''|}$,
	\begin{equation*}
		\mathbf{Ind}^{\mathbf{V}}_{\mathbf{V}',\mathbf{V}''}(L_{\boldsymbol{\nu}'} \boxtimes L_{\boldsymbol{\nu}''})= L_{\boldsymbol{\nu}' \boldsymbol{\nu}''},
	\end{equation*}
	where $\boldsymbol{\nu}' \boldsymbol{\nu}''=((i'_{1})^{a'_{1}} , \cdots ,(i'_{m})^{a'_{m}},(i''_{1})^{a''_{1}},\cdots,(i''_{n})^{a''_{n}} )$ is the flag type obtained by connecting $\boldsymbol{\nu}'=((i'_{1})^{a'_{1}} , \cdots ,(i'_{m})^{a'_{m}})$ and $\boldsymbol{\nu}''=((i''_{1})^{a''_{1}},\cdots,(i''_{n})^{a''_{n}} ).$
	
	For $ \boldsymbol{\nu}=(i_{1}^{a_{1}},i_{2}^{a_{2}},\cdots,i_{m}^{a_{m}} )\in \mathcal{S}_{|\mathbf{V}|}$,
	\begin{equation*}
		\mathbf{Res}^{\mathbf{V}}_{\mathbf{T},\mathbf{W}}( L_{\boldsymbol{\nu}}) =\bigoplus \limits_{\boldsymbol{\tau}+\boldsymbol{\omega}=\boldsymbol{\nu} } L_{\boldsymbol{\tau}} \boxtimes L_{\boldsymbol{\omega}}[M(\boldsymbol{\tau},\boldsymbol{\omega})],
	\end{equation*}
	here the notation $\boldsymbol{\tau}+\boldsymbol{\omega}=\boldsymbol{\nu}$ means that the flag types $$\boldsymbol{\tau}=(i_{1}^{b_{1}},i_{2}^{b_{2}},\cdots,i_{m}^{b_{m}} ) \in \mathcal{S}_{|\mathbf{T}|},~
	\boldsymbol{\omega}=(i_{1}^{c_{1}},i_{2}^{c_{2}},\cdots,i_{m}^{c_{m}} )\in \mathcal{S}_{|\mathbf{W}|}$$ satisfy $b_{k}+c_{k}=a_{k}$ for every $1 \leqslant k \leqslant m$, and the integer $	M(\boldsymbol{\tau},\boldsymbol{\omega})$ is given by the following formula
	\begin{equation*}
		\begin{split}
			M(\boldsymbol{\tau},\boldsymbol{\omega})=&-\sum\limits_{h \in H, l' < l}(\tau^{l'}_{h'}\omega^{l}_{h''}+\tau^{l'}_{h''}\omega^{l}_{h'})
			+\sum\limits_{h \in H}({\rm{dim} \mathbf{T}_{h'}}{\rm{dim} \mathbf{W}_{h''}}+{\rm{dim} \mathbf{T}_{h''}}{\rm{dim} \mathbf{W}_{h'}}) \\
			&-\sum\limits_{i \in I,l < l'} \tau_{i}^{l'}\omega_{i}^{l}+\sum\limits_{i \in I,l > l'} \tau_{i}^{l'}\omega_{i}^{l}-\sum\limits_{i \in I}{\rm{dim} \mathbf{T}_{i}}{\rm{dim} \mathbf{W}_{i}} .   
		\end{split}
	\end{equation*}
\end{proposition}

As a corollary, the induction and restriction functors can be restricted to
\begin{align*}
	&\mathbf{Ind}^{\mathbf{V}}_{\mathbf{V}',\mathbf{V}''}:\mathcal{Q}_{\mathbf{V}'}\times \mathcal{Q}_{\mathbf{V}''}\rightarrow \mathcal{Q}_{\mathbf{V}},\\
	&\mathbf{Res}^{\mathbf{V}}_{\mathbf{T},\mathbf{W}}:\mathcal{Q}_{\mathbf{V}}\rightarrow \mathcal{Q}_{\mathbf{T}}\boxtimes \mathcal{Q}_{\mathbf{W}}.
\end{align*}

\begin{remark}
	Let $\mathcal{D}^{b,ss}_{G_{\mathbf{V}}}(\mathbf{E}_{\mathbf{V},\Omega})$ be the full subcatgory of $\mathcal{D}^{b}_{G_{\mathbf{V}}}(\mathbf{E}_{\mathbf{V},\Omega})$ consisting of semisimple complexes, then the induction functor and restrcition functor also restrict to 
	\begin{align*}
		&\mathbf{Ind}^{\mathbf{V}}_{\mathbf{V}',\mathbf{V}''}:\mathcal{D}^{b,ss}_{G_{\mathbf{V}'}}(\mathbf{E}_{\mathbf{V}',\Omega})\times \mathcal{D}^{b,ss}_{G_{\mathbf{V}''}}(\mathbf{E}_{\mathbf{V}'',\Omega})\rightarrow \mathcal{D}^{b,ss}_{G_{\mathbf{V}}}(\mathbf{E}_{\mathbf{V},\Omega}),\\
		&\mathbf{Res}^{\mathbf{V}}_{\mathbf{T},\mathbf{W}}:\mathcal{D}^{b,ss}_{G_{\mathbf{V}}}(\mathbf{E}_{\mathbf{V},\Omega})\rightarrow \mathcal{D}^{b,ss}_{G_{\mathbf{T}}\times G_{\mathbf{W}}}(\mathbf{E}_{\mathbf{T},\Omega} \times \mathbf{E}_{\mathbf{W},\Omega}).
	\end{align*}
\end{remark}

Let $\mathcal{K}_{\mathbf{V},\Omega}$ be the Grothendieck group of
$\mathcal{Q}_{\mathbf{V},\Omega}$ and $\mathcal{K}_{\Omega}=\bigoplus\limits_{\mathbf{V}} \mathcal{K}_{\mathbf{V},\Omega}$ which have an $\mathcal{A}$-module structure given by
\begin{center}
	$v[L]=[L[1]].$ 
\end{center} 

\begin{theorem}\cite[Theorem 10.17]{MR1088333}\label{Lusztig1}
	With the induction and restriction functors, the Grothendieck group $\mathcal{K}$ becomes a bialgebra, and is canonically isomorphic to the (integral form of) positive part of the quantized enveloping algebra ${_{\mathcal{A}}}\mathbf{U}^{+}$:
	\begin{equation*}
		\varsigma:[L_{i^{(p)}}] \mapsto E_{i}^{(p)}
	\end{equation*} 
	where $L_{i^{(p)}}$ is the constant sheaf on $\mathbf{E}_{\mathbf{V},\Omega}$ with $|\mathbf{V}|=pi$.
	Moreover, the images of simple perverse sheaves in $\mathcal{P}_{\mathbf{V}}$ form a $\mathcal{A}=\mathbb{Z}[v,v^{-1}]$-basis of ${_{\mathcal{A}}}\mathbf{U}_{|\mathbf{V}|}^{+}$, which is called the canonical basis. The canonical basis is bar-invariant and has positivity.
\end{theorem}

\subsection{Fourier-Deligne transform}
We fix a nontrivial character $\mathbb{F}_{q} \rightarrow \bar{\mathbb{Q}}_{l}^{\ast}$. This character defines an Artin-Schreier local system of rank $1$ on $\mathbf{k}$. 

For two orientations $\Omega,\Omega'$, we define $T: \mathbf{E}_{\mathbf{V},\Omega \cup \Omega'} \rightarrow k$ by $T(x)=\sum \limits_{h \in \Omega \backslash \Omega'}tr(x_{h}x_{\bar{h}})$. Then the inverse image of the Artin-Schreier local system under $T$ is a well-defined $G_{\mathbf{V}}$-equivariant local system of rank $1$ on $\mathbf{E}_{\mathbf{V},\Omega \cup \Omega'}$, denote by $\mathcal{L}_{T}$.

Consider the following diagram
\begin{center}
	$\mathbf{E}_{\mathbf{V},\Omega} \xleftarrow{\delta} \mathbf{E}_{\mathbf{V},\Omega \cup \Omega'} \xrightarrow{\delta'} \mathbf{E}_{\mathbf{V},\Omega'}$
\end{center}
where $\delta,\delta'$ are the forgetting maps defined by
\begin{center}
	$\delta((x_{h})_{h \in \Omega \cup \Omega'} )= ((x_{h})_{h \in \Omega}),$\\
	$\delta'((x_{h})_{h \in \Omega \cup \Omega'} )= ((x_{h})_{h \in \Omega'}).$
\end{center}

Lusztig defined the Fourier-Deligne transform for quivers to be the functor
\begin{align*}
	&\mathcal{F}_{\Omega,\Omega'}:\mathcal{D}^{b}_{G_{\mathbf{V}}}(E_{\mathbf{V},\Omega}) \rightarrow \mathcal{D}^{b}_{G_{\mathbf{V}}}(E_{\mathbf{V},\Omega'})\\ &\mathcal{F}_{\Omega,\Omega'}(L)=\delta'_{!}(\delta^{\ast}(L)\otimes \mathcal{L}_{T})[D],
\end{align*}  
here $D=\sum\limits_{h \in \Omega \backslash \Omega'}\dim \mathbf{V}_{h'}\dim\mathbf{V}_{h''}$.

\begin{proposition} \cite[Theorem 5.4]{MR1088333} \label{FD0}
	With the notations above, for any semisimple perverse sheaves $L_{1}$ and $L_{2}$,
	\begin{center}
		$\mathcal{F}_{\Omega,\Omega'}(\mathbf{Ind}^{\mathbf{V}} _{\mathbf{V}',\mathbf{V}''}(L_{1} \boxtimes L_{2})) \cong  \mathbf{Ind}^{\mathbf{V}} _{\mathbf{V}',\mathbf{V}''}(\mathcal{F}_{\Omega,\Omega'}(L_{1}) \boxtimes \mathcal{F}_{\Omega,\Omega'}(L_{2})).$
	\end{center}
\end{proposition}

\begin{corollary}\cite[Proposition 10.14]{MR1088333} \label{FD1}
	The functor $\mathcal{F}_{\Omega,\Omega'}$ induces an algebra isomorphism between $\mathcal{K}_{\Omega}$ and $\mathcal{K}_{\Omega'}$. 
\end{corollary}

\begin{corollary}\cite[Corollary 5.6]{MR1088333} \label{FD2}
	The functor $\mathcal{F}_{\Omega,\Omega'}$ induces an equivalence of categories $\mathcal{Q}_{\mathbf{V},\Omega} \cong \mathcal{Q}_{\mathbf{V},\Omega'}$ and a bijection $\eta_{\Omega,\Omega'}:\mathcal{P}_{\Omega}\rightarrow \mathcal{P}_{\Omega'}$. Moreover, for orientations $\Omega,\Omega',\Omega''$, we have $\eta_{\Omega',\Omega''}\eta_{\Omega,\Omega'}=\eta_{\Omega,\Omega''}$. 
\end{corollary}

With two corollaries above, we can denote $\mathcal{K}_{\Omega}$ by $\mathcal{K}$ and $\mathcal{P}_{\mathbf{V},\Omega}$ by $\mathcal{P}_{\mathbf{V}}$ respectively, if there is no ambiguity.

\begin{remark}\label{remarkFD}
	(1) In fact, the Fourier-Deligne transform gives equivalence between the bounded derived categories $\mathcal{D}^{b}_{G_{\mathbf{V}}}(\mathbf{E}_{\mathbf{V},\Omega})$ and $\mathcal{D}^{b}_{G_{\mathbf{V}}}(\mathbf{E}_{\mathbf{V},\Omega'})$. (See the proof of \cite[Theorem 3.13]{MR3202708}.) 
	
	(2) Similarly as \cite[Proposition 10.4.5]{Achar}, the Fourier-Deligne transform also commutes with the induction functor   $\mathbf{Ind}^{\mathbf{V}}_{\mathbf{V'},\mathbf{V''}}: \mathcal{D}^{b}_{G_{\mathbf{V}'}}(\mathbf{E}_{\mathbf{V}',\Omega}) \times \mathcal{D}^{b}_{G_{\mathbf{V}''}}(\mathbf{E}_{\mathbf{V}'',\Omega}) \rightarrow \mathcal{D}^{b}_{G_{\mathbf{V}}}(\mathbf{E}_{\mathbf{V},\Omega})$ between bounded derived categories as in Proposition \ref{FD0}. In fact, the authour in \cite{Achar} considered varieties over the complex field and Fourier-Laumon transforms, but his proof also works in our case by \cite[\uppercase\expandafter{\romannumeral3}. Theorem 13.2 and 13.3]{MR1855066}.  
\end{remark}

\subsection{Analysis at sink}
Fix $i \in I$ and an orientation $\Omega$ such that $i$ is a sink, for any $p\in \mathbb{N}$, we define $\mathbf{E}_{\mathbf{V},i,p}$ to be the locally closed subset of $\mathbf{E}_{\mathbf{V},\Omega}$ consisting of $x$ such that ${\rm{codim}}_{\mathbf{V}_{i}} ( {\rm{Im}} \sum\limits_{h \in \Omega, h''=i} x_{h}) =p$. Then $\mathbf{E}_{\mathbf{V}}$ has a partition $\mathbf{E}_{\mathbf{V},\Omega}= \bigcup \limits_{p} \mathbf{E}_{\mathbf{V},i,p}$, and the union $\mathbf{E}_{\mathbf{V},i, \geqslant p}= \bigcup\limits_{p' \geqslant p} \mathbf{E}_{\mathbf{V},i, p'}$ is a closed subset of $\mathbf{E}_{\mathbf{V},\Omega}$.

Given a simple perverse sheaf $L$, there exists a unique integer $t$ such that $$\textrm{supp}(L) \subseteq \mathbf{E}_{\mathbf{V},i, \geqslant t},~ \textrm{supp}(L) \nsubseteq \mathbf{E}_{\mathbf{V},i, \geqslant t+1}$$ and we set $t_{i}(L)=t$. Notice that $t_{i}(L) \leqslant \nu_{i}$.

The following lemma is the key lemma  \cite[Lemma 6.4]{MR1088333} in Lusztig's categorification theory. In fact, Lusztig's proof works not only for simple objects in $\mathcal{P}_{\mathbf{V}}$, but also for all simple perverse sheaves.
\begin{lemma} \label{lkey'}
	With the notation above, fix $0 \leqslant t \leqslant \nu_{i}$ and assume $|\mathbf{T}|=|\mathbf{V}'|=ti$. 
	
	(1) Let $L\in\mathcal{D}^{b,ss}_{G_{\mathbf{V}}}(\mathbf{E}_{\mathbf{V},\Omega})$ be a simple perverse sheaf such that $t_{i}(L)=t$, then the complex $\mathbf{Res}^{\mathbf{V}}_{\mathbf{T},\mathbf{W}}(L)$ is a direct sum of finitely many summands of the form $K'[f']$ for various simple perverse sheaves $K'\in\mathcal{D}^{b,ss}_{G_{\mathbf{W}}}(\mathbf{E}_{\mathbf{W},\Omega})$ and $f' \in \mathbb{Z}$. Moreover, exactly one of these summands, denoted by $K[f]$, satisfies $t_{i}(K)=0$ and $f=0$ and the others satisfy $t_{i}(K')> 0$. If $L \in \mathcal{P}_{\mathbf{V}}$, then those $K'$ are also Lusztig's sheaves.
	
	(2) Let $K \in \mathcal{D}^{b,ss}_{G_{\mathbf{V}''}}(\mathbf{E}_{\mathbf{V}'',\Omega})$ be a simple perverse sheaf such that $t_{i}(K)=0$, then $\mathbf{Ind}^{\mathbf{V}}_{\mathbf{V}',\mathbf{V}''}(\bar{\mathbb{Q}}_{l} \boxtimes K)$ is a direct sum of finitely many summands of the form $L'[g']$ for various simple perverse sheaves $L' \in \mathcal{D}^{b,ss}_{G_{\mathbf{V}}}(\mathbf{E}_{\mathbf{V},\Omega})$ and $g' \in \mathbb{Z}$. Moreover, exactly one of these summands, denoted by $L[g]$, satisfies $t_{i}(L)=t$ and $g=0$ and the others satisfy $t_{i}(L')> t$. If $K \in \mathcal{P}_{\mathbf{V}''}$, then those $L'$ are also Lusztig's sheaves.
	
	(3) There is a bijection $\pi_{i,t}$ between the set $\{K \textrm{ is a simple object in }  \mathcal{D}^{b,ss}_{G_{\mathbf{V}''}}(\mathbf{E}_{\mathbf{V}'',\Omega})$ and $t_{i}(K)=0 \}$  and the set $\{L \textrm{ is a simple object in }  \mathcal{D}^{b,ss}_{G_{\mathbf{V}}}(\mathbf{E}_{\mathbf{V},\Omega})$ and $t_{i}(L)=t \}$, which is induced by the decompositions of the direct sums above. Moreover, the bijection above also restricts to a bijection between the set $\{K \in \mathcal{P}_{\mathbf{V}''}|t_{i}(K)=0 \}$  and the set $\{L \in \mathcal{P}_{\mathbf{V}}|t_{i}(L)=t \}$.
\end{lemma}

If $|\mathbf{V}'|=ri,|\mathbf{V}''|=|\mathbf{V}|-ri$, we denote $\mathbf{V'}$ by $\mathbf{V}'_{ri}$ and $\mathbf{V}''$ by $\mathbf{V}''_{\nu-ri}$. For an orientation $\Omega'$ and a simple perverse sheaf $L$, we define $s_{i}(L)$ to be the largest integer $r$ satisfying that there exists a semisimple complex $L'$ such that $L$ is isomorphic to a shift of a direct summand of $\mathbf{Ind}^{\mathbf{V}}_{\mathbf{V}'_{ri},\mathbf{V}''_{\nu-ri}}(\bar{\mathbb{Q}}_{l} \boxtimes L')$. Notice that the definition of $s_{i}(L)$ does not depend on the choice of $\Omega'$ by Proposition \ref{FD0}. Lusztig's proof of Proposition 6.6 in \cite{MR1088333} can also be generalized to  $\mathcal{D}^{b,ss}_{G_{\mathbf{V}}}(\mathbf{E}_{\mathbf{V},\Omega'})$ as the following.
\begin{proposition}\label{lt'}
	Let $L\in\mathcal{D}^{b,ss}_{G_{\mathbf{V}}}(\mathbf{E}_{\mathbf{V},\Omega'})$ be a simple perverse sheaf and $s_{i}(L)=r$.
	
	(1)  There exist semisimple complexes $L'_{r'} \in\mathcal{D}^{b,ss}_{G_{\mathbf{V}''_{\nu-r'i}}}(\mathbf{E}_{\mathbf{V}''_{\nu-r'i},\Omega'}) $ for $r' > s_{i}(L)$ and semisimple complexes $L''_{r'} \in\mathcal{D}^{b,ss}_{G_{\mathbf{V}''_{\nu-r'i}}}(\mathbf{E}_{\mathbf{V}''_{\nu-r'i},\Omega'})$ for $r' \geqslant s_{i}(L)$ such that 
	\begin{center}
		$L \oplus \bigoplus \limits_{r' >s_{i}(L)}\mathbf{Ind}^{\mathbf{V}}_{\mathbf{V}'_{r'i},\mathbf{V}''_{\nu-r'i}}(\bar{\mathbb{Q}}_{l} \boxtimes L'_{r'}) \cong \bigoplus \limits_{r' \geqslant s_{i}(L)}\mathbf{Ind}^{\mathbf{V}}_{\mathbf{V}'_{r'i},\mathbf{V}''_{\nu-r'i}}(\bar{\mathbb{Q}}_{l} \boxtimes L''_{r'}) .$
	\end{center}
	
	Moreover, if $L \in \mathcal{P}_{\mathbf{V},\Omega'}$, then those $L'_{r'}$ and $L''_{r'}$ can be chosen in $\mathcal{Q}_{\mathbf{V}''_{\nu-r'i},\Omega'}$.
	
	(2) $s_{i}(L)=t_{i}(L)$, if $i$ is a sink in $\Omega'$.
	
\end{proposition}

\subsection{Analysis at source}

Fix $i \in I$ and an orientation $\Omega$ such that $i$ is a source, we define $\mathbf{E}_{\mathbf{V},i}^{p}$ to be the subset of $\mathbf{E}_{\mathbf{V},\Omega}$ consisting of $x$ such that ${\rm{dim}}( {\rm{Ker}} \bigoplus\limits_{h \in \Omega, h'=i} x_{h}) =p$. Then $\mathbf{E}_{\mathbf{V}}$ has a partition $\mathbf{E}_{\mathbf{V},\Omega}= \bigcup \limits_{p} \mathbf{E}_{\mathbf{V},i}^{p}$. and the union $\mathbf{E}_{\mathbf{V},i}^{\geqslant p}= \bigcup\limits_{p' \geqslant p} \mathbf{E}_{\mathbf{V},i}^{p'}$ is a closed subset.

Given a simple perverse sheaf $L$, there exists a unique integer $t$ such that $\textrm{supp}(L) \subseteq \mathbf{E}_{\mathbf{V},i}^{ \geqslant t}$ but $\textrm{supp}(L) \nsubseteq \mathbf{E}_{\mathbf{V},i}^{\geqslant t+1}$ and we write $t_{i}^{\ast}(L)=t$. Notice that $t_{i}^{\ast}(L) \leqslant \nu_{i}$.

The following lemmas are dual to Lemma \ref{lkey'}.

\begin{lemma}\label{rkey'}
	With the notation above, fix $0 \leqslant t \leqslant \nu_{i}$ and assume $|\mathbf{W}|=|\mathbf{V}''|=ti$.
	
	(1) Let $L\in\mathcal{D}^{b,ss}_{G_{\mathbf{V}}}(\mathbf{E}_{\mathbf{V},\Omega})$ be a simple perverse sheaf such that $t_{i}^{\ast}(L)=t$, then the complex $\mathbf{Res}^{\mathbf{V}}_{\mathbf{T},\mathbf{W}}(L) \in \mathcal{D}^{b,ss}_{G_{\mathbf{T}}}(\mathbf{E}_{\mathbf{T},\Omega})$ is a direct sum of finitely many summands of the form $K'[f']$ for various simple perverse sheaves $K'$ and $f' \in \mathbb{Z}$. Moreover, exactly one of these summands, denoted by $K[f]$, satisfies $t_{i}^{\ast}(K)=0$ and $f=0$ and the others satisfy $t_{i}^{\ast}(K')> 0$. If $L \in \mathcal{P}_{\mathbf{V}}$, then those $K'$  are also Lusztig's sheaves.
	
	(2) Let $K \in \mathcal{D}^{b,ss}_{G_{\mathbf{V}'}}(\mathbf{E}_{\mathbf{V}',\Omega})$ be a simple perverse sheaf such that $t_{i}^{\ast}(K)=0$, then $\mathbf{Ind}^{\mathbf{V}}_{\mathbf{V}',\mathbf{V}''}(K \boxtimes \bar{\mathbb{Q}}_{l} )$ is a direct sum of finitely many summands of the form $L'[g']$ for various simple perverse sheaves $L' \in \mathcal{D}^{b,ss}_{G_{\mathbf{V}}}(\mathbf{E}_{\mathbf{V},\Omega})$ and $g' \in \mathbb{Z}$. Moreover, exactly one of these summands, denoted by $L[g]$, satisfies $t_{i}^{\ast}(L)=t$ and $g=0$ and the others satisfy $t_{i}^{\ast}(L')> t$.  If $K \in \mathcal{P}_{\mathbf{V}'}$, then those $L'$  are also Lusztig's sheaves.
	
	(3) There is a bijection $\pi^{\ast}_{i,t}$ between the set $\{K \textrm{ is a simple object in }  \mathcal{D}^{b,ss}_{G_{\mathbf{V}'}}(\mathbf{E}_{\mathbf{V}',\Omega})$ and $ t_{i}^{\ast}(K)=0 \}$ and the set $\{L \textrm{ is a simple object in } \mathcal{D}^{b,ss}_{G_{\mathbf{V}}}(\mathbf{E}_{\mathbf{V},\Omega})$ and $t_{i}^{\ast}(L)=t \} $, which is  induced by the decompositions of the direct sums above. Moreover, the bijection above also restricts to a bijection between the set $\{K \in \mathcal{P}_{\mathbf{V}'}|t^{\ast}_{i}(K)=0 \}$  and the set $\{L \in \mathcal{P}_{\mathbf{V}}|t^{\ast}_{i}(L)=t \}$.
\end{lemma}

In this section, if $|\mathbf{V}'|=|\mathbf{V}|-ri$ and $|\mathbf{V}''|=ri$, we denote $\mathbf{V}'$ by $\mathbf{V}'_{\nu-ri}$ and $\mathbf{V}''$ by $\mathbf{V}''_{ri}$. For an orientation $\Omega'$ and a simple perverse sheaf $L$, we define $s_{i}^{\ast}(L)$ to be the largest integer $r$ satisfying that there exists $L'$ such that $L$ is isomorphic to a direct summand of $\mathbf{Ind}^{\mathbf{V}}_{\mathbf{V}'_{\nu-ri},\mathbf{V}''_{ri}}(L' \boxtimes \bar{\mathbb{Q}}_{l})$. The following Proposition is dual to Proposition \ref{lt'}.

\begin{proposition}\label{rt'}
	Let $L \in \mathcal{D}^{b,ss}_{G_{\mathbf{V}}}(\mathbf{E}_{\mathbf{V},\Omega'})$ be a simple perverse sheaf with $s^{\ast}_{i}(L)=r$.
	
	(1)  There exist semisimple complexes $L'_{r'} \in \mathcal{D}^{b,ss}_{G_{\mathbf{V}'_{\nu-r'i}}}(\mathbf{E}_{\mathbf{V}'_{\nu-r'i},\Omega'})$ for $r' > s_{i}^{\ast}(L)$ and semisimple complexes $L''_{r'} \in  \mathcal{D}^{b,ss}_{G_{\mathbf{V}'_{\nu-r'i}}}(\mathbf{E}_{\mathbf{V}'_{\nu-r'i},\Omega'})$ for $r' \geqslant s_{i}^{\ast}(L)$ such that 
	\begin{center}
		$L \oplus \bigoplus \limits_{r' >s_{i}^{\ast}(L)}\mathbf{Ind}^{\mathbf{V}}_{\mathbf{V}'_{\nu-r'i},\mathbf{V}''_{r'i}}(L'_{r'} \boxtimes \bar{\mathbb{Q}}_{l}) \cong \bigoplus \limits_{r' \geqslant s_{i}^{\ast}(L)}\mathbf{Ind}^{\mathbf{V}}_{\mathbf{V}'_{\nu-r'i},\mathbf{V}''_{r'i}}(L''_{r'} \boxtimes \bar{\mathbb{Q}}_{l}). $
	\end{center}
	Moreover, if $L \in \mathcal{P}_{\mathbf{V},\Omega'}$, then $L'_{r'}$ and $L''_{r'}$ can be chosen in $\mathcal{Q}_{\mathbf{V}'_{\nu-r'i},\Omega' }$.
	
	(2) $s_{i}^{\ast}(L)=t_{i}^{\ast}(L)$, if $i$ is a source in $\Omega'$.
	
\end{proposition}

\section{Realization of the integrable highest weight modules}

Given a symmetric Cartan datum $(I,(-,-))$, we denote by $\alpha_{i}^{\vee}$ the simple coroot for $i\in I$. In this section, we fix a dominant weight $\Lambda$ and set $d_{i}=\langle \Lambda,\alpha_{i}^{\vee} \rangle\in \mathbb{N}$ for $i \in I$. 

Consider the framed quiver $\hat{Q}=(I \cup \hat{I},\hat{H},\hat{\Omega})$, where $$\hat{I}= \{\hat{i}| i\in I \},~ \hat{H}=H \cup \{i \rightarrow \hat{i},\hat{i} \rightarrow i| i\in I \},~ \hat{\Omega}=\Omega \cup \{i \rightarrow \hat{i}| i\in I \}.$$ Take an $\hat{I}$-graded space $\mathbf{W}$ such that $\textrm{dim} \mathbf{W}_{\hat{i}}=d_{i}$ for any $i \in I$ and set 
$$\mathbf{E}_{\mathbf{V},\mathbf{W},\hat{\Omega}}=\mathbf{E}_{\mathbf{V},\Omega}\oplus \bigoplus\limits_{i \in I} \mathbf{Hom}(\mathbf{V}_{i},\mathbf{W}_{\hat{i}}). $$

The algebraic group $G_{\mathbf{V}}$ acts naturally on $\mathbf{E}_{\mathbf{V},\mathbf{W},\hat{\Omega}}$ and there is a natural projection $$\pi_{\mathbf{W}}: \mathbf{E}_{\mathbf{V},\mathbf{W},\hat{\Omega}} \rightarrow \mathbf{E}_{\mathbf{V},\Omega}, $$
which is a  $G_{\mathbf{V}}$-equivariant trivial vector bundle.

Let $\mathcal{Q}_{\mathbf{V},\mathbf{W}}$ be the full subcategory of $\mathcal{D}^{b}_{G_{\mathbf{V}}}(\mathbf{E}_{\mathbf{V},\mathbf{W},\hat{\Omega}} )$ consisting objects of the form $(\pi_{\mathbf{W}})^{\ast}(L)$ for some object $L$ in $\mathcal{Q}_{\mathbf{V}}$. In particular, the set of simple objects in $\mathcal{Q}_{\mathbf{V},\mathbf{W}}$ is naturally bijective to $\mathcal{P}_{\mathbf{V}}$ via $(\pi_{\mathbf{W}})^{\ast}[\textrm{rank}(\pi_{\mathbf{W}})]$, since $(\pi_{\mathbf{W}})^{\ast}$ is fully faithful. The following proposition follows from an observation of Y.Li in \cite{MR3177922}. 

\begin{proposition}
	Fix any ordrer $i_{1},i_{2},\cdots,i_{n}$ of $I$, let $\boldsymbol{d}$ be the flag type of $\mathbf{W}$ such that $\boldsymbol{d}=((\hat{i}_{1})^{d_{\hat{i}_{1}}},(\hat{i}_{2})^{d_{\hat{i}_{2}}},\cdots, (\hat{i}_{n})^{d_{\hat{i}_{n}}})$. Then for any flag type $\boldsymbol{\nu}$ of $\mathbf{V}$, $L_{\boldsymbol{\nu} \boldsymbol{d}}\cong (\pi_{\mathbf{W}})^{\ast}(L_{\boldsymbol{\nu}})[\sum\limits_{ i \in I }\nu_{i}d_{i}]$.  In particular, $\mathcal{Q}_{\mathbf{V},\mathbf{W}}$ is the full subcategory consisting of direct sums of shifted summands of $L_{\boldsymbol{\nu}\boldsymbol{d}}$ for any flag types $\boldsymbol{\nu}$. 
\end{proposition}
\begin{proof}
	By definition of the induction functor, one can easily check that $$(\pi_{\mathbf{W}})^{\ast}[\sum\limits_{ i \in I }\nu_{i}d_{i} ] \cong \mathbf{Ind}^{\mathbf{V}\oplus\mathbf{W}}_{\mathbf{V},\mathbf{W}}(- \boxtimes L_{\boldsymbol{d}}). $$ By Proposition \ref{indres formula}, we finish the proof.
\end{proof}

Now we define the induction functor for framed quivers. For graded spaces $\mathbf{V}, \mathbf{V}',\mathbf{V}''$ of dimension vectors $\nu,\nu',\nu''$ respectively, such that $\nu'+\nu''=\nu$.
The induction functor $$\mathbf{Ind}^{\mathbf{V}\oplus\mathbf{W}}_{\mathbf{V}',\mathbf{V}''\oplus \mathbf{W}}(-\boxtimes-)=(p_{3})_{!}(p_{2})_{\flat}(p_{1})^{\ast}(-\boxtimes-)[d_{1}-d_{2}] $$  is defined by the following diagram
\begin{center}
	$\mathbf{E}_{\mathbf{V}',0,\hat{\tilde{\Omega}}} \times \mathbf{E}_{\mathbf{V}'',\mathbf{W},\hat{\Omega}}\xleftarrow{p_{1}} \mathbf{E}'_{\mathbf{V},\mathbf{W},\hat{\Omega}} \xrightarrow{p_{2}} \mathbf{E}''_{\mathbf{V},\mathbf{W},\hat{\Omega}} \xrightarrow{p_{3}} \mathbf{E}_{\mathbf{V},\mathbf{W},\hat{\Omega}},$
\end{center}
here $\mathbf{E}'_{\mathbf{V},\mathbf{W},\hat{\Omega}}$ is the variety consisting of $(x,\mathbf{S}, \rho_{1}, \rho_{2})$, where $x \in \mathbf{E}_{\mathbf{V},\mathbf{W},\hat{\Omega}}$ and $\mathbf{S}$ is an $x$-stable subspace of $\mathbf{V}\oplus \mathbf{W}$ with dimension vector $|\mathbf{V}''\oplus \mathbf{W}|$, $ \rho_{1}: \mathbf{V}\oplus\mathbf{W}/\mathbf{S} \simeq \mathbf{V}'$ is a linear isomorphism of graded spaces, and $\rho_{2}:\mathbf{S} \simeq \mathbf{V}''\oplus\mathbf{W}$ is a linear isomorphism of graded spaces such that $\rho_{2}|_{\mathbf{W}}=id_{\mathbf{W}}$. (In this case, we also forget $\mathbf{W}$ and just say $\rho_{2}$ is a linear isomorphism of $\mathbf{S} \simeq \mathbf{V}''$.)  Here $\mathbf{E}''_{\mathbf{V},\mathbf{W},\hat{\Omega}}$ is the variety consisting of $(x,\mathbf{S})$ with the same conditions as above.  The morphisms $p_{1},p_{2},p_{3}$ are defined by same formulas in equations (\ref{Lind}). The shifts $d_{1},d_{2}$ are given by the relative dimensions of $p_{1}$ and $p_{2}$ respectively.

Notice that $p_{1}$ and $p_{3}$ are $G_{\mathbf{V}}$-equivariant and $p_{2}$ is a $G_{\mathbf{V}'}\times G_{\mathbf{V}''}$ principal bundle, we can see that the induction functor, defined by the composition
\begin{equation*}
	\begin{split}
		&\mathcal{D}^{b}_{G_{\mathbf{V}'}}(\mathbf{E}_{\mathbf{V}',\Omega}) \times \mathcal{D}^{b}_{G_{\mathbf{V}''}}(\mathbf{E}_{\mathbf{V}'',\mathbf{W},\hat{\Omega}}) \cong \mathcal{D}^{b}_{G_{\mathbf{V}'}}(\mathbf{E}_{\mathbf{V}',0,\hat{\Omega}}) \times \mathcal{D}^{b}_{G_{\mathbf{V}''}}(\mathbf{E}_{\mathbf{V}'',\mathbf{W},\hat{\Omega}}) \xrightarrow{\boxtimes} \\ &\mathcal{D}^{b}_{G_{\mathbf{V}'}\times G_{\mathbf{V}''}}(\mathbf{E}_{\mathbf{V}',0,\hat{\Omega}} \times\mathbf{E}_{\mathbf{V}'',\mathbf{W},\hat{\Omega}}) \xrightarrow{(p_{1})^{\ast}}  \mathcal{D}^{b}_{G_{\mathbf{V}'}\times G_{\mathbf{V}''} \times\mathbf{G}_{\mathbf{V}}}(\mathbf{E}'_{\mathbf{V},\mathbf{W},\hat{\Omega}})\xrightarrow{(p_{2})_{\flat}} \\ &\mathcal{D}^{b}_{\mathbf{G}_{\mathbf{V}}}(\mathbf{E}''_{\mathbf{V},\mathbf{W},\hat{\Omega}}) \xrightarrow{(p_{3})_{!}} \mathcal{D}^{b}_{G_{\mathbf{V}}}(\mathbf{E}_{\mathbf{V},\mathbf{W},\hat{\Omega}}),
	\end{split}
\end{equation*}
denfines a functor $\mathcal{D}^{b}_{G_{\mathbf{V}'}}(\mathbf{E}_{\mathbf{V}',\Omega}) \times \mathcal{D}^{b}_{G_{\mathbf{V}''}}(\mathbf{E}_{\mathbf{V}'',\mathbf{W},\hat{\Omega}}) \rightarrow  \mathcal{D}^{b}_{G_{\mathbf{V}}}(\mathbf{E}_{\mathbf{V},\mathbf{W},\hat{\Omega}})$. Dually, we can also define $\mathbf{Ind}^{\mathbf{V}\oplus\mathbf{W}}_{\mathbf{V}'\oplus \mathbf{W},\mathbf{V}''}: \mathcal{D}^{b}_{G_{\mathbf{V}'}}(\mathbf{E}_{\mathbf{V}',\mathbf{W},\hat{\Omega}}) \times \mathcal{D}^{b}_{G_{\mathbf{V}''}}(\mathbf{E}_{\mathbf{V}'',\Omega}) \rightarrow  \mathcal{D}^{b}_{G_{\mathbf{V}}}(\mathbf{E}_{\mathbf{V},\mathbf{W},\hat{\Omega}})$.

Recall that in the definition of Lusztig's induction functor, one need to consider the following diagram
\begin{center}
	$\mathbf{E}_{\mathbf{V}',0,\hat{\tilde{\Omega}}} \times \mathbf{E}_{\mathbf{V}'',\mathbf{W},\hat{\Omega}}\xleftarrow{p^{\diamond}_{1}} \mathbf{E}^{\diamond}_{\mathbf{V},\mathbf{W},\hat{\Omega}} \xrightarrow{p^{\diamond}_{2}} \mathbf{E}''_{\mathbf{V},\mathbf{W},\hat{\Omega}} \xrightarrow{p_{3}} \mathbf{E}_{\mathbf{V},\mathbf{W},\hat{\Omega}},$
\end{center}
where $\mathbf{E}^{'}_{\mathbf{V},\mathbf{W},\hat{\Omega}}$ is replaced by the variety $\mathbf{E}^{\diamond}_{\mathbf{V},\mathbf{W},\hat{\Omega}}$, which consists of $(x,\mathbf{S}, \rho_{1}, \rho_{2})$ such that $x \in \mathbf{E}_{\mathbf{V},\mathbf{W},\hat{\Omega}}$ and $\mathbf{S}$ is an $x$-stable subspace of $\mathbf{V}\oplus \mathbf{W}$ with dimension vector $|\mathbf{V}''\oplus \mathbf{W}|$, $ \rho_{1}: \mathbf{V}\oplus\mathbf{W}/\mathbf{S} \simeq \mathbf{V}', \rho_{2}:\mathbf{S} \simeq \mathbf{V}''\oplus\mathbf{W}$ are  linear isomorphisms of graded spaces. The morphisms are defined by equations  (\ref{Lind}). Then $p_{1},p_{2},p^{\diamond}_{1},p^{\diamond}_{2}$ form  commutative diagrams
\[
\xymatrix{
	&  \mathbf{E}^{\diamond}_{\mathbf{V},\mathbf{W},\hat{\Omega}}  \ar[d]^{p^{\diamond}_{2}} \ar[ld]_{\rho}
	\\
	\mathbf{E}^{'}_{\mathbf{V},\mathbf{W},\hat{\Omega}} \ar[r]^{p_{2}}
	&  	\mathbf{E}^{''}_{\mathbf{V},\mathbf{W},\hat{\Omega}}  ,
}
\]
\[
\xymatrix{
	\mathbf{E}_{\mathbf{V}',0,\hat{\tilde{\Omega}}} \times \mathbf{E}_{\mathbf{V}'',\mathbf{W},\hat{\Omega}}
	&  \mathbf{E}^{\diamond}_{\mathbf{V},\mathbf{W},\hat{\Omega}} \ar[l]_-{p^{\diamond}_{1}}  
	\\
	\mathbf{E}^{'}_{\mathbf{V},\mathbf{W},\hat{\Omega}}\ar[u]^{p_{1}} \ar[ur]_{e}
	&  	 ,
}
\]
where $\rho$ is the obvious $G_{\mathbf{W}}$-principal bundle and $e$ is the inclusion. These commutative diagrams imply that $(p_{2})_{\flat}(p_{1})^{\ast} \mathbf{For}^{G_{\mathbf{W}}} \cong \mathbf{For}^{G_{\mathbf{W}}}(p^{\diamond}_{2})_{\flat}(p^{\diamond}_{1})^{\ast} $, here $\mathbf{For}^{G_{\mathbf{W}}}$ is the functor forgetting the $G_{\mathbf{W}}$-equivariant structures. It implies that if $B$ is also $G_{\mathbf{W}}$-equivariant, then $\mathbf{Ind}^{\mathbf{V}\oplus\mathbf{W}}_{\mathbf{V}',\mathbf{V}''\oplus \mathbf{W}}(A\boxtimes B)$ coincides with Lusztig's induction. That's why we can still denote our functors by $\mathbf{Ind}$. In particular, since $L_{\boldsymbol{\nu} \boldsymbol{d}}$ is $G_{\mathbf{W}}$-equivariant, our induction also satisfies the induction formula in Proposition \ref*{indres formula},
\begin{equation*}
	\begin{split}
		\mathbf{Ind}^{\mathbf{V}\oplus\mathbf{W}}_{\mathbf{V}',\mathbf{V}''\oplus \mathbf{W}}(L_{\boldsymbol{\nu}'}\boxtimes L_{\boldsymbol{\nu}''\boldsymbol{d}})=L_{\boldsymbol{\nu}'\boldsymbol{\nu}''\boldsymbol{d}},\\
		\mathbf{Ind}^{\mathbf{V}\oplus\mathbf{W}}_{\mathbf{V}'\oplus \mathbf{W},\mathbf{V}''}(L_{\boldsymbol{\nu}'\boldsymbol{d}}\boxtimes L_{\boldsymbol{\nu}''})=L_{\boldsymbol{\nu}'\boldsymbol{d}\boldsymbol{\nu}''}.
	\end{split}
\end{equation*} 
Moreover, by a similar argument as Remark \ref{remarkFD}, the induction functor also commutes with Fourier-Deligne transforms.

\subsection{Localizations}
For any $i \in I$, we choose an orientation $\Omega^{i}$ of $Q$ such that $i$ is a source in $\Omega^{i}$, then $i$ is also a source in the orientation $\hat{\Omega}^{i}$ of the framed quiver $\hat{Q}$. Regard $\hat{Q}$ as a new quiver, then by section 2.4, there is a partition $\mathbf{E}_{\mathbf{V},\mathbf{W},i}^{ \geqslant 1} \cup \mathbf{E}_{\mathbf{V},\mathbf{W},i}^{0}= \mathbf{E}_{\mathbf{V},\mathbf{W},\hat{\Omega}^{i}}$, where $\mathbf{E}_{\mathbf{V},\mathbf{W},i}^{0}$ is the open subset 
$$\mathbf{E}_{\mathbf{V},\mathbf{W},i}^{0}=\{x \in \mathbf{E}_{\mathbf{V},\mathbf{W},\hat{\Omega}^{i}}| \textrm{dim} \textrm{Ker} (\bigoplus\limits_{h \in \hat{\Omega}^{i},h'=i} x_{h}: \mathbf{V}_{i} \rightarrow   \mathbf{W}_{\hat{i}} \oplus \bigoplus\limits_{h \in \Omega^{i},h'=i} \mathbf{V}_{h''} )=0  \} ,$$
and  $\mathbf{E}_{\mathbf{V},\mathbf{W},i}^{ \geqslant 1}$ is its complement. Denote the open embedding of $\mathbf{E}_{\mathbf{V},\mathbf{W},i}^{0}$ by $j_{\mathbf{V},i}$.

Following \cite{MR3200442}, let $\mathcal{N}_{\mathbf{V},i}$ be the full subcategory of $\mathcal{D}^{b}_{G_{\mathbf{V}}}(\mathbf{E}_{\mathbf{V},\mathbf{W},\hat{\Omega}^{i}})$ consisting of objects whose supports are contained in $\mathbf{E}_{\mathbf{V},\mathbf{W},i}^{ \geqslant 1}$, then $\mathcal{N}_{\mathbf{V},i}$ is a thick subcategory. We can see that the Verdier quotient $\mathcal{D}^{b}_{G_{\mathbf{V}}}(\mathbf{E}_{\mathbf{V},\mathbf{W},\hat{\Omega}^{i}})/\mathcal{N}_{\mathbf{V},i}$ is a triangulated category,  with a natural perverse $t$-structure induced from the perverse $t$-structure of $\mathcal{D}^{b}_{G_{\mathbf{V}}}(\mathbf{E}_{\mathbf{V},\mathbf{W},\hat{\Omega}^{i}})$. See details in \cite{BBD}.

Similarly, given an orientation $\Omega$ of $Q$, let $\hat{\Omega}$ be the associated orientation of $\hat{Q}$. Let $\mathcal{N}_{\mathbf{V}}$ be the thick subcategory of $\mathcal{D}^{b}_{G_{\mathbf{V}}}(\mathbf{E}_{\mathbf{V},\mathbf{W},\hat{\Omega}})$ generated by objects in $\mathcal{F}_{\hat{\Omega}^{i},\hat{\Omega}}( \mathcal{N}_{\mathbf{V},i}),i\in I$, then we can also define the Verdier quotient $\mathcal{D}^{b}_{G_{\mathbf{V}}}(\mathbf{E}_{\mathbf{V},\mathbf{W},\hat{\Omega}})/\mathcal{N}_{\mathbf{V}}$. It is also a triangulated category, with a natural perverse $t$-structure. If there is no ambiguity, we omit $\mathcal{F}_{\hat{\Omega},\hat{\Omega}^{i}}$ and denote $\mathcal{F}_{\hat{\Omega}^{i},\hat{\Omega}}(\mathcal{N}_{\mathbf{V},i})$ by $\mathcal{N}_{\mathbf{V},i}$.

\begin{definition}
	(a) For any $i\in I$ and an orientation $\Omega^{i}$ such that $i$ is a source in $\Omega^{i}$, define the localizations  of $\mathcal{Q}_{\mathbf{V},\mathbf{W}}$ and $\mathcal{D}^{b,ss}_{G_{\mathbf{V}}}(\mathbf{E}_{\mathbf{V},\mathbf{W},\hat{\Omega}^{i}})$  at $i$  to be the full subcategories of $\mathcal{D}^{b}_{G_{\mathbf{V}}}(\mathbf{E}_{\mathbf{V},\mathbf{W},\hat{\Omega}^{i}})/\mathcal{N}_{\mathbf{V},i}$ consisting of objects which are isomorphic to objects of $\mathcal{Q}_{\mathbf{V},\mathbf{W}}$ and $\mathcal{D}^{b,ss}_{G_{\mathbf{V}}}(\mathbf{E}_{\mathbf{V},\mathbf{W},\hat{\Omega}^{i}})$ respectively. Denote them  by $\mathcal{Q}_{\mathbf{V},\mathbf{W}}/\mathcal{N}_{\mathbf{V},i}$ and $\mathcal{D}^{b,ss}_{G_{\mathbf{V}}}(\mathbf{E}_{\mathbf{V},\mathbf{W},\hat{\Omega}^{i}})/\mathcal{N}_{\mathbf{V},i}$ respectively.\\
	(b)For an orientation $\Omega$ of $Q$, define the global localizations of $\mathcal{Q}_{\mathbf{V},\mathbf{W}}$ and $\mathcal{D}^{b,ss}_{G_{\mathbf{V}}}(\mathbf{E}_{\mathbf{V},\mathbf{W},\hat{\Omega}})$    to be the full subcategories of $\mathcal{D}^{b}_{G_{\mathbf{V}}}(\mathbf{E}_{\mathbf{V},\mathbf{W},\hat{\Omega}})/\mathcal{N}_{\mathbf{V}}$ consisting of objects which are isomorphic to objects of $\mathcal{Q}_{\mathbf{V},\mathbf{W}}$ and $\mathcal{D}^{b,ss}_{G_{\mathbf{V}}}(\mathbf{E}_{\mathbf{V},\mathbf{W},\hat{\Omega}})$ respectively. Denote them  by $\mathcal{Q}_{\mathbf{V},\mathbf{W}}/\mathcal{N}_{\mathbf{V}}$ and $\mathcal{D}^{b,ss}_{G_{\mathbf{V}}}(\mathbf{E}_{\mathbf{V},\mathbf{W},\hat{\Omega}})/\mathcal{N}_{\mathbf{V}}$.
\end{definition}

\begin{remark}
	Notice that if $\tilde{\Omega}^{i}$ is another orientation such that $i$ is a source in $\tilde{\Omega}^{i}$, then by definition $\tilde{\Omega}^{i}$ and $\Omega^{i}$ determine the same $\mathcal{N}_{\mathbf{V},i}$ up to Fourier-Deligne transforms. In particular, the localizations defined above are independent of the choices of those $\hat{\Omega}^{i}$.
\end{remark}

For any open embedding $j: U\rightarrow X$, the middle extension functor
\begin{equation*}
	j_{!\ast}:Perv(U) \rightarrow Perv(X)
\end{equation*}
can be naturally extended to an action on objects of the semisimple category. (But it is not a functor, since it can not be defined on morphisms.)  More precisely, for any direct sum of simple perverse sheaves up to shifts $L=\bigoplus\limits K[n]$, we set
\begin{equation*}
	j_{!\ast}(L)=\bigoplus\limits j_{!\ast}(K)[n].
\end{equation*}

\begin{lemma}\label{local}
	For any semisimple complex $L$ on $\mathbf{E}^{0}_{\mathbf{V},\mathbf{W},i}$,  there is an isomorphism   in the localization $\mathcal{D}^{b}_{G_{\mathbf{V}}}(\mathbf{E}_{\mathbf{V},\mathbf{W},\hat{\Omega}^{i}})/\mathcal{N}_{\mathbf{V},i},$ 
	$$(j_{\mathbf{V},i})_{!\ast} (L) \cong (j_{\mathbf{V},i})_{!} (L). $$ 
	In particular,  the functor $(j_{\mathbf{V},i})_{!}$ restricts to  equivalences of categories 
	\[
	\xymatrix{
		(j_{\mathbf{V},i})^{\ast} (\mathcal{Q}_{\mathbf{V},\mathbf{W}}) \ar@<0.5ex>[r]^{(j_{\mathbf{V},i})_{!}} & \mathcal{Q}_{\mathbf{V},\mathbf{W}}/\mathcal{N}_{\mathbf{V},i} \ar@<0.5ex>[l]^{(j_{\mathbf{V},i})^{\ast}}
	};
	\]
	\[
	\xymatrix{
		\mathcal{D}^{b,ss}_{G_{\mathbf{V}}}(\mathbf{E}^{0}_{\mathbf{V},\mathbf{W},i}) \ar@<0.5ex>[r]^{(j_{\mathbf{V},i})_{!}} & \mathcal{D}^{b,ss}_{G_{\mathbf{V}}}(\mathbf{E}_{\mathbf{V},\mathbf{W},\hat{\Omega}^{i}})/\mathcal{N}_{\mathbf{V},i} \ar@<0.5ex>[l]^{(j_{\mathbf{V},i})^{\ast}}
	}.
	\]
	
\end{lemma}
\begin{proof}
	
	We only need to consider simple perverse sheaf $L$. Let $K=(j_{\mathbf{V},i})_{!\ast} (L) $, then there is a canonical triangle 
	\begin{equation*}
		(j_{\mathbf{V},i})_{!} (j_{\mathbf{V},i})^{\ast} (K)  \rightarrow K \rightarrow  i_{\ast}i^{\ast}(K) \rightarrow (j_{\mathbf{V},i})_{!} (j_{\mathbf{V},i})^{\ast} (K)[1]
	\end{equation*}
	where $i:\mathbf{E}^{\geqslant 1}_{\mathbf{V},\mathbf{W},i}\rightarrow \mathbf{E}_{\mathbf{V},\mathbf{W},\hat{\Omega}^{i}}$ is the closed embedding. Notice that $i_{\ast}i^{\ast}(K)$ has support contained in $\mathbf{E}^{\geqslant 1}_{\mathbf{V},\mathbf{W},i}$. Hence  $(j_{\mathbf{V},i})_{!} (j_{\mathbf{V},i})^{\ast} (K)$ is isomorphic to $K$ in the localization $\mathcal{D}^{b}_{G_{\mathbf{V}}}(\mathbf{E}_{\mathbf{V},\mathbf{W},\hat{\Omega}^{i}})/\mathcal{N}_{\mathbf{V},i}$, and the first statement follows from $(j_{\mathbf{V},i})^{\ast} (K) \cong L$.
	
	Notice that   $(j_{\mathbf{V},i})^{\ast}$ and $(j_{\mathbf{V},i})_{!}$ are quasi-inverse equivalences between  $\mathcal{D}^{b}_{G_{\mathbf{V}}}(\mathbf{E}^{0}_{\mathbf{V},\mathbf{W},i})$ and the localization $\mathcal{D}^{b}_{G_{\mathbf{V}}}(\mathbf{E}_{\mathbf{V},\mathbf{W},\hat{\Omega}^{i}})/\mathcal{N}_{\mathbf{V},i}$, the above argument implies that these equivalences restrict to $\mathcal{Q}_{\mathbf{V},\mathbf{W}}$ and $\mathcal{D}^{b,ss}_{G_{\mathbf{V}}}(\mathbf{E}_{\mathbf{V},\mathbf{W},\hat{\Omega}^{i}})$ and we finish the proof.
\end{proof}

\subsection{The functors of localizations at $i$}
In this subsection, we fix an orientation $\Omega=\Omega^{i}$ of $Q$ such that $i$ is a source.

For any $n\in \mathbb{N}$, take graded spaces $\mathbf{V}, \mathbf{V}'$ of dimension vectors $\nu,\nu'$ respectively, such that $\nu'+ni=\nu$, we will define varieties and morphisms appearing in following diagram, and then define a functor $\mathcal{E}_{i}^{(n)}$.
\[
\xymatrix{
	\mathbf{E}_{\mathbf{V},\mathbf{W},\hat{\Omega}}
	&
	& \mathbf{E}_{\mathbf{V}',\mathbf{W},\hat{\Omega}}  \\
	\mathbf{E}^{0}_{\mathbf{V},\mathbf{W},i} \ar[d]_{\phi_{\mathbf{V},i}} \ar[u]^{j_{\mathbf{V},i}}
	&
	& \mathbf{E}^{0}_{\mathbf{V}',\mathbf{W},i} \ar[d]^{\phi_{\mathbf{V}',i}} \ar[u]_{j_{\mathbf{V}',i}} \\
	\dot{\mathbf{E}}_{\mathbf{V},\mathbf{W},i} \times \mathbf{Gr}(\nu_i, \tilde{\nu}_{i})
	& \dot{\mathbf{E}}_{\mathbf{V},\mathbf{W},i} \times \mathbf{Fl}(\nu'_{i},\nu_{i},\tilde{\nu}_{i}) \ar[r]^{q_{2}} \ar[l]_{q_{1}}
	& \dot{\mathbf{E}}_{\mathbf{V},\mathbf{W},i} \times \mathbf{Gr}(\nu'_{i}, \tilde{\nu}_{i})
}
\]

Define an affine subspace
\begin{equation*}
	\dot{\mathbf{E}}_{\mathbf{V},\mathbf{W},i} =\bigoplus\limits_{h \in \Omega, h'\neq i} \mathbf{Hom}(\mathbf{V}_{h'},\mathbf{V}_{h''}) \oplus \bigoplus \limits_{ i' \neq i } \mathbf{Hom}(\mathbf{V}_{i'},\mathbf{W}_{\hat{i}'})  .
\end{equation*}
For any $x \in \mathbf{E}_{\mathbf{V},\mathbf{W},\hat{\Omega}}$, we denote by $\dot{x}=(x_{h})_{h\in \hat{\Omega},h' \neq i}$. Then there is a morphism 
\begin{align*}
	\phi_{\mathbf{V},i}:\mathbf{E}^{0}_{\mathbf{V},\mathbf{W},i} &\rightarrow  \dot{\mathbf{E}}_{\mathbf{V},\mathbf{W},i} \times \mathbf{Gr}(\nu_i, \tilde{\nu}_{i})\\
	x &\mapsto (\dot{x}, {\rm{Im}}  (\bigoplus \limits_{h \in \hat{\Omega}, h'=i} x_{h}  ) ),
\end{align*}
where $\nu_{i}={\rm{dim}}\mathbf{V}_{i}$, $\tilde{\nu}_{i}=\sum \limits_{h\in \Omega,h'=i}{\rm{dim}}\mathbf{V}_{h''}+d_{i}$, and $\mathbf{Gr}(\nu_i, \tilde{\nu}_{i})$ is the Grassmannian consisting of $\nu_{i}$-dimensional subspaces of $\tilde{\nu}_{i}$-dimensional space $(\bigoplus\limits_{h'=i,h \in \Omega^{i}}\mathbf{V}_{h''})\oplus \mathbf{W}_{\hat{i}}$. We can check by definition that $\phi_{\mathbf{V},i}$ is a principal $\mathbf{GL}(\mathbf{V}_{i})$-bundle. Let  
\begin{equation*}
	\mathbf{Fl}(\nu_{i}-n,\nu_{i},\tilde{\nu}_{i})=\{ \mathbf{S}_{1}\subset \mathbf{S}_{2}  \subset (\bigoplus\limits_{h'=i}\mathbf{V}_{h''})\oplus \mathbf{W}_{\hat{i}} )|{\rm{dim}} \mathbf{S}_{1} = \nu_{i}-n, {\rm{dim}}\mathbf{S}_{2}=\nu_{i}  \}.
\end{equation*}
be the flag variety and $q_{1},q_{2}$ are natural projections
\begin{equation*}
	q_{1}(\dot{x}, \mathbf{S}_{1},\mathbf{S}_{2})=(\dot{x},\mathbf{S}_{2});
\end{equation*}
\begin{equation*}
	q_{2}(\dot{x},\mathbf{S}_{1},\mathbf{S}_{2})=(\dot{x},\mathbf{S}_{1}).
\end{equation*}

For any $I$ graded space $\mathbf{S}=\bigoplus\limits_{ j\in I }\mathbf{S}_{j}$, we always denote $\bigoplus\limits_{ j\neq i\in I }\mathbf{S}_{j}$ by $\dot{\mathbf{S}}$. In particular, $G_{\dot{\mathbf{V}}}$ acts on $\dot{\mathbf{E}}_{\mathbf{V},\mathbf{W},i}$ naturally. Let $G_{\mathbf{V}_{i}}$ acts trivially on $ \dot{\mathbf{E}}_{\mathbf{V},\mathbf{W},i} $ and those Grassmannians and flag varieties, then $q_{1},q_{2}$ are $G_{\dot{\mathbf{V}}}$-equivariant morphisms.

\begin{definition}\label{defineE}
	Define the functor $\mathcal{E}^{(n)}_{i}:\mathcal{D}^{b}_{G_{\mathbf{V}}}(\mathbf{E}_{\mathbf{V},\mathbf{W},\hat{\Omega}}) \rightarrow \mathcal{D}^{b}_{G_{\mathbf{V}'}}(\mathbf{E}_{\mathbf{V}',\mathbf{W},\hat{\Omega}})$ via
	\begin{equation*}
		\mathcal{E}^{(n)}_{i}= (j_{\mathbf{V}',i})_{!} (\phi_{\mathbf{V}',i})^{\ast} (q_{2})_{!}(q_{1})^{\ast} (\phi_{\mathbf{V},i})_{\flat}(j_{\mathbf{V},i})^{\ast}[-n\nu_{i}]
	\end{equation*}
	\begin{equation*}
		\begin{split}
			&\mathcal{D}^{b}_{G_{\mathbf{V}}}(\mathbf{E}_{\mathbf{V},\mathbf{W},\hat{\Omega}}) \xrightarrow{(j_{\mathbf{V},i})^{\ast}} \mathcal{D}^{b}_{G_{\mathbf{V}}}(\mathbf{E}^{0}_{\mathbf{V},\mathbf{W},i}) \xrightarrow{(\phi_{\mathbf{V},i})_\flat}  \mathcal{D}^{b}_{G_{\dot{\mathbf{V}}}}(\dot{\mathbf{E}}_{\mathbf{V},\mathbf{W},i} \times \mathbf{Gr}(\nu_{i}, \tilde{\nu}_{i})) \xrightarrow{(q_{2})_{!}(q_{1})^{\ast}} \\ &\mathcal{D}^{b}_{G_{\dot{\mathbf{V}}}}(\dot{\mathbf{E}}_{\mathbf{V},\mathbf{W},i} \times \mathbf{Gr}(\nu'_{i}, \tilde{\nu}_{i})) \xrightarrow{(\phi_{\mathbf{V}',i})^{\ast}} \mathcal{D}^{b}_{G_{\mathbf{V}'}}(\mathbf{E}^{0}_{\mathbf{V}',\mathbf{W},i}) \xrightarrow{(j_{\mathbf{V}',i})_{!}} \mathcal{D}^{b}_{G_{\mathbf{V}'}}(\mathbf{E}^{0}_{\mathbf{V}',\mathbf{W},\hat{\Omega}}).
		\end{split}
	\end{equation*}
	
	Notice that $\mathcal{E}^{(n)}_{i}(\mathcal{N}_{\mathbf{V},i} )=0$, thus $\mathcal{E}^{(n)}_{i}$ descends to a functor between localizations, still denoted by 
	\begin{equation*}
		\mathcal{E}^{(n)}_{i}:\mathcal{D}^{b}_{G_{\mathbf{V}}}(\mathbf{E}_{\mathbf{V},\mathbf{W},\hat{\Omega}})/\mathcal{N}_{\mathbf{V},i} \rightarrow \mathcal{D}^{b}_{G_{\mathbf{V}'}}(\mathbf{E}_{\mathbf{V}',\mathbf{W},\hat{\Omega}})/\mathcal{N}_{\mathbf{V'},i}.
	\end{equation*}
	In particular, we denote by $\mathcal{E}_{i}=\mathcal{E}_{i}^{(1)}. $
\end{definition} 

\begin{lemma}
	If $\tilde{\Omega}$ is another orientation such that the edges near $i$ have the same orientations in $\tilde{\Omega}$ and $\Omega$. In particular, $i$ is a source in $\tilde{\Omega}$. We define $\tilde{\mathcal{E}}^{(n)}_{i}$ in the same way as Definition \ref{defineE}, then  $\tilde{\mathcal{E}}^{(n)}_{i} \mathcal{F}_{\hat{\Omega},\hat{\tilde{\Omega}}} \cong  \mathcal{F}_{\hat{\Omega},\hat{\tilde{\Omega}}} \mathcal{E}^{(n)}_{i}$. In particular, $\mathcal{E}^{(n)}_{i}$ is independent of the choice of $\Omega$ such that $i$ is a source in $\Omega$.  
\end{lemma}

\begin{proof}
	Since $i$ is a source in $\tilde{\Omega}$ and $\Omega$,  $\mathcal{F}_{\hat{\Omega},\hat{\tilde{\Omega}}} $ only involves the edges which do not connect to $i$. Let $\dot{Q}$ be the quiver obtained by removing $i$ from $Q$, then  $\mathcal{F}_{\hat{\Omega},\hat{\tilde{\Omega}}} $ indeed acts on $\dot{\mathbf{E}}_{\mathbf{V},\mathbf{W},i}$ as the Fourier-Deligne transform of $\dot{Q}$. By definition and base change, it is easy to check that $(q_{2})_{!}(q_{1})^{\ast}$ commutes with the Fourier-Deligne transforms of $\dot{Q}$ and we finish the proof.
\end{proof}

For $j \in I$ and graded spaces $\mathbf{V},\mathbf{V}'$  and $\mathbf{V}''$ such that  $|\mathbf{V}''|-nj=|\mathbf{V}|,|\mathbf{V}'|=nj$, define the functor $\mathcal{F}_{j}^{(n)}:\mathcal{D}^{b}_{G_{\mathbf{V}}}(\mathbf{E}_{\mathbf{V},\mathbf{W},\hat{\Omega}}) \rightarrow \mathcal{D}^{b}_{G_{\mathbf{V}''}}(\mathbf{E}_{\mathbf{V}'',\mathbf{W},\hat{\Omega}})$ and $\mathcal{F}_{j}^{(n)}:\mathcal{Q}_{\mathbf{V},\mathbf{W}} \rightarrow \mathcal{Q}_{\mathbf{V}'',\mathbf{W}} $   by
\begin{equation*}
	\mathcal{F}_{j}^{(n)}= \mathbf{Ind}^{\mathbf{V}''\oplus\mathbf{W}}_{\mathbf{V}',\mathbf{V}\oplus \mathbf{W}}(\overline{\mathbb{Q}}_{l} \boxtimes -)
\end{equation*}
where $\overline{\mathbb{Q}}_{l}$ is the constant sheaf of $\mathbf{E}_{\mathbf{V}',0,\hat{\Omega}} \cong \mathbf{E}_{\mathbf{V}',\Omega}$.

\begin{lemma}
	The functor $\mathcal{F}^{(n)}_{j}$ sends objects of $ \mathcal{N}_{\mathbf{V},k}$ to objects of $ \mathcal{N}_{\mathbf{V}'',k}$.
\end{lemma}
\begin{proof}
	Since the induction functor for $\mathcal{D}^{b}_{G_{\mathbf{V}}}(\mathbf{E}_{\mathbf{V},\mathbf{W},\hat{\Omega}})$ commutes with Fourier-Deligne tarnsformation, we can assume $k$ is a source. In this case, the simple representation $S_{k}$ of $Q$ (and $\hat{Q}$) at $k$ is injective. In particular,  any exact sequence of representations of $\hat{Q}$ of the form
	$$ 0 \rightarrow S_{k}^{\oplus m} \rightarrow Y \rightarrow Z\rightarrow 0$$
	must be split, then by definition of $\mathcal{F}^{(n)}_{j}$, we can see that $\text{supp} (\mathcal{F}^{(n)}_{j}(L)) \subseteq \mathbf{E}^{\geqslant 1}_{\mathbf{V}'',\mathbf{W},k}$ if  $\text{supp} (L) \subseteq \mathbf{E}^{\geqslant 1}_{\mathbf{V},\mathbf{W},k}$. we finish the proof.
\end{proof}

As a corollary, we can well define functors $\mathcal{F}_{j}^{(n)}$ between localizations at $i$.
\begin{definition}
	For $j \in I$ and graded spaces $\mathbf{V},\mathbf{V}'$  and $\mathbf{V}''$ such that  $|\mathbf{V}''|-nj=|\mathbf{V}|$ and $|\mathbf{V}'|=nj$, define the functor $\mathcal{F}_{j}^{(n)}:\mathcal{D}^{b}_{G_{\mathbf{V}}}(\mathbf{E}_{\mathbf{V},\mathbf{W},\hat{\Omega}}) /\mathcal{N}_{\mathbf{V},i}  \rightarrow \mathcal{D}^{b}_{G_{\mathbf{V}''}}(\mathbf{E}_{\mathbf{V}'',\mathbf{W},\hat{\Omega}})/\mathcal{N}_{\mathbf{V}'',i} $ and $\mathcal{F}_{j}^{(n)}:\mathcal{Q}_{\mathbf{V},\mathbf{W}}/\mathcal{N}_{\mathbf{V},i} \rightarrow \mathcal{Q}_{\mathbf{V}'',\mathbf{W}}/\mathcal{N}_{\mathbf{V}'',i}$ by
	\begin{equation*}
		\mathcal{F}_{j}^{(n)}= \mathbf{Ind}^{\mathbf{V}''\oplus\mathbf{W}}_{\mathbf{V}',\mathbf{V}\oplus \mathbf{W}}(\overline{\mathbb{Q}}_{l} \boxtimes -)
	\end{equation*}
	where $\overline{\mathbb{Q}}_{l}$ is the constant sheaf of $\mathbf{E}_{\mathbf{V}',0,\hat{\Omega}}$. In particular, we set $\mathcal{F}_{j}=\mathcal{F}_{j}^{(1)}. $
\end{definition}

The functor $\mathcal{F}_{j}$ is defined by Lusztig's induction functor. It can be described by functors induced by morphisms appearing in the definition of $\mathcal{E}^{(n)}_i$ as follows.

For graded spaces $\mathbf{V}$ and $\mathbf{V}''$ such that  $|\mathbf{V}|+j=|\mathbf{V}''|$, let
$\mathbf{E}^{'',0}_{\mathbf{V}'',\mathbf{W},i}=(p_{3})^{-1}(\mathbf{E}^{0}_{\mathbf{V}'',\mathbf{W},i})$ and $\mathbf{E}^{',0}_{\mathbf{V}'',\mathbf{W},i}=(p_{2})^{-1}(\mathbf{E}^{'',0}_{\mathbf{V}'',\mathbf{W},i})$, then we have the following commutative diagram
\[
\xymatrix{
	\mathbf{E}_{\mathbf{V},\mathbf{W},\hat{\Omega}} & \mathbf{E}^{'}_{\mathbf{V}'',\mathbf{W},\hat{\Omega}}\ar[l]_{p_{1}} \ar[r]^{p_{2}}& \mathbf{E}^{''}_{\mathbf{V}'',\mathbf{W},\hat{\Omega}} \ar[r]^{p_{3}}& \mathbf{E}_{\mathbf{V}'',\mathbf{W},\hat{\Omega}} \\
	\mathbf{E}^{0}_{\mathbf{V},\mathbf{W},i} \ar[u]_{j_{1}} & \mathbf{E}^{',0}_{\mathbf{V}'',\mathbf{W},i} \ar[l]_{\tilde{p}_{1}} \ar[r]^{\tilde{p}_{2}} \ar[u]_{j_{2}} & \mathbf{E}^{'',0}_{\mathbf{V}'',\mathbf{W},i}  \ar[r]^{\tilde{p}_{3}} \ar[u]_{j_{3}} & \mathbf{E}^{0}_{\mathbf{V}'',\mathbf{W},i} \ar[u]_{j_{4}} 
}
\]
where $j_{i},i=1,2,3,4$ are open embeddings  $\tilde{p}_{1},\tilde{p}_{2},\tilde{p}_{3}$ are obvious morphisms induced from $p_{1},p_{2},p_{3}$ respectively. Notice that the right square is Cartesian, we have 
\begin{equation}\label{Eq1}
	(j_{4})^{\ast}(p_{3})_{!}(p_{2})_{\flat}(p_{1})^{\ast} =(\tilde{p}_{3})_{!}(\tilde{p}_{2})_{\flat}(\tilde{p}_{1})^{\ast} (j_{1})^{\ast}.
\end{equation}

\textbf{Case (I)} $i=j$. Consider the following commutative diagram
\[
\xymatrix{
	\mathbf{E}^{0}_{\mathbf{V},\mathbf{W},i} \ar[d]_{\phi_{1}} & \mathbf{E}^{',0}_{\mathbf{V}'',\mathbf{W},i} \ar[l]_{\tilde{p}_{1}} \ar[r]^{\tilde{p}_{2}} \ar[d]_{\phi_{2}}& \mathbf{E}^{'',0}_{\mathbf{V}'',\mathbf{W},i}  \ar[r]^{\tilde{p}_{3}}  \ar[d]_{\phi_{3}}& \mathbf{E}^{0}_{\mathbf{V}'',\mathbf{W},i} \ar[d]_{\phi_{4}} \\
	\txt{$\dot{\mathbf{E}}_{\mathbf{V},\mathbf{W},i}$\\ $\times$\\ $\mathbf{Gr}(\nu_{i}, \tilde{\nu}_{i})$}  & \txt{$\dot{\mathbf{E}}_{\mathbf{V},\mathbf{W},i}$\\ $\times$\\ $\mathbf{Fl}(\nu_{i},\nu''_{i},\tilde{\nu}_{i})$} \ar[l]_{q_{1}} \ar@2{-}[r]  & \txt{$\dot{\mathbf{E}}_{\mathbf{V},\mathbf{W},i}$\\  $\times$\\  $\mathbf{Fl}(\nu_{i},\nu''_{i},\tilde{\nu}_{i})$}  \ar[r]^{q_{2}}  & \txt{$\dot{\mathbf{E}}_{\mathbf{V}'',\mathbf{W},i}$\\ $\times$\\$\mathbf{Gr}(\nu''_{i}, \tilde{\nu}_{i}$)}
}
\]
where
\begin{equation*}
	q_{1}((\dot{x},\mathbf{S}_{1}\subset \mathbf{S}_{2} ))=(\dot{x},\mathbf{S}_{1}), 
\end{equation*}
\begin{equation*}
	q_{2}((\dot{x},\mathbf{S}_{1}\subset \mathbf{S}_{2} ))=(\dot{x},\mathbf{S}_{2}),
\end{equation*}
\begin{equation*}
	\phi_{1}=\phi_{\mathbf{V},i},\  \phi_{2}=\tilde{\phi}\circ \tilde{p}_{2},\ \phi_{3}=\tilde{\phi},\ \phi_{4}=\phi_{\mathbf{V}'',i}
\end{equation*}
and $\tilde{\phi}: \mathbf{E}^{'',0}_{\mathbf{V}'',\mathbf{W},i} \rightarrow \dot{\mathbf{E}}_{\mathbf{V},\mathbf{W},i} \times \mathbf{Fl}(\nu_{i},\nu''_{i},\tilde{\nu}_{i})$ is defined by
\begin{equation*}
	\tilde{\phi}((x,\tilde{\mathbf{W}}))=(\dot{x},{\rm{Im}}  ( \bigoplus \limits_{h \in \hat{\Omega}, h'=i}  x_{h})|_{\tilde{\mathbf{W}}}) \subset {\rm{Im}}  ( \bigoplus \limits_{h \in \hat{\Omega}, h'=i} x_{h} )  \subset (\bigoplus\limits_{h\in\Omega,h'=i}\mathbf{V}_{h''})\oplus \mathbf{W}_{\hat{i}} ).
\end{equation*}
Notice that the right square is Cartesian, and we have 
\begin{equation}\label{equation 2}
	(q_{2})_{!}(q_{1})^{\ast}(\phi_{1})_{\flat} =(\phi_{4})_{\flat}(\tilde{p}_{3})_{!}(\tilde{p}_{2})_{\flat}(\tilde{p}_{1})^{\ast}.
\end{equation}

\textbf{Case (II)} $i \neq j$ and $i,j$ are connected by some edge. 
Let
$Z_{1}$ be the variety consisting of quadruples $(\dot{x},\dot{\mathbf{S}} \subset \dot{\mathbf{V}}'',\dot{\rho},\tilde{\mathbf{V}}''\subset  \bigoplus\limits_{h\in \Omega,h'=i} \mathbf{S}_{h''} \oplus \mathbf{W}_{\hat{i}})$ such that $\dot{\mathbf{S}}$ is a $\dot{x}$-stable subspace of $\dot{\mathbf{V}}''$ with $|\dot{\mathbf{S}}|=|\dot{\mathbf{V}}|$, $\dot{\rho}: \dot{\mathbf{S}} \rightarrow \dot{\mathbf{V}}$ is a linear isomorphism and $\tilde{\mathbf{V}}''$ is a subspace of dimension $\nu_{i}$. And let $Z_{2}$ be the variety consisting of triples $(\dot{x},\dot{\mathbf{S}} \subset \dot{\mathbf{V}}'',\tilde{\mathbf{V}}''\subset  \bigoplus\limits_{h\in \Omega,h'=i} \mathbf{S}_{h''} \oplus \mathbf{W}_{\hat{i}} )$ such that $\dot{\mathbf{S}}$ is a $\dot{x}$-stable subspace of $\dot{\mathbf{V}}''$ with $|\dot{\mathbf{S}}|=|\dot{\mathbf{V}}|$ and $\tilde{\mathbf{V}}''$ is subspace of dimension $\nu_{i}$. Consider the following commutative diagram
\[
\xymatrix{
	\mathbf{E}^{0}_{\mathbf{V},\mathbf{W},i}\ar[d]_{\phi_{1}} & \mathbf{E}^{',0}_{\mathbf{V}'',\mathbf{W},i} \ar[l]_{\tilde{p}_{1}} \ar[r]^{\tilde{p}_{2}} \ar[d]_{\phi_{2}}& \mathbf{E}^{'',0}_{\mathbf{V}'',\mathbf{W},i}  \ar[r]^{\tilde{p}_{3}} \ar[d]_{\phi_{3}} & \mathbf{E}^{0}_{\mathbf{V}'',\mathbf{W},i}  \ar[d]_{\phi_{4}} \\
	\txt{$\dot{\mathbf{E}}_{\mathbf{V},\mathbf{W},i}$\\ $\times$\\$\mathbf{Gr}(\nu_{i}, \tilde{\nu}_{i})$}  & Z_{1} \ar[l]_-{q_{1}} \ar[r]^{q_{2}}  & Z_{2} \ar[r]^-{q_{3}}  & \txt{$\dot{\mathbf{E}}_{\mathbf{V}'',\mathbf{W},i}$\\ $\times$\\$\mathbf{Gr}(\nu_{i}, \tilde{\nu}_{i}+a_{i,j})$} 
}
\]
where $a_{i,j}$ is the number of edges connecting $i$ and $j$ and
\begin{align*}
	&q_{1}(\dot{x},\dot{\mathbf{S}},\dot{\rho},\tilde{\mathbf{V}}'')=((\dot{\rho})_{\ast}(\dot{x}|_{\dot{\mathbf{S}}\oplus \dot{\mathbf{W}}}), \dot{\rho}(\tilde{\mathbf{V}}'')),\\
	&q_{2}(\dot{x},\dot{\mathbf{S}},\dot{\rho},\tilde{\mathbf{V}}'')=(\dot{x},\dot{\mathbf{S}},\tilde{\mathbf{V}}''),\\
	&q_{3}((\dot{x},\dot{\mathbf{S}},\tilde{\mathbf{V}}''))=(\dot{x},\tilde{\mathbf{V}}'')\\
	&\phi_{1}=\phi_{\mathbf{V},i},\\
	&\phi_{2}((x,\mathbf{S},\rho))=(\dot{x},\dot{\mathbf{S}},\dot{\rho},{\rm{Im}}  ( \bigoplus \limits_{h \in \hat{\Omega}, h'=i} x_{h}) ),\\
	&\phi_{3}((x,\mathbf{S}))=(\dot{x},\dot{\mathbf{S}},{\rm{Im}}  ( \bigoplus \limits_{h \in \hat{\Omega}, h'=i} x_{h})),\\
	&\phi_{4}=\phi_{\mathbf{V}'',i}.
\end{align*}
Notice that the right square is Cartesian, we have 
\begin{equation}
	(q_{3})_{!}(q_{2})_{\flat}(q_{1})^{\ast}(\phi_{1})_{\flat} =(\phi_{4})_{\flat}(\tilde{p}_{3})_{!}(\tilde{p}_{2})_{\flat}(\tilde{p}_{1})^{\ast}.
\end{equation}

\textbf{Case (III)} $i \neq j$ and there is no edge connecting $i,j$. Let
$Z_{1}$ be the variety consisting of quadruples $(\dot{x},\dot{\mathbf{S}} \subset \dot{\mathbf{V}}'',\dot{\rho},\tilde{\mathbf{V}}''\subset  \bigoplus\limits_{h\in \Omega,h'=i} \mathbf{S}_{h''} \oplus \mathbf{W}_{\hat{i}})$ such that $\dot{\mathbf{S}}$ is a $\dot{x}$-stable subspace of $\dot{\mathbf{V}}''$,  $\dot{\rho}: \dot{\mathbf{S}} \rightarrow \dot{\mathbf{V}}$ is a linear isomorphism and $\tilde{\mathbf{V}}''$ is a subspace of dimension $\nu_{i}$. And let $Z_{2}$ be the variety which consists of $(\dot{x},\dot{\mathbf{S}} \subset \dot{\mathbf{V}}'',\tilde{\mathbf{V}}''\subset  \bigoplus\limits_{h\in \Omega,h'=i} \mathbf{S}_{h''} \oplus \mathbf{W}_{\hat{i}} )$ such that $\dot{\mathbf{S}}$ is a $\dot{x}$-stable subspace of $\dot{\mathbf{V}}''$ with $|\dot{\mathbf{S}}|=|\dot{\mathbf{V}}|$ and $\tilde{\mathbf{V}}''$ is a subspace of dimension $\nu_{i}$. Consider the following commutative diagram
\[
\xymatrix{
	\mathbf{E}^{0}_{\mathbf{V},\mathbf{W},i}\ar[d]_{\phi_{1}} & \mathbf{E}^{',0}_{\mathbf{V}'',\mathbf{W},i} \ar[l]_{\tilde{p}_{1}} \ar[r]^{\tilde{p}_{2}} \ar[d]_{\phi_{2}}& \mathbf{E}^{'',0}_{\mathbf{V}'',\mathbf{W},i}  \ar[r]^{\tilde{p}_{3}} \ar[d]_{\phi_{3}} & \mathbf{E}^{0}_{\mathbf{V}'',\mathbf{W},i}  \ar[d]_{\phi_{4}} \\
	\dot{\mathbf{E}}_{\mathbf{V},\mathbf{W},i} \times \mathbf{Gr}(\nu_{i}, \tilde{\nu}_{i})  & Z_{1} \ar[l]_-{q_{1}} \ar[r]^{q_{2}}  & Z_{2} \ar[r]^-{q_{3}}  & \dot{\mathbf{E}}_{\mathbf{V}'',\mathbf{W},i} \times \mathbf{Gr}(\nu_{i}, \tilde{\nu}_{i})
}
\]
where $q_{1},q_{2},q_{3},\phi_1,\phi_2,\phi_3,\phi_4$ are given by the same formulas in \textbf{Case (II)} above.
Notice that the right square is Cartesian, we have 
\begin{equation}\label{Eq4}
	(q_{3})_{!}(q_{2})_{\flat}(q_{1})^{\ast}(\phi_{1})_{\flat} =(\phi_{4})_{\flat}(\tilde{p}_{3})_{!}(\tilde{p}_{2})_{\flat}(\tilde{p}_{1})^{\ast}.
\end{equation}

In conclusion,  the functor $\mathcal{F}_{j}:\mathcal{D}^{b}_{G_{\mathbf{V}}}(\mathbf{E}_{\mathbf{V},\mathbf{W},\hat{\Omega}})/\mathcal{N}_{\mathbf{V},i} \rightarrow \mathcal{D}^{b}_{G_{\mathbf{V}''}}(\mathbf{E}_{\mathbf{V}'',\mathbf{W},\hat{\Omega}})/\mathcal{N}_{\mathbf{V}'',i}$ can be described by
\begin{equation} \label{functorF}
	\mathcal{F}_{j}\cong \begin{cases}(j_{\mathbf{V}'',i})_{!} (\phi_{\mathbf{V}'',i})^{\ast} (q_{2})_{!}(q_{1})^{\ast} (\phi_{\mathbf{V},i})_{\flat}(j_{\mathbf{V},i})^{\ast}[\tilde{\nu}_{i}+\nu_{i}], i=j,\\(j_{\mathbf{V}'',i})_{!} (\phi_{\mathbf{V}'',i})^{\ast} (q_{3})_{!} (q_{2})_{\flat}(q_{1})^{\ast} (\phi_{\mathbf{V},i})_{\flat}(j_{\mathbf{V},i})^{\ast}[\tilde{\nu}_{i}+\nu_{i}], i\neq j.\end{cases}
\end{equation}

We also introduce  functors $\mathcal{F}^{(n),\vee}_{j}$ and $\mathcal{E}_{i,r,a}$ which will be used in the next section.

\begin{definition}
	For $j \in I$ and $j \neq i$, we take graded spaces $\mathbf{V},\mathbf{V}'$  and $\mathbf{V}''$ such that  $|\mathbf{V}'|-nj=|\mathbf{V}|,|\mathbf{V}''|=nj$ and define the functor $\mathcal{F}_{j}^{(n),\vee}:\mathcal{Q}_{\mathbf{V},\mathbf{W}}\rightarrow \mathcal{Q}_{\mathbf{V}',\mathbf{W}}/\mathcal{N}_{\mathbf{V}',i}$ (or $\mathcal{F}_{j}^{(n),\vee}:\mathcal{D}^{b}_{G_{\mathbf{V}}}(\mathbf{E}_{\mathbf{V},\mathbf{W},\hat{\Omega} }) \rightarrow \mathcal{D}^{b}_{G_{\mathbf{V}'}}(\mathbf{E}_{\mathbf{V}',\mathbf{W},\hat{\Omega}})/\mathcal{N}_{\mathbf{V}',i} $ ) via
	\begin{equation*}
		\mathcal{F}_{j}^{(n),\vee}= \mathbf{Ind}^{\mathbf{V}'\oplus\mathbf{W}}_{\mathbf{V}\oplus \mathbf{W},\mathbf{V}''}(- \boxtimes\overline{ \mathbb{Q}}_{l} ),
	\end{equation*}
	where $\overline{\mathbb{Q}}_{l}$ is the constant sheaf of $\mathbf{E}_{\mathbf{V}'',0,\hat{\Omega}}$. 
\end{definition}

For  $a=0,1$ and $ r \leqslant \nu_{i}$, we take graded spaces $\mathbf{V},\mathbf{V}'$   such that  $|\mathbf{V}'|+i=|\mathbf{V}|$. Fix a decomposition $\mathbf{V}=\mathbf{V}_{r}\oplus\mathbf{S}_{r}$ of $\mathbf{V}$   such that $|\mathbf{S}_{r}|=ri$ and a decomposition $\mathbf{V}'=\mathbf{V}'_{r-a}\oplus\mathbf{S}'_{r-a}$ of $\mathbf{V}'$   such that $|\mathbf{S}'_{r-a}|=(r-a)i$. Since the subset $'F_{r}$ of $\mathbf{E}^{r}_{\mathbf{V},\mathbf{W},i}$ consisting of $x$ such that $\mathbf{S}_{r}$ is $x$-stable is naturally isomorphic to $\mathbf{E}^{0}_{\mathbf{V}_{r},\mathbf{W},i}$, there is a closed embedding  $\iota_{\mathbf{V},r}:\mathbf{E}^{0}_{\mathbf{V}_{r},\mathbf{W},i} \rightarrow \mathbf{E}^{r}_{\mathbf{V},\mathbf{W},i}$.  Let $Q_{r}$ be the stablizer of $\mathbf{S}_{r}$ in $G_{\mathbf{V}}$ and $U_{r}$ be its unipotent radical, then there is an isomorphism
$\mathbf{G}_{\mathbf{V}} \times_{Q_{r}} \mathbf{E}^{0}_{\mathbf{V}_{r},\mathbf{W},i} \cong \mathbf{E}^{r}_{\mathbf{V},\mathbf{W},i}$. By \cite[Theorem 6.5.10]{Achar}, there is an induction equivalence $$(\iota_{\mathbf{V},r})^{\ast}[-{\rm{dim}}(G_{\mathbf{V}}/ Q_{r}) ] \circ \mathbf{For}^{G_{\mathbf{V}}}_{Q_{r}}:\mathcal{D}^{b}_{G_{\mathbf{V}}}(\mathbf{E}^{r}_{\mathbf{V},\mathbf{W},i} )\rightarrow \mathcal{D}^{b}_{Q_{r}}(\mathbf{E}^{0}_{\mathbf{V}_{r},\mathbf{W},i} ) \cong \mathcal{D}^{b}_{G_{\mathbf{V}_{r}}}(\mathbf{E}^{0}_{\mathbf{V}_{r},\mathbf{W},i} ),$$
where $\mathbf{For}^{G_{\mathbf{V}}}_{Q_{r}}$ is the forgetful functor, and the last equivalence holds by\cite[Theorem 6.6.16]{Achar}, since $Q_{r}/U_{r} \cong G_{\mathbf{V}_{r}} \times G_{\mathbf{S}_{r}}$ and  $U_{r},G_{\mathbf{S}_{r}}$ act trivially on $'F_{r}\cong \mathbf{E}^{0}_{\mathbf{V}_{r},\mathbf{W},i} $. We denote this equivalence by $(\iota_{\mathbf{V},r})^{_\clubsuit}$ and its quasi-inverse by $(\iota_{\mathbf{V},r})_{_\clubsuit}$. Similarly, there is also an induction equivalence $$(\iota_{\mathbf{V}',r-a})^{\ast}[-{\rm{dim}}(G_{\mathbf{V}'}/ Q'_{r-a}) ] \circ \mathbf{For}^{G_{\mathbf{V}'}}_{Q'_{r-a}}:\mathcal{D}^{b}_{G_{\mathbf{V}'}}(\mathbf{E}^{r-a}_{\mathbf{V}',\mathbf{W},i} )\rightarrow  \mathcal{D}^{b}_{G_{\mathbf{V}'_{r-a}}}(\mathbf{E}^{0}_{\mathbf{V}'_{r-a},\mathbf{W},i} ).$$

Now we consider the following diagram
\[
\xymatrix{
	\mathbf{E}_{\mathbf{V},\mathbf{W},\hat{\Omega}}
	&
	& \mathbf{E}_{\mathbf{V}',\mathbf{W},\hat{\Omega}}  \\
	\mathbf{E}^{r}_{\mathbf{V},\mathbf{W},i}  \ar[u]^{j_{\mathbf{V},r}}
	&
	& \mathbf{E}^{r-a}_{\mathbf{V}',\mathbf{W},i}  \ar[u]_{j_{\mathbf{V}',r-a}} \\
	\mathbf{E}^{0}_{\mathbf{V}_{r},\mathbf{W},i} \ar[d]_{\phi_{\mathbf{V}_{r},i}} \ar[u]^{\iota_{\mathbf{V},r}}
	&
	& \mathbf{E}^{0}_{\mathbf{V}'_{r-a},\mathbf{W},i} \ar[d]^{\phi_{\mathbf{V}'_{r-a},i}} \ar[u]_{\iota_{\mathbf{V}',r-a}}\\
	\txt{$\dot{\mathbf{E}}_{\mathbf{V},\mathbf{W},i}$\\ $\times$\\ $\mathbf{Gr}(\nu_{i}-r, \tilde{\nu}_{i})$}
	&     \txt{$\dot{\mathbf{E}}_{\mathbf{V},\mathbf{W},i}$\\ $\times$\\ $\mathbf{Fl}(\nu_{i}-1-r+a,\nu_{i}-r,\tilde{\nu}_{i})$} \ar[r]^-{q_{2,r,a}} \ar[l]_-{q_{1,r,a}}
	&\txt{$\dot{\mathbf{E}}_{\mathbf{V},\mathbf{W},i}$\\ $\times$\\ $\mathbf{Gr}(\nu_{i}-1-r+a, \tilde{\nu}_{i})$},
}
\]
where $j_{\mathbf{V},r}$ is the natural locally closed embedding and $q_{1,r,a}$, $q_{2,r,a}$ are natural projections. In particular, $j_{\mathbf{V},r}$ with $r=0$ is equal to  $j_{\mathbf{V},i}$
\begin{definition}
	For $r \leqslant \nu_{i}$ and $a=0,1$,  take graded spaces $\mathbf{V},\mathbf{V}'$   such that  $|\mathbf{V}'|+i=|\mathbf{V}|$. Define the functor $\mathcal{E}_{i,r,a}: \mathcal{D}^{b}_{G_{\mathbf{V}}}(	\mathbf{E}_{\mathbf{V},\mathbf{W},\hat{\Omega}} ) \rightarrow \mathcal{D}^{b}_{G_{\mathbf{V}'}}(	\mathbf{E}_{\mathbf{V}',\mathbf{W},\hat{\Omega}} )$ by
	$$ \mathcal{E}_{i,r,a}= (j_{\mathbf{V}',r-a})_{!} (\iota_{\mathbf{V}',r-a})_{_\clubsuit} (\phi_{\mathbf{V}'_{r-a},i})^{\ast}(q_{2,r,a})_{!}(q_{1,r,a})^{\ast}(\phi_{\mathbf{V}_{r},i})_{\flat} (\iota_{\mathbf{V},r})^{_\clubsuit}(j_{\mathbf{V},r})^{\ast} ,$$
	\begin{equation*}
		\begin{split}
			&\mathcal{D}^{b}_{G_{\mathbf{V}}}(\mathbf{E}_{\mathbf{V},\mathbf{W},\hat{\Omega}}) \xrightarrow{(j_{\mathbf{V},r})^{\ast}} \mathcal{D}^{b}_{G_{\mathbf{V}}}(\mathbf{E}^{r}_{\mathbf{V},\mathbf{W},i})\xrightarrow{(\iota_{\mathbf{V},r})^{\clubsuit}}  \mathcal{D}^{b}_{G_{\mathbf{V}_{r}}}(\mathbf{E}^{0}_{\mathbf{V}_{r},\mathbf{W},i}) \xrightarrow{(\phi_{\mathbf{V}_{r},i})_\flat} \\ &\mathcal{D}^{b}_{G_{\dot{\mathbf{V}}}}(\dot{\mathbf{E}}_{\mathbf{V},\mathbf{W},i} \times \mathbf{Gr}(\nu_{i}-r, \tilde{\nu}_{i})) \xrightarrow{(q_{2,r,a})_{!}(q_{1,r,a})^{\ast}}  \mathcal{D}^{b}_{G_{\dot{\mathbf{V}}}}(\dot{\mathbf{E}}_{\mathbf{V},\mathbf{W},i} \times \mathbf{Gr}(\nu_{i}-1-r+a, \tilde{\nu}_{i})) \\& \xrightarrow{(\phi_{\mathbf{V}'_{r-a},i})^{\ast}} 
			\mathcal{D}^{b}_{G_{\mathbf{V}'_{r-a}}}(\mathbf{E}^{0}_{\mathbf{V}'_{r-a},\mathbf{W},i}) \xrightarrow{(\iota_{\mathbf{V}',r-a})_{\clubsuit}} \mathcal{D}^{b}_{G_{\mathbf{V}'}}(\mathbf{E}^{r-a}_{\mathbf{V}',\mathbf{W},i}) \xrightarrow{(j_{\mathbf{V}',r-a})_{!}} \mathcal{D}^{b}_{G_{\mathbf{V}'}}(\mathbf{E}_{\mathbf{V}',\mathbf{W},\hat{\Omega}}).
		\end{split}
	\end{equation*}
	
\end{definition}

\subsection{Commutative relations in localizations at $i$}
At the beginning of this section, we recall two lemmas in algebraic geometry.

\begin{lemma}\label{projection formula}
	For any morphism $f:X\rightarrow Y$ and any complex $A$ on $Y$, we have $f_{!}f^{\ast}A \cong f_{!}\overline{\mathbb{Q}}_{l} \otimes A$.
\end{lemma}	
\begin{proof}
	By the projection formula, see \cite[Theorem 1.4.9]{Achar}, we have 
	$$f_{!}\overline{\mathbb{Q}}_{l} \otimes A\cong f_{!}(\overline{\mathbb{Q}}_{l}\otimes f^*A)\cong f_{!}f^{\ast}(A),$$ 
	as desired.
\end{proof}

\begin{lemma}\cite[Section 8.1.6]{MR1227098}\label{Lusztig BBD}
	If $X =\coprod\limits_{n} X_{n}$ is a partition such that for any $n$, $X^{\leqslant n}=\coprod\limits_{m \leqslant n }X_{m} $ is closed and the restriction $f_{n}$ of a morphism $f:X \rightarrow Y$ on $X_{n}$ can be decomposed as
	\begin{equation*}
		X_{n} \xrightarrow{g_{n}} Z_{n} \xrightarrow{h_{n}} Y
	\end{equation*}	
	where $Z_{n}$ is smooth, $g_{n}$ is a vector bundle of rank $d_{n}$ and $h_{n}$ is proper, then
	\begin{equation*}
		f_{!}(\overline{\mathbb{Q}}_{l}|_{X}) \cong \bigoplus\limits_{n} (f_{n})_{!} (\overline{\mathbb{Q}}_{l}|_{X_n})\cong \bigoplus\limits_{n} (h_{n})_{!} (\overline{\mathbb{Q}}_{l}|_{Z_n})[-2d_{n}].
	\end{equation*}	
\end{lemma}

\begin{lemma} \label{lemma c1}
	For any graded space $\mathbf{V},\mathbf{V}'$ such that $|\mathbf{V}'|+(n-1)i=|\mathbf{V}| $, there is an isomorphism as functors  $\mathcal{D}^{b}_{G_{\mathbf{V}}}(\mathbf{E}_{\mathbf{V},\mathbf{W},\hat{\Omega}})/\mathcal{N}_{\mathbf{V},i} \rightarrow \mathcal{D}^{b}_{G_{\mathbf{V}'}}(\mathbf{E}_{\mathbf{V}',\mathbf{W},\hat{\Omega}})/\mathcal{N}_{\mathbf{V}',i}$,
	\begin{equation*}
		\mathcal{E}^{(n)}_{i}\mathcal{F}_{i} \oplus \bigoplus\limits_{0\leqslant m \leqslant N-1} \mathcal{E}^{(n-1)}_{i}[N-1-2m] \cong \mathcal{F}_{i}\mathcal{E}^{(n)}_{i} \oplus \bigoplus\limits_{0\leqslant m \leqslant -N-1} \mathcal{E}^{(n-1)}_{i}[-2m-N-1],
	\end{equation*}
	where $N=2\nu_{i}-\tilde{\nu}_{i}-n+1$. More precisely,
	\begin{align*}
		&\mathcal{E}^{(n)}_{i}\mathcal{F}_{i}\cong \mathcal{F}_{i}\mathcal{E}^{(n)}_{i}, &\textrm{if}\ N=0;\\
		&\mathcal{E}^{(n)}_{i}\mathcal{F}_{i}\oplus\bigoplus_{0\leqslant m\leqslant N-1}\mathcal{E}^{(n-1)}_{i}[N-1-2m]\cong \mathcal{F}_{i}\mathcal{E}^{(n)}_{i},&\textrm{if}\  N\geqslant 1;\\
		&\mathcal{E}^{(n)}_{i}\mathcal{F}_{i}\cong \mathcal{F}_{i}\mathcal{E}^{(n)}_{i} \oplus \bigoplus\limits_{0\leqslant m \leqslant -N-1} \mathcal{E}^{(n-1)}_{i}[-2m-N-1], &\textrm{if}\ N\leqslant -1.
	\end{align*}
\end{lemma}
\begin{proof}
	On the one hand, take graded space $\mathbf{V}''$  such that $|\mathbf{V}|=|\mathbf{V}''|-i $ and consider the diagrams
	\[
	\xymatrix{
		\mathbf{E}_{\mathbf{V},\mathbf{W},\hat{\Omega}} & \mathbf{E}^{'}_{\mathbf{V}'',\mathbf{W},\hat{\Omega}}\ar[l]_{p_{1}} \ar[r]^{p_{2}}& \mathbf{E}^{''}_{\mathbf{V}'',\mathbf{W},\hat{\Omega}} \ar[r]^{p_{3}}& \mathbf{E}_{\mathbf{V}'',\mathbf{W},\hat{\Omega}}\\
		\mathbf{E}^{0}_{\mathbf{V},\mathbf{W},i} \ar[d]_{\phi_{1}}  \ar[u]^{j_{1}}& \mathbf{E}^{',0}_{\mathbf{V}'',\mathbf{W},i} \ar[l]_{\tilde{p}_{1}} \ar[r]^{\tilde{p}_{2}} \ar[d]_{\phi_{2}}\ar[u]^{j_{2}}& \mathbf{E}^{'',0}_{\mathbf{V}'',\mathbf{W},i}  \ar[r]^{\tilde{p}_{3}}  \ar[d]_{\phi_{3}}\ar[u]^{j_{3}}& \mathbf{E}^{0}_{\mathbf{V}'',\mathbf{W},i} \ar[d]_{\phi_{4}} \ar[u]^{j_{4}}\\
		\txt{$\dot{\mathbf{E}}_{\mathbf{V},\mathbf{W},i}$\\ $\times$\\ $\mathbf{Gr}(\nu_{i}, \tilde{\nu}_{i})$}  & \txt{$\dot{\mathbf{E}}_{\mathbf{V},\mathbf{W},i}$\\ $\times$\\ $\mathbf{Fl}(\nu_{i},\nu''_{i},\tilde{\nu}_{i})$} \ar[l]_{q'_{1}} \ar@2{-}[r]  & \txt{$\dot{\mathbf{E}}_{\mathbf{V},\mathbf{W},i}$\\ $\times$\\ $\mathbf{Fl}(\nu_{i},\nu''_{i},\tilde{\nu}_{i})$}  \ar[r]^{q_{2}'}  & \txt{$\dot{\mathbf{E}}_{\mathbf{V}'',\mathbf{W},i}$\\ $\times$\\ $\mathbf{Gr}(\nu''_{i}, \tilde{\nu}_{i}),$}
	}
	\]
	
	\[
	\xymatrix{
		\mathbf{E}_{\mathbf{V}'',\mathbf{W},\hat{\Omega}}
		&
		& \mathbf{E}_{\mathbf{V}',\mathbf{W},\hat{\Omega}} \\	
		\mathbf{E}^{0}_{\mathbf{V}'',\mathbf{W},i}  \ar[d]_{\phi_{\mathbf{V}'',i}} \ar[u]^{j_{\mathbf{V}'',i}}
		&
		& \mathbf{E}^{0}_{\mathbf{V}',\mathbf{W},i}\ar[d]^{\phi_{\mathbf{V}',i}} \ar[u]_{j_{\mathbf{V}',i}} \\
		\txt{$\dot{\mathbf{E}}_{\mathbf{V}'',\mathbf{W},i}$\\ $\times$\\$\mathbf{Gr}(\nu''_{i}, \tilde{\nu}_{i})$}
		& \txt{$\dot{\mathbf{E}}_{\mathbf{V}'',\mathbf{W},i}$\\ $\times$\\$\mathbf{Fl}(\nu'_{i},\nu''_{i},\tilde{\nu}_{i})$}
		\ar[r]^{q_{2}} \ar[l]_{q_{1}}
		& \txt{$\dot{\mathbf{E}}_{\mathbf{V}',\mathbf{W},i}$\\ $\times$\\$\mathbf{Gr}(\nu'_{i}, \tilde{\nu}_{i})$}.
	}
	\]
	By base change, we have
	\begin{equation*}
		\begin{split}
			\mathcal{E}^{(n)}_{i}\mathcal{F}_{i}
			\cong &(j_{\mathbf{V}',i})_{!} (\phi_{\mathbf{V}',i})^{\ast} (q_{2})_{!}(q_{1})^{\ast} (\phi_{\mathbf{V}'',i})_{\flat}(j_{\mathbf{V}'',i})^{\ast}\\
			& (j_{4})_{!}(\phi_{4})^{\ast}(q'_{2})_{!}(q'_{1})^{\ast}(\phi_{1})_{\flat}(j_{1})^{\ast}[\tilde{\nu}_{i}+\nu_i-n\nu_{i}-n]\\
			\cong &(j_{\mathbf{V}',i})_{!} (\phi_{\mathbf{V}',i})^{\ast} (q_{2})_{!}(q_{1})^{\ast} (q'_{2})_{!}(q'_{1})^{\ast}(\phi_{1})_{\flat}(j_{1})^{\ast}[\tilde{\nu}_{i}+\nu_i-n\nu_{i}-n]\\
			\cong &(j_{\mathbf{V}',i})_{!} (\phi_{\mathbf{V}',i})^{\ast} (q_{2})_{!}(q_{1})^{\ast} (q'_{2})_{!}(q'_{1})^{\ast}(\phi_{\mathbf{V},i})_{\flat}(j_{\mathbf{V},i})^{\ast}[\tilde{\nu}_{i}+\nu_i-n\nu_{i}-n],
		\end{split}
	\end{equation*}
	where the second isomorphism follows from $j_{4}= j_{\mathbf{V}'',i}, \phi_{4}=\phi_{\mathbf{V}'',i}$.
	
	On the other hand, take a graded space $\mathbf{V}'''$ such that $|\mathbf{V}|=|\mathbf{V}'''|+ni$ and consider the diagrams

	\[
	\xymatrix{
		\mathbf{E}_{\mathbf{V},\mathbf{W},\hat{\Omega}} 
		&
		& \mathbf{E}_{\mathbf{V}''',\mathbf{W},\hat{\Omega}} \\	
		\mathbf{E}^{0}_{\mathbf{V},\mathbf{W},i} \ar[d]_{\phi_{\mathbf{V},i}} \ar[u]^{j_{\mathbf{V},i}}
		&
		& \mathbf{E}^{0}_{\mathbf{V}''',\mathbf{W},i} \ar[d]^{\phi_{\mathbf{V}''',i}} \ar[u]_{j_{\mathbf{V}''',i}} \\
		\txt{$\dot{\mathbf{E}}_{\mathbf{V},\mathbf{W},i}$\\ $\times$\\$\mathbf{Gr}(\nu_{i}, \tilde{\nu}_{i})$}
		&\txt{$\dot{\mathbf{E}}_{\mathbf{V},\mathbf{W},i}$\\ $\times$\\$\mathbf{Fl}(\nu'''_{i},\nu_{i},\tilde{\nu}_{i})$}
		\ar[r]^{\tilde{q}_{2}} \ar[l]_{\tilde{q}_{1}}
		& \txt{$\dot{\mathbf{E}}_{\mathbf{V}''',\mathbf{W},i}$\\ $\times$\\$\mathbf{Gr}(\nu'''_{i}, \tilde{\nu}_{i})$},
	}
	\]
	
	\[
	\xymatrix{
		\mathbf{E}_{\mathbf{V}''',\mathbf{W},\hat{\Omega}} & \mathbf{E}^{'}_{\mathbf{V}',\mathbf{W},\hat{\Omega}}\ar[l]_{p_{1}} \ar[r]^{p_{2}}& \mathbf{E}^{''}_{\mathbf{V}',\mathbf{W},\hat{\Omega}} \ar[r]^{p_{3}}& \mathbf{E}_{\mathbf{V}',\mathbf{W},\hat{\Omega}} \\
		\mathbf{E}^{0}_{\mathbf{V}''',\mathbf{W},i} \ar[d]_{\phi_{1}}  \ar[u]^{j_{1}}& \mathbf{E}^{',0}_{\mathbf{V}',\mathbf{W},i} \ar[l]_{\tilde{p}_{1}} \ar[r]^{\tilde{p}_{2}} \ar[d]_{\phi_{2}}\ar[u]^{j_{2}}& \mathbf{E}^{'',0}_{\mathbf{V}',\mathbf{W},i}  \ar[r]^{\tilde{p}_{3}}  \ar[d]_{\phi_{3}}\ar[u]^{j_{3}}& \mathbf{E}^{0}_{\mathbf{V}',\mathbf{W},i}\ar[d]_{\phi_{4}} \ar[u]^{j_{4}}\\
		\txt{$\dot{\mathbf{E}}_{\mathbf{V}''',\mathbf{W},i}$\\$\times$\\ $\mathbf{Gr}(\nu'''_{i}, \tilde{\nu}_{i})$}  & \txt{$\dot{\mathbf{E}}_{\mathbf{V}''',\mathbf{W},i}$\\ $\times$\\ $\mathbf{Fl}(\nu'''_{i},\nu'_{i},\tilde{\nu}_{i})$} \ar[l]_{\tilde{q}'_{1}} \ar@2{-}[r]  & \txt{$\dot{\mathbf{E}}_{\mathbf{V}',\mathbf{W},i}$\\ $\times$\\  $\mathbf{Fl}(\nu'''_{i},\nu'_{i},\tilde{\nu}_{i})$}  \ar[r]^{\tilde{q}_{2}'}  & \txt{$\dot{\mathbf{E}}_{\mathbf{V}',\mathbf{W},i}$\\ $\times$\\ $\mathbf{Gr}(\nu'_{i}, \tilde{\nu}_{i})$.}
	}
	\]

	Similarly, we have
	\begin{equation*}
		\begin{split}
			\mathcal{F}_{i}\mathcal{E}^{(n)}_{i}=& (j_{4})_{!}(\phi_{4})^{\ast}(\tilde{q}'_{2})_{!}(\tilde{q}'_{1})^{\ast}(\phi_{1})_{\flat}(j_{1})^{\ast}\\
			& (j_{\mathbf{V}''',i})_{!} (\phi_{\mathbf{V}''',i})^{\ast} (\tilde{q}_{2})_{!}(\tilde{q}_{1})^{\ast} (\phi_{\mathbf{V},i})_{\flat}(j_{\mathbf{V},i})^{\ast}[\tilde{\nu}_{i}+\nu_i-n\nu_{i}-n]\\
			\cong&(j_{4})_{!}(\phi_{4})^{\ast}(\tilde{q}'_{2})_{!}(\tilde{q}'_{1})^{\ast} (\tilde{q}_{2})_{!}(\tilde{q}_{1})^{\ast} (\phi_{\mathbf{V},i})_{\flat}(j_{\mathbf{V},i})^{\ast}[\tilde{\nu}_{i}+\nu_i-n\nu_{i}-n]\\
			\cong&(j_{\mathbf{V}',i})_{!}(\phi_{\mathbf{V}',i})^{\ast}(\tilde{q}'_{2})_{!}(\tilde{q}'_{1})^{\ast} (\tilde{q}_{2})_{!}(\tilde{q}_{1})^{\ast} (\phi_{\mathbf{V},i})_{\flat}(j_{\mathbf{V},i})^{\ast}[\tilde{\nu}_{i}+\nu_i-n\nu_{i}-n].
		\end{split}
	\end{equation*}
	
	Since $\phi_{\flat}$ and $\phi^{\ast}$ are quasi inverse to each other, we only need to calculate the difference between $(q_{2})_{!}(q_{1})^{\ast} (q'_{2})_{!}(q'_{1})^{\ast}$ and $(\tilde{q}'_{2})_{!}(\tilde{q}'_{1})^{\ast} (\tilde{q}_{2})_{!}(\tilde{q}_{1})^{\ast}$. 
	Notice that  $$\nu_{i}=\nu'_{i}+n-1=\nu_{i}''-1=\nu_{i}'''+n,$$ $$\tilde{\nu}_{i}=\tilde{\nu}'_{i}=\tilde{\nu}''_{i}=\tilde{\nu}'''_{i},$$  $$\dot{\mathbf{E}}_{\mathbf{V},\mathbf{W},i}=\dot{\mathbf{E}}_{\mathbf{V}',\mathbf{W},i}=\dot{\mathbf{E}}_{\mathbf{V}'',\mathbf{W},i}=\dot{\mathbf{E}}_{\mathbf{V}''',\mathbf{W},i}.$$ Consider the following commutative diagram
	\[
	\xymatrix{
		Y_{0} \ar[rr]^{\tilde{\pi}} \ar[dd]_{\tilde{\pi}'} &  & \txt{$\dot{\mathbf{E}}_{\mathbf{V},\mathbf{W},i}$\\$\times$\\ $\mathbf{Gr}(\nu_{i}+1-n, \tilde{\nu}_{i})$} \\
		&  Y_{1} \ar[ul]^{r_{1}} \ar[r]^-{\pi_1} \ar[d]^{\pi'_1} &
		\txt{$\dot{\mathbf{E}}_{\mathbf{V},\mathbf{W},i}$\\$\times$\\ $\mathbf{Fl}(\nu_{i}+1-n,\nu_{i}+1,\tilde{\nu}_{i})$}
		\ar[u]_{q_{2}} \ar[d]^{q_{1}}\\
		\txt{$\dot{\mathbf{E}}_{\mathbf{V},\mathbf{W},i}$\\$\times$\\ $\mathbf{Gr}(\nu_{i}, \tilde{\nu}_{i})$}  & \txt{$\dot{\mathbf{E}}_{\mathbf{V},\mathbf{W},i}$\\$\times$\\ $\mathbf{Fl}(\nu_{i},\nu_{i}+1,\tilde{\nu}_{i})$}  \ar[l]_{q'_{1}} \ar[r]^-{q'_{2}} &  \txt{$\dot{\mathbf{E}}_{\mathbf{V},\mathbf{W},i}$\\$\times$\\ $\mathbf{Gr}(\nu_{i}+1, \tilde{\nu}_{i})$} ,
	}
	\]
	where the varieties
	\begin{equation*}
		Y_{0}=\dot{\mathbf{E}}_{\mathbf{V},\mathbf{W},i}\times \mathbf{Gr}(\nu_{i}, \tilde{\nu}_{i}) \times \mathbf{Gr}(\nu_{i}+1-n, \tilde{\nu}_{i}), 
	\end{equation*}
	and \begin{align*}
		Y_{1}=&\{(\dot{x}, \mathbf{V}^{1} \subseteq \mathbf{V}^{3} \subseteq  (\bigoplus\limits_{h'=i}\mathbf{V}_{h''})\oplus \mathbf{W}_{\hat{i}}), \mathbf{V}^{2} \subseteq \mathbf{V}^{3} \subseteq  (\bigoplus\limits_{h'=i}\mathbf{V}_{h''})\oplus \mathbf{W}_{\hat{i}}) \\
		&|~ \dot{x} \in   \dot{\mathbf{E}}_{\mathbf{V},\mathbf{W},i},~{\rm{dim}} \mathbf{V}^{1}=\nu_{i},~{\rm{dim}} \mathbf{V}^{2}=\nu_{i}-n+1,~{\rm{dim}} \mathbf{V}^{3}=\nu_{i}+1.\}
	\end{align*}
	is the fiber product of $q'_{2},q_{1}$, the morphisms $\pi_1,\pi'_1,\tilde{\pi},\tilde{\pi}'$ are natural projections and
	\begin{equation*}
		r_{1}((\dot{x},\mathbf{V}^{1},\mathbf{V}^{2},\mathbf{V}^{3}))=(\dot{x},\mathbf{V}^{1},\mathbf{V}^{2}). 
	\end{equation*}
	By base change, we have
	\begin{equation*}
		\begin{split}
			(q_{2})_{!}(q_{1})^{\ast} (q'_{2})_{!}(q'_{1})^{\ast}\cong& (q_{2})_{!}(\pi_1)_{!} (\pi'_1)^{\ast}(q'_{1})^{\ast}\\
			=&  (\tilde{\pi})_{!} (r_{1})_{!}(r_{1})^{\ast}(\tilde{\pi}')^{\ast}.
		\end{split}
	\end{equation*}
	
	Similarly, consider the commutative diagram
	\[
	\xymatrix{
		Y_{0} \ar[rr]^{\tilde{\pi}} \ar[dd]_{\tilde{\pi}'} &  &
		\txt{$\dot{\mathbf{E}}_{\mathbf{V},\mathbf{W},i}$\\$\times$\\ $\mathbf{Gr}(\nu_{i}-n+1, \tilde{\nu}_{i})$}
		\\
		&  Y_{2} \ar[ul]^{r_{2}} \ar[r]^-{\pi_2} \ar[d]^{\pi'_2} &
		\txt{$\dot{\mathbf{E}}_{\mathbf{V},\mathbf{W},i}$\\$\times$\\ $\mathbf{Fl}(\nu_{i}-n,\nu_{i}-n+1,\tilde{\nu}_{i})$}
		\ar[u]_{\tilde{q}_{2}'} \ar[d]^{\tilde{q}_{1}'}\\
		\txt{$\dot{\mathbf{E}}_{\mathbf{V},\mathbf{W},i}$\\$\times$\\ $\mathbf{Gr}(\nu_{i}, \tilde{\nu}_{i})$}
		&\txt{$\dot{\mathbf{E}}_{\mathbf{V},\mathbf{W},i}$\\$\times$\\ $\mathbf{Fl}(\nu_{i}-n,\nu_{i},\tilde{\nu}_{i})$} 
		\ar[l]_{\tilde{q}_{1}} \ar[r]^-{\tilde{q}_{2}} & 
		\txt{$\dot{\mathbf{E}}_{\mathbf{V},\mathbf{W},i}$\\$\times$\\ $\mathbf{Gr}(\nu_{i}-n, \tilde{\nu}_{i})$},
	}
	\]
	where the variety
	\begin{align*}
		Y_{2}=&\{(\dot{x},\mathbf{V}^{0}\subseteq \mathbf{V}^{1} \subseteq  (\bigoplus\limits_{h'=i}\mathbf{V}_{h''})\oplus \mathbf{W}_{\hat{i}},\mathbf{V}^{0}\subseteq \mathbf{V}^{2} \subseteq  (\bigoplus\limits_{h'=i}\mathbf{V}_{h''})\oplus \mathbf{W}_{\hat{i}})\\
		&|~ \dot{x} \in   \dot{\mathbf{E}}_{\mathbf{V},\mathbf{W},i}, {\rm{dim}} \mathbf{V}^{1}= \nu_{i}, {\rm{dim}} \mathbf{V}^{2} =\nu_{i}-n+1 ,{\rm{dim}} \mathbf{V}^{0}=\nu_{i}-n.\}
	\end{align*}
	is the fiber product of $\tilde{q}'_{1},\tilde{q}_{2}$, the morphisms $\pi_2,\pi'_2,\tilde{\pi},\tilde{\pi}'$ are natural projections and 
	\begin{equation*}
		r_{2}((\dot{x},\mathbf{V}^{0},\mathbf{V}^{1},\mathbf{V}^{2}))=(\dot{x},\mathbf{V}^{1},\mathbf{V}^{2}). 
	\end{equation*}
	By base change, we have
	\begin{equation*}
		(\tilde{q}'_{2})_{!}(\tilde{q}'_{1})^{\ast} (\tilde{q}_{2})_{!}(\tilde{q}_{1})^{\ast}\cong(\tilde{\pi})_{!} (r_{2})_{!}(r_{2})^{\ast}(\tilde{\pi}')^{\ast}.
	\end{equation*}
	
	It remains to calculate  $(\tilde{\pi})_{!} (r_{1})_{!}(r_{1})^{\ast}(\tilde{\pi}')^{\ast}$ and $(\tilde{\pi})_{!} (r_{2})_{!}(r_{2})^{\ast}(\tilde{\pi}')^{\ast}$. We divide $Y_0$ into the disjoint union $Y_{0}=Y_{0}^{0} \cup Y_{0}^{1}$, where
	\begin{equation*}
		Y^{0}_{0}=\{(\dot{x},\mathbf{V}^{1},\mathbf{V}^{2})\in Y_{0}|\mathbf{V}^{2} \subseteq \mathbf{V}^{1} \},
	\end{equation*}
	\begin{equation*}
		Y^{1}_{0}=\{(\dot{x},\mathbf{V}^{1},\mathbf{V}^{2})\in Y_{0}|\mathbf{V}^{2} \nsubseteq \mathbf{V}^{1} \}.
	\end{equation*}
	For $a=1,2$ and $b=0,1$, let $Y_{a}^{0}=(r_{a})^{-1}(Y_{0}^{0}), Y_{a}^{1}=(r_{a})^{-1}(Y_{0}^{1})$ and $\iota^{b}:Y_{0}^{b} \rightarrow Y_{0}$ be the natural embedding. Let $\tilde{r}_{a}^{b}:Y_a^b\rightarrow Y_0^b$ be the restriction of $r_{a}$ on $Y_{a}^{b}$, and $r_a^b=\iota^b\tilde{r}_{a}^{b}:Y_a^b\rightarrow Y_0$.
	
	Notice that when $\mathbf{V}^{2} \nsubseteq \mathbf{V}^{1}$, we have $\mathbf{V}^{0} = \mathbf{V}^{1}\cap\mathbf{V}^{2} , \mathbf{V}^{3}= \mathbf{V}^{1}\cup\mathbf{V}^{2}$, hence we can see that $\tilde{r}^{1}_{a}: Y^1_a\rightarrow Y_0^1$ restricts to an isomorphism onto a subvariety $Y'$ of $Y_0^1$. Hence the restriction of $(r_{1})_{!}\overline{\mathbb{Q}}_{l} $  and $(r_{2})_{!}\overline{\mathbb{Q}}_{l} $ on $Y_0^1$ are isomorphic.
	
	Since $\tilde{r}^{0}_{1}$ is a fiber bundle with each fiber isomorphic to $\mathbf{Gr}(1,\tilde{\nu}_{i}-\nu_{i})$ and $\tilde{r}^{0}_{2}$ is a fiber bundle with each fiber isomorphic to $\mathbf{Gr}(\nu_{i}-n,\nu_{i}-n+1)$, we have
	\begin{equation*}
		(\tilde{r}^{0}_{1})_{!}(\overline{\mathbb{Q}}_{l}) \cong \bigoplus\limits_{0\leqslant m \leqslant \tilde{\nu}_{i}-\nu_i-1} \overline{\mathbb{Q}}_{l}[-2m],
	\end{equation*}
	\begin{equation*}
		(\tilde{r}^{0}_{2})_{!}(\overline{\mathbb{Q}}_{l}) \cong \bigoplus\limits_{0\leqslant m \leqslant \nu_{i}-n} \overline{\mathbb{Q}}_{l}[-2m].
	\end{equation*}
	Since $(r_{1})_{!}\overline{\mathbb{Q}}_{l} $  and $(r_{2})_{!}\overline{\mathbb{Q}}_{l} $ are semisimple complexes, we must have $$(r_{1})_{!}\overline{\mathbb{Q}}_{l} \oplus  (\iota^{0})_{!}(\tilde{r}^{0}_{2})_{!}(\overline{\mathbb{Q}}_{l}) \cong (r_{2})_{!}\overline{\mathbb{Q}}_{l} \oplus (\iota^{0})_{!}(\tilde{r}^{0}_{1})_{!}(\overline{\mathbb{Q}}_{l}).$$ (Since both left hand side and right hand side are semisimple complexes and their restrictions on $Y^{1}_{0}$ or $Y^{0}_{0}$ are same.)  
	
	Note that $Y^{0}_{0}$ is isomorphic to $\dot{\mathbf{E}}_{\mathbf{V},\mathbf{W},i} \times \mathbf{Fl}(\nu_{i}-n+1,\nu_{i},\tilde{\nu}_{i})$ by definition, and morphisms $\tilde{\pi}'\iota^{0}:Y^{0}_{0} \rightarrow \dot{\mathbf{E}}_{\mathbf{V},\mathbf{W},i} \times \mathbf{Gr}(\nu_{i}, \tilde{\nu}_{i})$ and $\tilde{\pi}\iota^{0}:Y^{0}_{0} \rightarrow \dot{\mathbf{E}}_{\mathbf{V},\mathbf{W},i} \times \mathbf{Gr}(\nu_{i}-n+1, \tilde{\nu}_{i})$ can be respectively identified with the morphisms $\check{q}_{1}$ and $\check{q}_{2}$ in the definition of  $\mathcal{E}^{(n-1)}_{i}$ as follows
	\[
	\xymatrix{
		\mathbf{E}_{\mathbf{V},\mathbf{W},\hat{\Omega}} 
		&
		& \mathbf{E}_{\mathbf{V}',\mathbf{W},\hat{\Omega}} \\	
		\mathbf{E}^{0}_{\mathbf{V},\mathbf{W},i} \ar[d]_{\phi_{\mathbf{V},i}} \ar[u]^{j_{\mathbf{V},i}}
		&
		& \mathbf{E}^{0}_{\mathbf{V}',\mathbf{W},i} \ar[d]^{\phi_{\mathbf{V}',i}} \ar[u]_{j_{\mathbf{V}',i}} \\
		\dot{\mathbf{E}}_{\mathbf{V},\mathbf{W},i} \times \mathbf{Gr}(\nu_{i}, \tilde{\nu}_{i})
		& \dot{\mathbf{E}}_{\mathbf{V},\mathbf{W},i} \times \mathbf{Fl}(\nu'_{i},\nu_{i},\tilde{\nu}_{i}) \ar[r]^{\check{q}_{2}} \ar[l]_{\check{q}_{1}}
		& \dot{\mathbf{E}}_{\mathbf{V}',\mathbf{W},i}\times \mathbf{Gr}(\nu'_{i}, \tilde{\nu}_{i}).
	}
	\]
	Applying Lemma \ref{projection formula} for morphisms $\iota^0$ and $\tilde{\pi}'\iota^{0},\tilde{\pi}\iota^{0}$ respectively, we obtain
	\begin{equation*}
		(\tilde{\pi})_{!} ( (\iota^{0})_{!}\overline{\mathbb{Q}}_{l}\otimes (\tilde{\pi}')^{\ast} (-)) 
		\cong (\tilde{\pi})_{!} (\iota^{0})_{!} (\iota^{0})^{\ast} (\tilde{\pi}')^{\ast}(-)\cong (\check{q}_{2})_{!}(\check{q}_{1})^{\ast}(-).
	\end{equation*}
	Therefore, the difference between $(\tilde{\pi})_{!} (r_{1})_{!}(r_{1})^{\ast}(\tilde{\pi}')^{\ast}$ and $(\tilde{\pi})_{!} (r_{2})_{!}(r_{2})^{\ast}(\tilde{\pi}')^{\ast}$ only involves direct sums of shifts of $(\check{q}_{2})_{!}(\check{q}_{1})^{\ast}$, and so the difference between $\mathcal{E}^{(n)}_{i}\mathcal{F}_{i}$ and $\mathcal{F}_{i}\mathcal{E}^{(n)}_{i}$ only involves direct sums of shifts of $\mathcal{E}^{(n-1)}_{i}$. By direct calculation, we have
	\begin{align*}
		&\mathcal{E}^{(n)}_{i}\mathcal{F}_{i}\cong \mathcal{F}_{i}\mathcal{E}^{(n)}_{i}, &\textrm{if}\ N=0;\\
		&\mathcal{E}^{(n)}_{i}\mathcal{F}_{i}\oplus\bigoplus_{0\leqslant m\leqslant N-1}\mathcal{E}^{(n-1)}_{i}[N-1-2m]\cong \mathcal{F}_{i}\mathcal{E}^{(n)}_{i},&\textrm{if}\  N\geqslant 1;\\
		&\mathcal{E}^{(n)}_{i}\mathcal{F}_{i}\cong \mathcal{F}_{i}\mathcal{E}^{(n)}_{i} \oplus \bigoplus\limits_{0\leqslant m \leqslant -N-1} \mathcal{E}^{(n-1)}_{i}[-2m-N-1], &\textrm{if}\ N\leqslant -1.
	\end{align*}
	as desired.	
\end{proof}

\begin{lemma} \label{lemma c2}
	For any graded space $\mathbf{V},\mathbf{V}'$ such that $|\mathbf{V}'|+ni=|\mathbf{V}|+mj $ and $i \neq j$, there is an isomorphism of funtors $\mathcal{D}^{b}_{G_{\mathbf{V}}}(\mathbf{E}_{\mathbf{V},\mathbf{W},\hat{\Omega}})/\mathcal{N}_{\mathbf{V},i} \rightarrow \mathcal{D}^{b}_{G_{\mathbf{V}'}}(\mathbf{E}_{\mathbf{V}',\mathbf{W},\hat{\Omega}})/\mathcal{N}_{\mathbf{V}',i}$,
	\begin{equation*}
		\mathcal{E}^{(n)}_{i}\mathcal{F}^{(m)}_{j} \cong \mathcal{F}^{(m)}_{j}\mathcal{E}^{(n)}_{i}. 
	\end{equation*}
\end{lemma}
\begin{proof}
	We assume that $i,j$ are connected by some edges and $m=1$, the other cases can be proved by a similar argument.  
	
	On the one hand, take graded spaces $\mathbf{V}',\mathbf{V}''$ such that $|\mathbf{V}|+j=|\mathbf{V}''|=|\mathbf{V}'|+ni$ and consider the diagrams
	\[
	\xymatrix{
		\mathbf{E}_{\mathbf{V},\mathbf{W},\hat{\Omega}} & \mathbf{E}^{'}_{\mathbf{V}'',\mathbf{W},\hat{\Omega}}\ar[l]_{p_{1}} \ar[r]^{p_{2}}& \mathbf{E}^{''}_{\mathbf{V}'',\mathbf{W},\hat{\Omega}} \ar[r]^{p_{3}}& \mathbf{E}_{\mathbf{V}'',\mathbf{W},\hat{\Omega}}\\
		\mathbf{E}^{0}_{\mathbf{V},\mathbf{W},i} \ar[d]_{\phi_{1}}  \ar[u]^{j_{1}}& \mathbf{E}^{',0}_{\mathbf{V}'',\mathbf{W},i} \ar[l]_{\tilde{p}_{1}} \ar[r]^{\tilde{p}_{2}} \ar[d]_{\phi_{2}}\ar[u]^{j_{2}}& \mathbf{E}^{'',0}_{\mathbf{V}'',\mathbf{W},i}  \ar[r]^{\tilde{p}_{3}}  \ar[d]_{\phi_{3}}\ar[u]^{j_{3}}& \mathbf{E}^{0}_{\mathbf{V}'',\mathbf{W},i} \ar[d]_{\phi_{4}} \ar[u]^{j_{4}}\\
		\txt{$\dot{\mathbf{E}}_{\mathbf{V},\mathbf{W},i}$\\$\times$\\ $\mathbf{Gr}(\nu_{i}, \tilde{\nu}_{i})$}
		& Z_{1} \ar[l]_-{q_{1}} \ar[r]^{q_{2}}  & Z_{2} \ar[r]^-{q_{3}}&
		\txt{$\dot{\mathbf{E}}_{\mathbf{V}'',\mathbf{W},i}$\\$\times$\\ $\mathbf{Gr}(\nu''_{i}, \tilde{\nu}_{i}+a_{i,j})$} 
	}
	\]
	\[
	\xymatrix{
		\mathbf{E}_{\mathbf{V}'',\mathbf{W},\hat{\Omega}}
		&
		& \mathbf{E}_{\mathbf{V}',\mathbf{W},\hat{\Omega}} \\	
		\mathbf{E}^{0}_{\mathbf{V}'',\mathbf{W},i} \ar[d]_{\phi_{\mathbf{V}'',i}} \ar[u]^{j_{\mathbf{V}'',i}}
		&
		& \mathbf{E}^{0}_{\mathbf{V}',\mathbf{W},i} \ar[d]^{\phi_{\mathbf{V}',i}} \ar[u]_{j_{\mathbf{V}',i}} \\
		\txt{$\dot{\mathbf{E}}_{\mathbf{V}'',\mathbf{W},i}$\\$\times$\\ $\mathbf{Gr}(\nu''_{i}, \tilde{\nu}_{i}+a_{i,j})$}
		& \txt{$\dot{\mathbf{E}}_{\mathbf{V}'',\mathbf{W},i}$\\$\times$\\ $\mathbf{Fl}(\nu'_{i},\nu''_{i},\tilde{\nu}_{i}+a_{i,j})$}
		\ar[r]^{q'_{2}} \ar[l]_{q'_{1}}
		&  \txt{$\dot{\mathbf{E}}_{\mathbf{V}',\mathbf{W},i}$\\$\times$\\ $\mathbf{Gr}(\nu'_{i}, \tilde{\nu}_{i}+a_{i,j})$}.
	}
	\]
	By base change, up to shifts we have
	\begin{equation*}
		\begin{split}
			\mathcal{E}^{(n)}_{i}\mathcal{F}_{j }\cong&(j_{\mathbf{V}',i})_{!} (\phi_{\mathbf{V}',i})^{\ast} (q'_{2})_{!}(q'_{1})^{\ast} (\phi_{\mathbf{V}'',i})_{\flat}(j_{\mathbf{V}'',i})^{\ast}(j_{4})_{!}(\phi_{4})^{\ast}(q_{3})_{!}(q_{2})_{\flat}(q_{1})^{\ast}(\phi_{1})_{\flat}(j_{1})^{\ast}\\
			\cong&(j_{\mathbf{V}',i})_{!} (\phi_{\mathbf{V}',i})^{\ast} (q'_{2})_{!}(q'_{1})^{\ast} (q_{3})_{!}(q_{2})_{\flat}(q_{1})^{\ast}(\phi_{1})_{\flat}(j_{1})^{\ast}\\
			\cong&(j_{\mathbf{V}',i})_{!} (\phi_{\mathbf{V}',i})^{\ast} (q'_{2})_{!}(q'_{1})^{\ast} (q_{3})_{!}(q_{2})_{\flat}(q_{1})^{\ast}(\phi_{\mathbf{V},i})_{\flat}(j_{\mathbf{V},i})^{\ast}.
		\end{split}
	\end{equation*}

	On the other hand, take a graded space $\mathbf{V}'''$ such that $|\mathbf{V}|=|\mathbf{V}'''|+ni$ and consider the diagrams
	\[
	\xymatrix{
		\mathbf{E}_{\mathbf{V},\mathbf{W},\hat{\Omega}} 
		&
		& \mathbf{E}_{\mathbf{V}''',\mathbf{W},\hat{\Omega}} \\	
		\mathbf{E}^{0}_{\mathbf{V},\mathbf{W},i} \ar[d]_{\phi_{\mathbf{V},i}} \ar[u]^{j_{\mathbf{V},i}}
		&
		& \mathbf{E}^{0}_{\mathbf{V}''',\mathbf{W},i} \ar[d]^{\phi_{\mathbf{V}''',i}} \ar[u]_{j_{\mathbf{V}''',i}} \\
		\txt{$\dot{\mathbf{E}}_{\mathbf{V},\mathbf{W},i}$\\$\times$\\ $\mathbf{Gr}(\nu_{i}, \tilde{\nu}_{i})$}
		& \txt{$\dot{\mathbf{E}}_{\mathbf{V},\mathbf{W},i}$\\$\times$\\ $\mathbf{Fl}(\nu_{i}-n,\nu_{i},\tilde{\nu}_{i})$}
		\ar[r]^{\tilde{q}_{2}} \ar[l]_{\tilde{q}_{1}}
		& \txt{$\dot{\mathbf{E}}_{\mathbf{V}''',\mathbf{W},i}$\\$\times$\\ $\mathbf{Gr}(\nu_{i}-n, \tilde{\nu}_{i})$},
	}
	\]
	\[
	\xymatrix{
		\mathbf{E}_{\mathbf{V}''',\mathbf{W},\hat{\Omega}} & \mathbf{E}^{'}_{\mathbf{V}',\mathbf{W},\hat{\Omega}}\ar[l]_{p_{1}} \ar[r]^{p_{2}}& \mathbf{E}^{''}_{\mathbf{V}',\mathbf{W},\hat{\Omega}} \ar[r]^{p_{3}}& \mathbf{E}_{\mathbf{V}',\mathbf{W},\hat{\Omega}} \\
		\mathbf{E}^{0}_{\mathbf{V}''',\mathbf{W},i} \ar[d]_{\phi_{1}}  \ar[u]^{j_{1}}& \mathbf{E}^{',0}_{\mathbf{V}',\mathbf{W},i} \ar[l]_{\tilde{p}_{1}} \ar[r]^{\tilde{p}_{2}} \ar[d]_{\phi_{2}}\ar[u]^{j_{2}}& \mathbf{E}^{'',0}_{\mathbf{V}',\mathbf{W},i}  \ar[r]^{\tilde{p}_{3}}  \ar[d]_{\phi_{3}}\ar[u]^{j_{3}}& \mathbf{E}^{0}_{\mathbf{V}',\mathbf{W},i} \ar[d]_{\phi_{4}} \ar[u]^{j_{4}}\\
		\txt{$\dot{\mathbf{E}}_{\mathbf{V}''',\mathbf{W},i}$\\$\times$\\ $\mathbf{Gr}(\nu_{i}-n, \tilde{\nu}_{i})$}  & \tilde{Z}_{1} \ar[l]_-{\tilde{q}_{1}'} \ar[r]^{\tilde{q}'_{2}}  & \tilde{Z}_{2} \ar[r]^-{\tilde{q}'_{3}}& 
		\txt{$\dot{\mathbf{E}}_{\mathbf{V}',\mathbf{W},i}$\\$\times$\\ $\mathbf{Gr}(\nu_{i}-n, \tilde{\nu}_{i}+a_{i,j})$}.
	}
	\]
	Similarly, up to shifts we have
	\begin{equation*}
		\begin{split}
			\mathcal{F}_{j}\mathcal{E}^{(n)}_{i}\cong& (j_{4})_{!}(\phi_{4})^{\ast}(\tilde{q}_{3}')_{!}(\tilde{q}'_{2})_{\flat}(\tilde{q}'_{1})^{\ast}(\phi_{1})_{\flat}(j_{1})^{\ast} (j_{\mathbf{V}''',i})_{!} (\phi_{\mathbf{V}''',i})^{\ast} (\tilde{q}_{2})_{!}(\tilde{q}_{1})^{\ast} (\phi_{\mathbf{V},i})_{\flat}(j_{\mathbf{V},i})^{\ast}\\
			\cong&(j_{4})_{!}(\phi_{4})^{\ast}(\tilde{q}_{3}')_{!}(\tilde{q}'_{2})_{\flat}(\tilde{q}'_{1})^{\ast} (\tilde{q}_{2})_{!}(\tilde{q}_{1})^{\ast} (\phi_{\mathbf{V},i})_{\flat}(j_{\mathbf{V},i})^{\ast}\\
			\cong&(j_{\mathbf{V}',i})_{!}(\phi_{\mathbf{V}',i})^{\ast}(\tilde{q}_{3}')_{!}(\tilde{q}'_{2})_{\flat}(\tilde{q}'_{1})^{\ast} (\tilde{q}_{2})_{!}(\tilde{q}_{1})^{\ast} (\phi_{\mathbf{V},i})_{\flat}(j_{\mathbf{V},i})^{\ast}.
		\end{split}
	\end{equation*}

	We only need to calculate  $(q'_{2})_{!}(q'_{1})^{\ast} (q_{3})_{!}(q_{2})_{\flat}(q_{1})^{\ast}$ and $(\tilde{q}_{3}')_{!}(\tilde{q}'_{2})_{\flat}(\tilde{q}'_{1})^{\ast} (\tilde{q}_{2})_{!}(\tilde{q}_{1})^{\ast}$. Notice that  $$\nu_{i}'+n=\nu_{i}=\nu_{i}''=\nu'''_{i}+n,~\tilde{\nu}''_{i}=\tilde{\nu}_{i}+a_{i,j}=\tilde{\nu}'_{i},$$  $$\dot{\mathbf{E}}_{\mathbf{V}',\mathbf{W},i}=\dot{\mathbf{E}}_{\mathbf{V}'',\mathbf{W},i},~\dot{\mathbf{E}}_{\mathbf{V},\mathbf{W},i}=\dot{\mathbf{E}}_{\mathbf{V}''',\mathbf{W},i}.$$ Consider the following commutative diagram
	\[
	\xymatrix{
		&  & & \txt{$\dot{\mathbf{E}}_{\mathbf{V}',\mathbf{W},i}$\\$\times$\\ $\mathbf{Gr}(\nu_{i}-n, \tilde{\nu}_{i}+a_{i,j})$}  \\
		&  Y_{1}' \ar[dl]_{\pi_{1}} \ar[r]^-{r_{2}} \ar[d]^{\pi'} & Y_{1}'' \ar[ur]^{\pi_{2}} \ar[d]_{r_{1}} \ar[r]_-{r_{3}} &
		\txt{$\dot{\mathbf{E}}_{\mathbf{V}',\mathbf{W},i}$\\$\times$\\ $\mathbf{Fl}(\nu_{i}-n,\nu_{i},\tilde{\nu}_{i}+a_{i,j})$}  \ar[u]_{q'_{2}} \ar[d]^{q'_{1}}\\ 
		\txt{$\dot{\mathbf{E}}_{\mathbf{V},\mathbf{W},i}$\\$\times$\\ $\mathbf{Gr}(\nu_{i}, \tilde{\nu}_{i})$}
		& Z_{1} \ar[l]_-{q_{1}} \ar[r]^{q_{2}}  & Z_{2} \ar[r]^-{q_{3}}  &
		\txt{$\dot{\mathbf{E}}_{\mathbf{V}'',\mathbf{W},i}$\\$\times$\\ $\mathbf{Gr}(\nu_{i}, \tilde{\nu}_{i}+a_{i,j})$},
	}
	\]
	where $Y_{1}''$ is the fiber product of $q_{3}$ and $q'_{1}$, and $Y_{1}'$ is the fiber product of $q_{2}$ and $r_{1}$. 
	
	More precisely, the variety $Y_{1}''$ consists of quadruples $(\dot{x}, \dot{\mathbf{S}},\mathbf{V}^{1},\mathbf{V}^{2})$, where $\dot{x}$ belongs to $   \dot{\mathbf{E}}_{\mathbf{V}'',\mathbf{W},i}$, $\dot{\mathbf{S}}\subseteq \dot{\mathbf{V}}''$ is a $\dot{x}$-stable graded subspace and $\mathbf{V}^{1}\subseteq \mathbf{V}^{2} \subseteq  (\bigoplus\limits_{h'=i}\mathbf{S}_{h''})\oplus \mathbf{W}_{\hat{i}})$ is a flag such that $|\dot{\mathbf{S}}|=|\dot{\mathbf{V}}|$ and $ {\rm{dim}} \mathbf{V}^{1} =\nu_{i}-n,{\rm{dim}} \mathbf{V}^{2}=\nu_{i}$. 
	
	The variety $Y_{1}'$ consists of quadruples $(\dot{x}, \dot{\mathbf{S}},\mathbf{V}^{1},\mathbf{V}^{2},\dot{\rho})$ such that $\dot{x} \in   \dot{\mathbf{E}}_{\mathbf{V}'',\mathbf{W},i}$, $\dot{\mathbf{S}}\subseteq \dot{\mathbf{V}}''$ is $\dot{x}$-stable graded subspace, $\mathbf{V}^{1}\subseteq \mathbf{V}^{2} \subseteq  (\bigoplus\limits_{h'=i}\mathbf{S}_{h''})\oplus \mathbf{W}_{\hat{i}})$ is a flag and $\dot{\rho}: \dot{\mathbf{S}}\oplus\dot{\mathbf{W}} \rightarrow \dot{\mathbf{V}}\oplus\dot{\mathbf{W}}$ is a linear isomorphism which restricts to $id$ on $\dot{\mathbf{W}}$ and $ {\rm{dim}} \mathbf{V}^{1} =\nu_{i}-n,{\rm{dim}} \mathbf{V}^{2}=\nu_{i}$. 
	
	The morphisms $\pi',r_{1},r_{2},r_{3}$ are natural projections and 
	\begin{equation*}
		\pi_{1}( (\dot{x}, \dot{\mathbf{S}},\mathbf{V}^{1},\mathbf{V}^{2},\dot{\rho}) )=(\dot{\rho}_{\ast}(\dot{x}|_{\dot{\mathbf{S}}\oplus\dot{\mathbf{W}}}), \dot{\rho}(\mathbf{V}^{2})  ),
	\end{equation*}
	\begin{equation*}
		\pi_{2}( (\dot{x}, \dot{\mathbf{S}},\mathbf{V}^{1},\mathbf{V}^{2}) )=(\dot{x}, \mathbf{V}^{1}).
	\end{equation*}
	By base change, we have
	\begin{equation*}
		\begin{split}
			(q'_{2})_{!}(q'_{1})^{\ast} (q_{3})_{!}(q_{2})_{\flat}(q_{1})^{\ast}=& (q'_{2})_{!}(r_{3})_{!}(r_{1})^{\ast}(q_{2})_{\flat}(q_{1})^{\ast}\\
			=&  (q'_{2})_{!}(r_{3})_{!}(r_{2})_{\flat}(\pi')^{\ast}(q_{1})^{\ast} \\
			=&   (\pi_{2})_{!} (r_{2})_{\flat}(\pi_{1}^{\ast}).
		\end{split}
	\end{equation*}
	
	Similarly, consider the commutative diagram
	\[
	\xymatrix{
		&  & \txt{$\dot{\mathbf{E}}_{\mathbf{V}',\mathbf{W},i}$\\$\times$\\ $\mathbf{Gr}(\nu_{i}-n, \tilde{\nu}_{i}+a_{i,j})$}
		\\
		&  Y_{2}'' \ar[ur]^{\tilde{\pi}_{2}} \ar[r]^{r_{4}} & \tilde{Z}_{2} \ar[u]_{\tilde{q}_{3}'}   \\
		&  Y_{2}' \ar[dl]_{\tilde{\pi}_{1}} \ar[u]^{r_{2}} \ar[r]^-{r_{3}} \ar[d]^{r_{1}} & 
		\tilde{Z}_{1}
		\ar[u]_{\tilde{q}_{2}'} \ar[d]^{\tilde{q}_{1}'}\\
		\txt{$\dot{\mathbf{E}}_{\mathbf{V},\mathbf{W},i}$\\$\times$\\ $\mathbf{Gr}(\nu_{i}, \tilde{\nu}_{i})$}
		& \txt{$\dot{\mathbf{E}}_{\mathbf{V},\mathbf{W},i}$\\$\times$\\ $\mathbf{Fl}(\nu_{i}-n,\nu_{i},\tilde{\nu}_{i})$}
		\ar[l]_{\tilde{q}_{1}} \ar[r]^-{\tilde{q}_{2}} & 
		\txt{$\dot{\mathbf{E}}_{\mathbf{V}''',\mathbf{W},i}$\\$\times$\\ $\mathbf{Gr}(\nu_{i}-n, \tilde{\nu}_{i})$},
	}
	\]
	where the variety $Y_{2}'$ is the fiber product of $\tilde{q}'_{1}$ and $\tilde{q}_{2}$. 
	
	More precisely, recall that  $\tilde{Z}_{1}$ consists of  $(\dot{x}',\dot{\mathbf{S}} \subseteq \dot{\mathbf{V}}',\dot{\rho},\tilde{\mathbf{V}}'\subseteq  \bigoplus\limits_{h'=i} \mathbf{S}_{h''} \oplus \mathbf{W}_{\hat{i}} )$, where $\dot{x}' \in   \dot{\mathbf{E}}_{\mathbf{V}',\mathbf{W},i}$, $\dot{\mathbf{S}}$ is a $\dot{x}'$-stable subspace of  $\dot{\mathbf{V}'}$, $\dot{\rho}: \dot{\mathbf{S}}\oplus\dot{\mathbf{W}} \rightarrow \dot{\mathbf{V}}'''\oplus\dot{\mathbf{W}}$ is a linear isomorphism which restricts to $id$ on $\dot{\mathbf{W}}$ and $\tilde{\mathbf{V}}'$ is a  subspace of dimension $\nu_{i}-n$, then $Y_{2}'$ consists of quadruples $(\dot{x}', \dot{\mathbf{S}},\mathbf{V}^{1},\mathbf{V}^{2},\dot{\rho})$ such that $\dot{x}' \in   \dot{\mathbf{E}}_{\mathbf{V}',\mathbf{W},i}$, $\dot{\mathbf{S}}$ is a $\dot{x}'$-stable subspace of  $\dot{\mathbf{V}'}$, $\dot{\rho}: \dot{\mathbf{S}}\oplus\dot{\mathbf{W}}\rightarrow \dot{\mathbf{V}}'''\oplus\dot{\mathbf{W}}$ is a linear isomorphism which restricts to $id$ on $\dot{\mathbf{W}}$ and $\mathbf{V}^{1}\subseteq \mathbf{V}^{2} \subseteq  (\bigoplus\limits_{h'=i}\mathbf{S}_{h''})\oplus \mathbf{W}_{\hat{i}}$ is a flag such that $ {\rm{dim}} \mathbf{V}^{1} =\nu_{i}-n,{\rm{dim}} \mathbf{V}^{2}=\nu_{i}$. 
	
	Let $Y_{2}''$ be the variety consists of quadruples $(\dot{x}', \dot{\mathbf{S}},\mathbf{V}^{1},\mathbf{V}^{2})$ such that $(\dot{x}', \dot{\mathbf{S}},\mathbf{V}^{1},\mathbf{V}^{2})$ satisfies the same conditions as in $Y_{2}'$. 
	
	The morphisms $r_{2},r_{3},r_{4}$ are projections and 
	\begin{equation*}
		r_{1}((\dot{x}', \dot{\mathbf{S}},\mathbf{V}^{1},\mathbf{V}^{2},\dot{\rho}))=((\dot{\rho})_{\ast}(\dot{x}'|_{\dot{\mathbf{S}}\oplus \dot{\mathbf{W}}}),\dot{\rho}(\mathbf{V}^{1}),\dot{\rho}(\mathbf{V}^{2})),
	\end{equation*}
	\begin{equation*}
		\tilde{\pi}_{1} ((\dot{x}', \dot{\mathbf{S}},\mathbf{V}^{1},\mathbf{V}^{2},\dot{\rho}))= ((\dot{\rho})_{\ast}(\dot{x}'|_{\dot{\mathbf{S}}\oplus\dot{\mathbf{W}}}),\dot{\rho}(\mathbf{V}^{2})),
	\end{equation*}
	\begin{equation*}
		\tilde{\pi}_{2}((\dot{x}', \dot{\mathbf{S}},\mathbf{V}^{1},\mathbf{V}^{2}))= (\dot{x}',\mathbf{V}^{1}),
	\end{equation*}
	then the middle square is Cartesian. By base change, we have
	\begin{equation*}
		(\tilde{q}_{3}')_{!}(\tilde{q}'_{2})_{\flat}(\tilde{q}'_{1})^{\ast} (\tilde{q}_{2})_{!}(\tilde{q}_{1})^{\ast} \cong (\tilde{\pi}_{2})_{!} (r_{2})_{\flat}(\tilde{\pi}_{1}^{\ast}).
	\end{equation*}
	
	Notice that $\dot{\mathbf{E}}_{\mathbf{V}',\mathbf{W},i}=\dot{\mathbf{E}}_{\mathbf{V}'',\mathbf{W},i},\dot{\mathbf{E}}_{\mathbf{V},\mathbf{W},i}=\dot{\mathbf{E}}_{\mathbf{V}''',\mathbf{W},i}$, there are  natural isomorphisms $Y_{1}' \cong Y_{2}'$, $Y_{1}''\cong Y_{2}''$. Under the isomorphisms, we have $\tilde{\pi}_{1}=\tilde{\pi}_{2}$ and $\pi_{1}=\pi_{2}$, and so
	\begin{equation*}
		(q'_{2})_{!}(q'_{1})^{\ast} (q_{3})_{!}(q_{2})_{\flat}(q_{1})^{\ast}\cong(\tilde{q}_{3}')_{!}(\tilde{q}'_{2})_{\flat}(\tilde{q}'_{1})^{\ast} (\tilde{q}_{2})_{!}(\tilde{q}_{1})^{\ast},
	\end{equation*}
	as desired.
\end{proof}

\begin{lemma}\label{Lemma 16}
	The functors $\mathcal{E}^{(n)}_{i}$ for $n\geqslant 1$ satisfy the following relation
	\begin{equation*}
		\bigoplus \limits_{0 \leqslant m < n } \mathcal{E}^{(n)}_{i}[n-1-2m] \cong \mathcal{E}^{(n-1)}_{i}\mathcal{E}_{i}\cong \mathcal{E}_{i}\mathcal{E}^{(n-1)}_{i},\  n \geqslant 2,
	\end{equation*}
	as endofunctors of the localization $\bigoplus\limits_{\mathbf{V}}\mathcal{D}^{b}_{G_{\mathbf{V}}}(\mathbf{E}_{\mathbf{V},\mathbf{W},\hat{\Omega}})/\mathcal{N}_{\mathbf{V},i}$.
\end{lemma}

\begin{proof}
	We only prove the first isomorphism, the other one can be proved similarly. Take graded spaces $\mathbf{V},\mathbf{V}'$ and $\mathbf{V}''$ such that $|\mathbf{V}|-|\mathbf{V}'|=i$ and $|\mathbf{V}|-|\mathbf{V}''|=ni$, consider the following commutative diagram
	\[
	\xymatrix{
		\txt{$\dot{\mathbf{E}}_{\mathbf{V},\mathbf{W},i}$\\$\times$\\ $\mathbf{Fl}(\nu_{i}-n,\nu_{i},\tilde{\nu}_{i})$} \ar[rr]^{q''_{2}} \ar[dd]_{q_{1}''} &  & 	\txt{$\dot{\mathbf{E}}_{\mathbf{V},\mathbf{W},i}$\\$\times$\\ $\mathbf{Gr}(\nu_{i}-n, \tilde{\nu}_{i})$}  \\
		&	\txt{$\dot{\mathbf{E}}_{\mathbf{V},\mathbf{W},i}$\\$\times$\\ $\mathbf{Fl}(\nu_{i}-n,\nu_{i}-1,\nu_{i},\tilde{\nu}_{i})$}
		\ar[ul]^{r} \ar[r]^-{\pi} \ar[d]^{\pi'} & 	\txt{$\dot{\mathbf{E}}_{\mathbf{V},\mathbf{W},i}$\\$\times$\\ $\mathbf{Fl}(\nu_{i}-n,\nu_{i}-1,\tilde{\nu}_{i})$}
		\ar[u]_{q_{2}} \ar[d]^{q_{1}}\\
		\txt{$\dot{\mathbf{E}}_{\mathbf{V},\mathbf{W},i}$\\$\times$\\ $\mathbf{Gr}(\nu_{i},\tilde{\nu}_{i})$}
		& 	\txt{$\dot{\mathbf{E}}_{\mathbf{V},\mathbf{W},i}$\\$\times$\\ $\mathbf{Fl}(\nu_{i}-1,\nu_{i},\tilde{\nu}_{i})$}
		\ar[l]_{q'_{1}} \ar[r]^-{q'_{2}} & 	\txt{$\dot{\mathbf{E}}_{\mathbf{V},\mathbf{W},i}$\\$\times$\\ $\mathbf{Gr}(\nu_{i}-1, \tilde{\nu}_{i})$},
	}
	\]
	where the morphisms are obvious forgetting maps. By base change, we have 
	\begin{equation*}
		\begin{split}
			(q_{2})_{!}(q_{1})^{\ast}(q'_{2})_{!}(q'_{1})^{\ast} \cong & (q_{2} \pi)_{!}( q'_{1} \pi' )^{\ast} \\
			\cong &  (q''_{2})_{!} r_{!}r^{\ast} (q''_{1})^{\ast} \\
			\cong & \bigoplus \limits_{0 \leqslant m < n } (q''_{2})_{!}(q''_{1})^{\ast}[-2m ],
		\end{split}
	\end{equation*}
	where the last isomorphism holds by the projection formula, since $r$ is a trivial fiber bundle with fiber isomorphic to $\mathbb{P}^{(n-1)}$. Compose them with quasi-inverse equivalences $(\phi_{\mathbf{V},i})^{\ast},(\phi_{\mathbf{V},i})_{\flat}$ and $(j_{\mathbf{V},i})^{\ast},(j_{\mathbf{V},i})_{!}$, we get a proof.
\end{proof}

For a triangulated category $\mathcal{D}$ and a finite set of objects $\mathcal{B}=\{B_{\alpha} \}_{\alpha \in S}$, we say an object $A \in \mathcal{D}$ is generated by objects in $\mathcal{B}$ if $A$ belongs to the thick subcategory generated by objects in $\mathcal{B}$.

\begin{lemma}\label{commute4}
	Assume $i \neq j$ in $I$, take graded spaces $\mathbf{V},\mathbf{V}',\mathbf{V}''$ and $\mathbf{V}'''$ such that $|\mathbf{V}|+nj=|\mathbf{V}''|=|\mathbf{V}'|+i=|\mathbf{V}'''|+i+nj$, then  for any $A$ in  $\mathcal{D}^{b}_{G_{\mathbf{V}}}(\mathbf{E}_{\mathbf{V},\mathbf{W},\hat{\Omega}})$, the object $\mathcal{E}_{i}\mathcal{F}_{j}^{(n),\vee}(A)$  in $\mathcal{D}^{b}_{G_{\mathbf{V}'}}(\mathbf{E}_{\mathbf{V}',\mathbf{W},\hat{\Omega}})/\mathcal{N}_{\mathbf{V}',i}$ is generated by those $\mathcal{F}_{j}^{(n),\vee}\mathcal{E}_{i,r,a}(A)$ with $ r \leqslant \nu_{i}$ and $a=0,1$. 
\end{lemma}
\begin{proof}
	Since $i\neq j$, we can use Fourier-Deligne transforms changing the orientation of $j \rightarrow \hat{j}$. (Notice that this Fourier-Deligne transform commutes with $\mathcal{F}_{j}^{(n),\vee}$ and $\mathcal{E}^{(n)}_{i}$). We can assume that  $i$ is a source and $j$ is a sink, then any subspace $W$ of $\mathbf{V}''_{j}$ is automatically $x$-stable for all $x \in \mathbf{E}_{\mathbf{V}'',\mathbf{W},\hat{\Omega}}$.
	
	Let $X_{1,r}$ be the variety $(p_{1})^{-1}(\mathbf{E}^{r}_{\mathbf{V},\mathbf{W},i}) \cap  \mathbf{E}^{',0}_{\mathbf{V}'',\mathbf{W},i}$ , $X_{2,r}=p_{2}(X_{1,r})$, $X_{1, \geqslant r}= \bigcup\limits_{r' \geqslant r } X_{1,r'}$ and $X_{2, \geqslant r}= \bigcup\limits_{r' \geqslant r } X_{2,r'}$. More precisely, the variety $X_{2,r}$ consists of $(x,\mathbf{S})$, where $x \in \mathbf{E}^{0}_{\mathbf{V}'',\mathbf{W},i}$ and $\mathbf{S}$ is an $n$-dimensional subspace of $\mathbf{V}''_{j}$ such that ${\rm{Im}}(\bigoplus\limits_{h \in \hat{\Omega}, h'=i}(x_{h})) \cap \mathbf{S}^{\oplus a_{i,j}} $ is an $r$-dimensional subspace of $\bigoplus\limits_{h \in \Omega,h'=i}\mathbf{V}''_{h''} \oplus \mathbf{W}_{\hat{i}}$, here $a_{i,j}$ is the number of arrows between $i$ and $j$. The variety $X_{1,r}$ consists of $(x,\mathbf{S},\rho)$ such that $(x,\mathbf{S}) \in X_{2,r}$ and $\rho: \mathbf{V}''/\mathbf{S} \cong \mathbf{V}$ is a linear isomorphism of graded spaces. Then we have the following commutative diagrams
	\[
	\xymatrix{
		\mathbf{E}_{\mathbf{V},\mathbf{W},\hat{\Omega}} & \mathbf{E}^{'}_{\mathbf{V}'',\mathbf{W},\hat{\Omega}}\ar[l]_{p_{1}} \ar[r]^{p_{2}}& \mathbf{E}^{''}_{\mathbf{V}'',\mathbf{W},\hat{\Omega}} \ar[r]^{p_{3}}& \mathbf{E}_{\mathbf{V}'',\mathbf{W},\hat{\Omega}} \\
		\mathbf{E}_{\mathbf{V},\mathbf{W},\hat{\Omega}} \ar@2{-}[u] & \mathbf{E}^{',0}_{\mathbf{V}'',\mathbf{W},i} \ar[l]_{\tilde{p}_{1}} \ar[r]^{\tilde{p}_{2}} \ar[u]_{j_{2}} & \mathbf{E}^{'',0}_{\mathbf{V}'',\mathbf{W},i}  \ar[r]^{\tilde{p}_{3}} \ar[u]_{j_{3}} & \mathbf{E}^{0}_{\mathbf{V}'',\mathbf{W},i} \ar[u]_{j_{4}} 
		\\
		\mathbf{E}^{r}_{\mathbf{V},\mathbf{W},i} \ar[u]^{j_{\mathbf{V},r}} & X_{1,r} \ar[l]_{x_{1,r}} \ar[r]^{x_{2,r}} \ar[u]_{u_{1,r}} & X_{2,r} \ar[r]^{x_{3,r}} \ar[u]_{u_{2,r}} & \mathbf{E}^{0}_{\mathbf{V}'',\mathbf{W},i} \ar@2{-}[u]
	}
	\]
	\[
	\xymatrix{
		\mathbf{E}_{\mathbf{V},\mathbf{W},\hat{\Omega}} & \mathbf{E}^{'}_{\mathbf{V}'',\mathbf{W},\hat{\Omega}}\ar[l]_{p_{1}} \ar[r]^{p_{2}}& \mathbf{E}^{''}_{\mathbf{V}'',\mathbf{W},\hat{\Omega}} \ar[r]^{p_{3}}& \mathbf{E}_{\mathbf{V}'',\mathbf{W},\hat{\Omega}} \\
		\mathbf{E}_{\mathbf{V},\mathbf{W},\hat{\Omega}} \ar@2{-}[u] & \mathbf{E}^{',0}_{\mathbf{V}'',\mathbf{W},i} \ar[l]_{\tilde{p}_{1}} \ar[r]^{\tilde{p}_{2}} \ar[u]_{j_{2}} & \mathbf{E}^{'',0}_{\mathbf{V}'',\mathbf{W},i}  \ar[r]^{\tilde{p}_{3}} \ar[u]_{j_{3}} & \mathbf{E}^{0}_{\mathbf{V}'',\mathbf{W},i} \ar[u]_{j_{4}} 
		\\
		\mathbf{E}^{\geqslant r}_{\mathbf{V},\mathbf{W},i} \ar[u]^{j_{\mathbf{V},\geqslant r}} & X_{1,\geqslant r} \ar[l]_{x_{1,\geqslant r}} \ar[r]^{x_{2,\geqslant r}} \ar[u]_{u_{1,\geqslant r}} & X_{2,\geqslant r} \ar[r]^{x_{3,\geqslant r}} \ar[u]_{u_{2,\geqslant r}} & \mathbf{E}^{0}_{\mathbf{V}'',\mathbf{W},i} \ar@2{-}[u]
	}
	\]
	where those $x_{i,r}$ and $x_{i,\geqslant r},i=1,2,3$ are restrictions of $p_{i},i=1,2,3$ respectively and those $j_{2},j_{3},j_{4}$ and $u_{i,r},u_{i,\geqslant r},i=1,2$ are obvious embeddings. Then we have the following isomorphism for any object $A_{1}$ in $\mathcal{D}^{b}_{G_{\mathbf{V}}}(\mathbf{E}_{\mathbf{V},\mathbf{W},\hat{\Omega}} )$
	\begin{equation*}
		(j_{4})^{\ast}(p_{3})_{!}(p_{2})_{\flat}(p_{1})^{\ast}(A_{1})
		\cong (\tilde{p}_{3})_{!}(\tilde{p}_{2})_{\flat}(\tilde{p}_{1})^{\ast}(A_{1}).
	\end{equation*} 
	Since the square of $u_{1,r},u_{2,r},x_{2,r}$ and $\tilde{p}_{2}$ is Cartisian, we have
	\begin{equation*}
		\begin{split}
			& (\tilde{p}_{3})_{!}(\tilde{p}_{2})_{\flat}(u_{1,r})_{!}(u_{1,r})^{\ast}(\tilde{p}_{1})^{\ast}(A_{1}) \\
			\cong & (\tilde{p}_{3})_{!}(\tilde{p}_{2})_{\flat}(u_{1,r})_{!}(x_{1,r})^{\ast}(j_{\mathbf{V},r})^{\ast}(A_{1})\\
			\cong & 	(\tilde{p}_{3})_{!}(u_{2,r})_{!}(x_{2,r})_{\flat}(x_{1,r})^{\ast}(j_{\mathbf{V},r})^{\ast}(A_{1}) \\
			\cong & 	(x_{3,r})_{!}(x_{2,r})_{\flat}(x_{1,r})^{\ast}(j_{\mathbf{V},r})^{\ast}(A_{1}). 
		\end{split}
	\end{equation*} 
	Similarly, $$(\tilde{p}_{3})_{!}(\tilde{p}_{2})_{\flat}(u_{1,\geqslant r})_{!}(u_{1,\geqslant r})^{\ast}(\tilde{p}_{1})^{\ast}(A_{1}) \cong (x_{3,\geqslant r})_{!}(x_{2,\geqslant r})_{\flat}(x_{1,\geqslant r})^{\ast}(j_{\mathbf{V},\geqslant r})^{\ast}(A_{1}).$$
	Recall that $X_{1,0} \cup X_{1,\geqslant 1} =\mathbf{E}^{',0}_{\mathbf{V}'',\mathbf{W},i}$ induces distinguished triangles $$(u_{1,0})_{!}(u_{1,0})^{\ast} \rightarrow \mathbf{Id} \rightarrow (u_{1,\geqslant 1})_{!}(u_{1,\geqslant 1})^{\ast} \xrightarrow{+1},$$
	then the complex $(\tilde{p}_{3})_{!}(\tilde{p}_{2})_{\flat}(\tilde{p}_{1})^{\ast}(A)$ admits a triangle
	\begin{equation*}
		\begin{split}
			(x_{3,0})_{!}(x_{2,0})_{\flat}(x_{1,0})^{\ast}(j_{\mathbf{V},0})^{\ast}(A)&\rightarrow (\tilde{p}_{3})_{!}(\tilde{p}_{2})_{\flat}(\tilde{p}_{1})^{\ast}(A)\\ & \rightarrow 
			(x_{3,\geqslant 1})_{!}(x_{2,\geqslant 1})_{\flat}(x_{1,\geqslant 1})^{\ast}(j_{\mathbf{V},\geqslant 1})^{\ast}(A) \xrightarrow{+1}.
		\end{split}
	\end{equation*} 
	Similarly, we have distinguished triangles 
	\begin{equation*}
		\begin{split}
			(x_{3,r})_{!}(x_{2,r})_{\flat}(x_{1,r})^{\ast}(j_{\mathbf{V},r})^{\ast}(A_{1})&\rightarrow (x_{3,\geqslant r})_{!}(x_{2,\geqslant r})_{\flat}(x_{1,\geqslant r})^{\ast}(j_{\mathbf{V},\geqslant r})^{\ast}(A_{1}) \\
			&\rightarrow (x_{3,\geqslant r+1})_{!}(x_{2,\geqslant r+1})_{\flat}(x_{1,\geqslant r+1})^{\ast}(j_{\mathbf{V},\geqslant r+1})^{\ast}(A_{1}) \xrightarrow{+1},
		\end{split}
	\end{equation*}
	so $(\tilde{p}_{3})_{!}(\tilde{p}_{2})_{\flat}(\tilde{p}_{1})^{\ast}(A_{1})$ is generated by  those $(x_{3,r})_{!}(x_{2,r})_{\flat}(x_{1,r})^{\ast}(j_{\mathbf{V},r})^{\ast}(A_{1})$.
	
	Let $Y_{2,r}\subseteq \dot{\mathbf{E}}_{\mathbf{V}'',\mathbf{W},i}\times \mathbf{Gr}(n,\nu''_{j})\times \mathbf{Gr}(\nu_{i}'',\tilde{\nu}''_{i})$ be the variety consists of $(\dot{x},\mathbf{S},\mathbf{S}_{1})$ where $\dot{x} \in \dot{\mathbf{E}}_{\mathbf{V}'',\mathbf{W},i}$, $\mathbf{S}_{1}$ is an $\nu''_{i}$-dimensional subspace of $\tilde{\mathbf{V}}''_{i}=\bigoplus\limits_{h \in \Omega,h'=i}\mathbf{V}''_{h''} \oplus \mathbf{W}_{\hat{i}}$, and  $\mathbf{S}$ is an $n$-dimensional subspace of $\mathbf{V}''_{j}$ such that $\mathbf{S}_{1} \cap \mathbf{S}^{\oplus a_{i,j}} $ is an $r$-dimensional subspace of $\tilde{\mathbf{V}}''_{i}$. Let $Y_{1,r}$ be the variety consists of $(\dot{x},\mathbf{S},\mathbf{S}_{1},\dot{\rho})$ such that $(\dot{x},\mathbf{S},\mathbf{S}_{1}) \in Y_{2,r}$ and  $\dot{\rho}:\dot{\mathbf{V}}''/\mathbf{S} \cong \dot{\mathbf{V}} $ is a linear isomorphism. Denote $(x_{1,r})^{-1}(	\mathbf{E}^{0}_{\mathbf{V}_{r},\mathbf{W},i})$ by $X_{0,r}$, here we identify $	\mathbf{E}^{0}_{\mathbf{V}_{r},\mathbf{W},i}$ with $'F_{r}$. Then we have the following commutative diagram
	\[
	\xymatrix{
		& \mathbf{E}^{r}_{\mathbf{V},\mathbf{W},i} & X_{1,r}  \ar[l]_{x_{1,r}} \ar[r]^{x_{2,r}} \ar[dd]_{v_{1,r}} &  X_{2,r} \ar[r]^{x_{3,r}} \ar[dd]_{v_{2,r}} & \mathbf{E}^{0}_{\mathbf{V}'',\mathbf{W},i} \ar[dd]^{\phi_{\mathbf{V}'',i}} \\
		\mathbf{E}^{0}_{\mathbf{V}_{r},\mathbf{W},i} \ar[ur]^{\iota_{\mathbf{V},r}} \ar[dr]_{\phi_{\mathbf{V}_{r},i} } & X_{0,r} \ar[l]^{x_{0,r}} \ar[ur]_{\iota'_{\mathbf{V},r}} \ar[dr]_{v_{0,r}} &  &  & 
		\\
		&\txt{$\dot{\mathbf{E}}_{\mathbf{V}_{r},\mathbf{W},i}$\\$\times$\\ $\mathbf{Gr}(\nu_{i}-r,\tilde{\nu}_{i})$}
		& Y_{1,r} \ar[l]^-{y_{1,r}} \ar[r]_{y_{2,r}}  & Y_{2,r} \ar[r]_-{y_{3,r}}  &\txt{$\dot{\mathbf{E}}_{\mathbf{V}'',\mathbf{W},i}$\\$\times$\\ $\mathbf{Gr}(\nu_{i}'',\tilde{\nu}''_{i})$},
	}
	\]
	where $\iota'_{\mathbf{V},r}$ is the natural embedding, $x_{0,r}$ is the restriction of $x_{1,r}$, $v_{0,r}$ is the restriction of $v_{1,r}$. Recall that $\dot{\rho}$ also induces $ \dot{\rho}(\tilde{\mathbf{V}}''_{i}/\mathbf{S}^{\oplus a_{i,j}})= \tilde{\mathbf{V}}_{i}= \bigoplus\limits_{h \in \Omega,h'=i}\mathbf{V}_{h''} \oplus \mathbf{W}_{\hat{i}} $, the morphisms $v_{1,r},v_{2,r},y_{1,r},y_{2,r}$ and $y_{3,r}$ are defined by
	$$ y_{1,r}( (\dot{x},\mathbf{S},\mathbf{S}_{1},\dot{\rho} ) ) =((\dot{\rho})_{\ast}(\dot{x}), \dot{\rho}(\mathbf{S}_{1}/\mathbf{S}^{\oplus a_{i,j}}) ), $$ 
	$$ y_{2,r}( (\dot{x},\mathbf{S},\mathbf{S}_{1},\dot{\rho} ) ) =(\dot{x},\mathbf{S},\mathbf{S}_{1}), $$ 
	$$ y_{3,r}( (\dot{x},\mathbf{S},\mathbf{S}_{1} ) ) =(\dot{x},\mathbf{S}_{1}), $$ 
	$$ v_{1,r}( (x,\mathbf{S},\rho ) ) =(\dot{x},{\rm{Im}}(\bigoplus\limits_{h \in \hat{\Omega}, h'=i}(x_{h}))  ,\mathbf{S},\dot{\rho}), $$
	$$ v_{2,r}( (x,\mathbf{S}) ) =(\dot{x},{\rm{Im}}(\bigoplus\limits_{h \in \hat{\Omega}, h'=i}(x_{h}))  ,\mathbf{S}).$$
	Then we have the following equation 
	\begin{equation}\label{equationA2}
		(\phi_{\mathbf{V}'',i})_{\flat}(x_{3,r})_{!}(x_{2,r})_{\flat}(x_{1,r})^{\ast}(A_{2}) \cong  (y_{3,r})_{!}(y_{2,r})_{\flat}(v_{1,r})_{\flat}(x_{1,r})^{\ast}(A_{2}).
	\end{equation}
	Note that 
	\begin{equation}\label{sharp}
		\begin{split}
			& (y_{1,r})^{\ast}(\phi_{\mathbf{V}_{r},i})_{\flat}(\iota_{\mathbf{V},r})^{_\clubsuit} 
			\cong  (v_{0,r})_{\flat}(x_{0,r})^{\ast}(\iota_{\mathbf{V},r})^{_\clubsuit}\\
			\cong & (v_{0,r})_{\flat}(\iota'_{\mathbf{V},r})^{_\clubsuit}(x_{1,r})^{\ast} 
			\cong  (v_{1,r})_{\flat}(x_{1,r})^{\ast},
		\end{split}
	\end{equation}
	it follows that
	\begin{equation}\label{equationA3}
		(\phi_{\mathbf{V}'',i})_{\flat}(x_{3,r})_{!}(x_{2,r})_{\flat}(x_{1,r})^{\ast}(A_{2}) \cong  (y_{3,r})_{!}(y_{2,r})_{\flat}(y_{1,r})^{\ast}(\phi_{\mathbf{V}_{r},i})_{\flat}(\iota_{\mathbf{V},r})^{_\clubsuit}(A_{2})
	\end{equation} 
	for any $A_{2}$ in $\mathcal{D}^{b}_{G_{\mathbf{V}}}(\mathbf{E}^{r}_{\mathbf{V},\mathbf{W},i} )$.
	
	Let $Z_{2,r}$ be the variety  which  consists of $(\dot{x},\mathbf{S},\mathbf{S}_{2} \subseteq \mathbf{S}_{1})$, where $\dot{x} \in \dot{\mathbf{E}}_{\mathbf{V}'',\mathbf{W},i}$, $\mathbf{S}_{1}$ is an $\nu''_{i}$-dimensional subspace of $\tilde{\mathbf{V}}''_{i}$, $\mathbf{S}_{2}$ is an $(\nu''_{i}-1)$-dimensional subspace of $\tilde{\mathbf{V}}''_{i}$ , and  $\mathbf{S}$ is an $n$-dimensional subspace of $\mathbf{V}''_{j}$ such that $\mathbf{S}_{1} \cap \mathbf{S}^{\oplus a_{i,j}} $ is an $r$-dimensional subspace of $\tilde{\mathbf{V}}''_{i}$. Let $Z_{1,r}$ be the variety which consists of $(\dot{x},\mathbf{S},\mathbf{S}_{2} \subseteq \mathbf{S}_{1},\dot{\rho})$ such that $(\dot{x},\mathbf{S},\mathbf{S}_{2} \subseteq \mathbf{S}_{1}) \in Z_{2,r}$ and  $\dot{\rho}:\dot{\mathbf{V}}''/\mathbf{S} \cong \dot{\mathbf{V}} $ is a linear isomorphism.  For $l=1,2$ and $a=0,1$, let $Z_{l,r,a}$ be the subset of $Z_{l,r}$ consisting of points such that $\mathbf{S}_{2} \cap \mathbf{S}^{\oplus a_{i,j}} $ is an $(r-a)$-dimensional subspace of $\tilde{\mathbf{V}}''_{i}$. Then we have the following commutative diagrams
	
	\[
	\xymatrix{
		\txt{$\dot{\mathbf{E}}_{\mathbf{V}_{r},\mathbf{W},i}$\\$\times$\\ $\mathbf{Gr}(\nu_{i}-r,\tilde{\nu}_{i})$} & Y_{1,r} \ar[l]_-{y_{1,r}} \ar[r]^{y_{2,r}}& Y_{2,r} \ar[r]^-{y_{3,r}}& \txt{$\dot{\mathbf{E}}_{\mathbf{V}'',\mathbf{W},i}$\\$\times$\\ $\mathbf{Gr}(\nu_{i}'',\tilde{\nu}''_{i})$} \\
		& Z_{1,r} \ar[r]^{z_{2,r}} \ar[u]_{w_{1}} & Z_{2,r} \ar[r]^-{z_{3,r}} \ar[u]_{w_{2}} & \txt{$\dot{\mathbf{E}}_{\mathbf{V}'',\mathbf{W},i}$\\$\times$\\ $\mathbf{Fl}(\nu_{i}''-1,\nu_{i}'',\tilde{\nu}''_{i})$} \ar[u]_{q_{1}} 
		\\
		\txt{$\dot{\mathbf{E}}_{\mathbf{V}_{r},\mathbf{W},i}$\\$\times$\\ $\mathbf{Fl}(\nu_{i}-1-r+a,\nu_{i}-r,\tilde{\nu}_{i})$} \ar[uu]^{q_{1,r,a}} & Z_{1,r,a} \ar[l]_-{z_{1,r,a}} \ar[r]^{z_{2,r,a}} \ar[u]_{i_{1,r,a}} & Z_{2,r,a} \ar[r]^-{z_{3,r,a}} \ar[u]_{i_{2,r,a}} & \txt{$\dot{\mathbf{E}}_{\mathbf{V}'',\mathbf{W},i}$\\$\times$\\ $\mathbf{Fl}(\nu_{i}''-1,\nu_{i}'',\tilde{\nu}''_{i})$}, \ar@2{-}[u]
	}
	\]
	
	\[
	\xymatrix{
		\txt{$\dot{\mathbf{E}}_{\mathbf{V}_{r},\mathbf{W},i}$\\$\times$\\ $\mathbf{Gr}(\nu_{i}-r,\tilde{\nu}_{i})$} & Y_{1,r} \ar[l]_-{y_{1,r}} \ar[r]^{y_{2,r}}& Y_{2,r} \ar[r]^-{y_{3,r}}& \txt{$\dot{\mathbf{E}}_{\mathbf{V}'',\mathbf{W},i}$\\$\times$\\ $\mathbf{Gr}(\nu_{i}'',\tilde{\nu}''_{i})$} \\
		\txt{$\dot{\mathbf{E}}_{\mathbf{V}_{r},\mathbf{W},i}$\\$\times$\\ $\mathbf{Fl}(\nu_{i}-1-r+a,\nu_{i}-r,\tilde{\nu}_{i})$} \ar[u]^-{q_{1,r,a}}	  & Z_{1,r,a} \ar[l]_-{z_{1,r,a}} \ar[r]^{z_{2,r,a}} \ar[u]_{w_{1,a}} & Z_{2,r,a} \ar[r]^-{z_{3,r,a}} \ar[u]_{w_{2,a}} & \txt{$\dot{\mathbf{E}}_{\mathbf{V}'',\mathbf{W},i}$\\$\times$\\ $\mathbf{Fl}(\nu_{i}''-1,\nu_{i}'',\tilde{\nu}''_{i})$}, \ar[u]_-{q_{1}} 
	}
	\]
	where $z_{2,r}$ is defined by forgetting $\dot{\rho}$, $z_{3,r}$ is defined by forgetting $\mathbf{S}$, $w_{1}$ and $w_{2}$ are defined by forgetting $\mathbf{S}_{2}$, $i_{1,r,a}$ and $i_{2,r,a}$ are obvious embeddings, $w_{1,a}$ and $w_{2,a}$ are restriction of $w_{1}$ and $w_{2}$ respectively,  $z_{2,r,a}$ and  $z_{3,r,a}$ are restrictions of $z_{2,r}$ and $z_{3,r}$ respectively, and the morphism $z_{1,r,a}$ is defined by
	$$ z_{1,r,a}((\dot{x},\mathbf{S},\mathbf{S}_{2} \subseteq\mathbf{S}_{1},\dot{\rho} )  )=  ((\dot{\rho})_{\ast}(\dot{x}), \dot{\rho}(\mathbf{S}_{2}/\mathbf{S}^{\oplus a_{i,j}}) \subseteq \dot{\rho}(\mathbf{S}_{1}/\mathbf{S}^{\oplus a_{i,j}}) ).$$ 
	Then $(q_{1})^{\ast}(y_{3,r})_{!}(y_{2,r})_{\flat}(y_{1,r})^{\ast} \cong (z_{3,r})_{!}(z_{2,r})_{\flat}(w_{1})^{\ast} (y_{1,r})^{\ast}$. Note that there is a distinguished triangle $$(i_{1,r,1})_{!}(i_{1,r,1})^{\ast} \rightarrow \mathbf{Id} \rightarrow (i_{1,r,0})_{!}(i_{1,r,0})^{\ast} \xrightarrow{+1} $$ and the square of $z_{2,r},z_{2,r,a}, i_{1,r,a}$ and $i_{2,r,a}$ is Cartisian, we can see that for any object $A_{3}$ in $\mathcal{D}^{b}_{G_{\mathbf{V}}}(\dot{\mathbf{E}}_{\mathbf{V}_{r},\mathbf{W},i}\times \mathbf{Gr}(\nu_{i}-r,\tilde{\nu}_{i}) ) $, the object $(q_{1})^{\ast}(y_{3,r})_{!}(y_{2,r})_{\flat}(y_{1,r})^{\ast}(A_{3})$ is generated by objects $ (z_{3,r,a})_{!}(z_{2,r,a})_{\flat}(z_{1,r,a})^{\ast}(q_{1,r,a})^{\ast}(A_{3})$ with $a=0,1$.
	
	Let $X'_{2,r-a}$  be the variety consisting of $(x,\mathbf{S})$, where $x \in \mathbf{E}^{0}_{\mathbf{V}',\mathbf{W},i}$ and $\mathbf{S}$ is an $n$-dimensional subspace of $\mathbf{V}'_{j}$ such that ${\rm{Im}}(\bigoplus\limits_{h \in \hat{\Omega}, h'=i}(x_{h})) \cap \mathbf{S}^{\oplus a_{i,j}} $ is an $(r-a)$-dimensional subspace of $\tilde{\mathbf{V}}'_{i}=\bigoplus\limits_{h \in \Omega,h'=i}\mathbf{V}'_{h''} \oplus \mathbf{W}_{\hat{i}}$. Let $X'_{1,r-a}$ be the variety consisting of $(x,\mathbf{S},\rho)$ such that $(x,\mathbf{S}) \in X'_{2,r-a}$ and $\rho: \mathbf{V}'/\mathbf{S} \cong \mathbf{V}'''$ is a linear isomorphism of graded spaces. Let $Y'_{2,r-a}\subseteq \dot{\mathbf{E}}_{\mathbf{V}',\mathbf{W},i}\times \mathbf{Gr}(n,\nu'_{j})\times \mathbf{Gr}(\nu_{i}',\tilde{\nu}'_{i})$ be the variety which consists of $(\dot{x},\mathbf{S},\mathbf{S}_{1})$ where $\dot{x} \in \dot{\mathbf{E}}_{\mathbf{V}',\mathbf{W},i}$, $\mathbf{S}_{1}$ is an $\nu'_{i}$-dimensional subspace of $\tilde{\mathbf{V}}'_{i}$, and  $\mathbf{S}$ is an $n$-dimensional subspace of $\mathbf{V}'_{j}$ such that $\mathbf{S}_{1} \cap \mathbf{S}^{\oplus a_{i,j}} $ is an $(r-a)$-dimensional subspace of $\tilde{\mathbf{V}}'_{i}$. Let $Y'_{1,r-a}$ be the variety which consists of $(\dot{x},\mathbf{S},\mathbf{S}_{1},\dot{\rho})$ such that $(\dot{x},\mathbf{S},\mathbf{S}_{1}) \in Y'_{2,r-a}$ and  $\dot{\rho}:\dot{\mathbf{V}}'/\mathbf{S} \cong \dot{\mathbf{V}'''} $ is a linear isomorphism. Define morphisms $u'_{i,r-a},v'_{i,r-a},w'_{i,r-a}$ and $x'_{i,r-a}$, $y'_{i,r-a}$ for $a=0,1$, $r \leqslant \nu_{i}$ and $i=1,2,3$ similarly. Notice that $\tilde{\nu}''_{i}=\tilde{\nu}'_{i}$, $\tilde{\nu}'''_{i}=\tilde{\nu}_{i}$, $\dot{\mathbf{E}}_{\mathbf{V}'',\mathbf{W},i}=\dot{\mathbf{E}}_{\mathbf{V}',\mathbf{W},i}$ and $$\dot{\mathbf{E}}_{\mathbf{V},\mathbf{W},i}=\dot{\mathbf{E}}_{\mathbf{V}_{r},\mathbf{W},i}=\dot{\mathbf{E}}_{\mathbf{V}'''_{r-a},\mathbf{W},i}=\dot{\mathbf{E}}_{\mathbf{V}''',\mathbf{W},i},$$ we have the following commutative diagrams
	\[
	\xymatrix{
		\txt{$\dot{\mathbf{E}}_{\mathbf{V}'''_{r-a},\mathbf{W},i}$\\$\times$\\ $\mathbf{Fl}(\nu_{i}-1-r+a,\nu_{i}-r,\tilde{\nu}_{i})$} \ar[d]_-{q_{2,r,a}}	  & Z_{1,r,a} \ar[l]_-{z_{1,r,a}} \ar[r]^{z_{2,r,a}} \ar[d]_{w'_{1,a}} & Z_{2,r,a} \ar[r]^-{z_{3,r,a}} \ar[d]_{w'_{2,a}} & \txt{$\dot{\mathbf{E}}_{\mathbf{V}',\mathbf{W},i}$\\$\times$\\ $\mathbf{Fl}(\nu_{i}''-1,\nu_{i}'',\tilde{\nu}''_{i})$}, \ar[d]^-{q_{2}}  \\
		\txt{$\dot{\mathbf{E}}_{\mathbf{V}'''_{r-a},\mathbf{W},i}$\\$\times$\\ $\mathbf{Gr}(\nu_{i}-1-r+a,\tilde{\nu}_{i})$} & Y'_{1,r-a} \ar[l]_-{y'_{1,r-a}} \ar[r]^{y'_{2,r-a}}& Y'_{2,r-a} \ar[r]^-{y'_{3,r-a}}& \txt{$\dot{\mathbf{E}}_{\mathbf{V}',\mathbf{W},i}$\\$\times$\\ $\mathbf{Gr}(\nu_{i}''-1,\tilde{\nu}''_{i})$} 	,
	}
	\]
	\[
	\xymatrix{
		& 	\txt{$\dot{\mathbf{E}}_{\mathbf{V}'''_{r-a},\mathbf{W},i}$\\$\times$\\ $\mathbf{Gr}(\nu_{i}-1-r+a,\tilde{\nu}_{i})$} & Y'_{1,r-a} \ar[l]_-{y'_{1,r-a}} \ar[r]^{y'_{2,r-a}}& Y'_{2,r-a} \ar[r]^-{y'_{3,r-a}}& \txt{$\dot{\mathbf{E}}_{\mathbf{V}',\mathbf{W},i}$\\$\times$\\ $\mathbf{Gr}(\nu_{i}''-1,\tilde{\nu}''_{i})$} \\
		\mathbf{E}^{0}_{\mathbf{V}'''_{r-a},\mathbf{W},i} \ar[dr]_{\iota_{\mathbf{V}''',r-a}} \ar[ur]^-{\phi_{\mathbf{V}'''_{r-a},i} } & X'_{0,r-a} \ar[l]_{x'_{0,r-a}} \ar[ur]_{\iota'_{\mathbf{V}''',r-a}} \ar[dr]^{v'_{0,r-a}} &  &  & 
		\\
		&\mathbf{E}^{r-a}_{\mathbf{V}''',\mathbf{W},i}
		& X'_{1,r-a} \ar[l]^-{x'_{1,r-a}} \ar[uu]_{v'_{1,r-a}} \ar[r]_{x'_{2,r-a}}  & X'_{2,r-a}  \ar[uu]_{v'_{2,r-a}} \ar[r]_-{x'_{3,r-a}}  &\mathbf{E}^{0}_{\mathbf{V}',\mathbf{W},i} \ar[uu]_{\phi_{\mathbf{V}',i}} ,
	}
	\]
	
	Then by a similar argument as what we have done for equation (\ref{equationA2}), (\ref{sharp})  and (\ref{equationA3}), we have the following isomorphisms
	\begin{equation}\label{equationA4}
		(q_{2})_{!}(z_{3,r,a})_{!}(z_{2,r,a})_{\flat}(z_{1,r,a})^{\ast}(A_{4})\cong (y'_{3,r,a})_{!}(y'_{2,r-a})_{\flat}(y'_{1,r-a})^{\ast}(q_{2,r-a})_{!}(A_{4}),
	\end{equation}
	\begin{equation}\label{equationA5}
		\begin{split}
			&(\phi_{\mathbf{V}',i})^{\ast}(y'_{3,r-a})_{!}(y'_{2,r-a})_{\flat}(y'_{1,r-a})^{\ast}(A_{5})\\ \cong 
			&(x'_{3,r-a})_{!}(x'_{2,r-a})_{\flat}(x'_{1,r-a})^{\ast}(\iota_{\mathbf{V}''',r-a})_{_\clubsuit}(\phi_{\mathbf{V}'''_{r-a},i})^{\ast}(A_{5}),
		\end{split}
	\end{equation}
	\begin{equation}
		\begin{split}
			& (j_{\mathbf{V}',i})_{!}(x'_{3,r-a})_{!}(x'_{2,r-a})_{\flat}(x'_{1,r-a})^{\ast}(A_{6}) \\
			\cong & (p'_{3})_{!}(j'_{3})_{!}(u'_{2,r-a})_{!}(x'_{2,r-a})_{\flat}(x'_{1,r-a})^{\ast}(A_{6}) \\
			\cong & (p'_{3})_{!}(p'_{2})_{\flat}(j'_{2})_{!}(u'_{1,r-a})_{!} (x'_{1,r-a})^{\ast}(A_{6}) \\
			\cong & (p'_{3})_{!}(p'_{2})_{\flat}(j'_{2})_{!}(\tilde{p}'_{1})^{\ast} (j_{\mathbf{V}''',r-a})_{!}(A_{6}) \\
			\cong & (p'_{3})_{!}(p'_{2})_{\flat}(j'_{2})_{!}(j'_{2})^{\ast}(p'_{1})^{\ast} (j_{\mathbf{V}''',r-a})_{!}(A_{6}) .
		\end{split}
	\end{equation}
	Notice that the canonical morphism $(j'_{2})_{!}(j'_{2})^{\ast}(A_{7}) \rightarrow A_{7} $ has a mapping cone whose support is contained in $(p_{3}p_{2})^{-1}(\mathbf{E}^{\geqslant 1}_{\mathbf{V}',\mathbf{W},i} )$, the canonical morphism $$(p'_{3})_{!}(p'_{2})_{\flat}(j'_{2})_{!}(j'_{2})^{\ast}(p'_{1})^{\ast} (j_{\mathbf{V}''',r-a})_{!}(A_{6}) \rightarrow (p'_{3})_{!}(p'_{2})_{\flat}(p'_{1})^{\ast} (j_{\mathbf{V}''',r-a})_{!}(A_{6})$$ has a mapping cone whose support is contained in $\mathbf{E}^{\geqslant 1}_{\mathbf{V}',\mathbf{W},i}$, hence in localization at $i$, we have an isomorphism
	\begin{equation}\label{equationA6}
		(p'_{3})_{!}(p'_{2})_{\flat}(j'_{2})_{!}(j'_{2})^{\ast}(p'_{1})^{\ast} (j_{\mathbf{V}''',r-a})_{!}(A_{6}) \cong (p'_{3})_{!}(p'_{2})_{\flat}(p'_{1})^{\ast} (j_{\mathbf{V}''',r-a})_{!}(A_{6}).
	\end{equation} 
	
	Now for any object $A$, define the following object $$A_{i,r}=  (j_{\mathbf{V}',0})_{!}(\phi_{\mathbf{V}',i})^{\ast}(q_{2})_{!}(q_{1})^{\ast}(\phi_{\mathbf{V}'',i})_{\flat}(x_{3,r})_{!}(x_{2,r})_{\flat}(x_{1,r})^{\ast}(j_{\mathbf{V},r})^{\ast}(A),$$   
	$$A_{i,r,a}=  (j_{\mathbf{V}',0})_{!}(\phi_{\mathbf{V}',i})^{\ast}(q_{2})_{!}(z_{3,r,a})_{!}(z_{2,r,a})_{\flat}(z_{1,r,a})^{\ast}(q_{1,r,a})^{\ast}(\phi_{\mathbf{V}_{r},i})_{\flat}(\iota_{\mathbf{V},r})^{_\clubsuit}(j_{\mathbf{V},r})^{\ast}(A).$$
	
	Since $$\mathcal{E}_{i}\mathcal{F}^{(n),\vee}_{j}(A)=  (j_{\mathbf{V}',0})_{!}(\phi_{\mathbf{V}',i})^{\ast}(q_{2})_{!}(q_{1})^{\ast}(\phi_{\mathbf{V}'',i})_{\flat} (j_{\mathbf{V}'',0})^{\ast}(p_{3})_{!}(p_{2})_{\flat}(p_{1})^{\ast}(A),$$   
	then with the canonical distinguished triangles induced by the partition $X_{1,\geqslant r}= X_{1,r} \cup X_{1,\geqslant r+1}$,  we can consider the following commutative diagram and see that $\mathcal{E}_{i}\mathcal{F}^{(n),\vee}_{j}(A)$ is inductively generated by $A_{i,r}, 0 \leqslant r \leqslant \nu_{i}$.  Using equations (\ref{equationA2}), (\ref{equationA3}) and the canonical distinguished triangles induced by $Z_{1,r}= Z_{1,r,0} \cup Z_{1,r,1}$, we can see that $A_{i,r}$ is generated by $A_{i,r,a}, a=0,1$. Using equations
	(\ref{equationA4}), (\ref{equationA5}) and (\ref{equationA6}), we can see that $A_{i,r,a} \cong \mathcal{F}^{(n),\vee}_{j}\mathcal{E}_{i,r,a}(A)$ in localization. The lemma is proved.
	\[
	\xymatrix{
		\mathbf{E}_{\mathbf{V},\mathbf{W},\hat{\Omega}} & \mathbf{E}^{'}_{\mathbf{V}'',\mathbf{W},\hat{\Omega}}\ar[l]_{p_{1}} \ar[r]^{p_{2}}& \mathbf{E}^{''}_{\mathbf{V}'',\mathbf{W},\hat{\Omega}} \ar[r]^{p_{3}}& \mathbf{E}_{\mathbf{V}'',\mathbf{W},\hat{\Omega}} \\
		\mathbf{E}^{r}_{\mathbf{V},\mathbf{W},i} \ar[u]^{j_{\mathbf{V},r}} & X_{1,r} \ar[dd]_{v_{1,r}} \ar[l]_{x_{1,r}} \ar[r]^{x_{2,r}} \ar[u]_{j_{2}u_{1,r}} & X_{2,r} \ar[dd]_{v_{2,r}} \ar[r]^{x_{3,r}} \ar[u]_{j_{3}u_{2,r}} & \mathbf{E}^{0}_{\mathbf{V}'',\mathbf{W},i} \ar[dd]_{\phi_{\mathbf{V}'',i}} \ar[u]_{j_{\mathbf{V}'',0}} \\
		\mathbf{E}^{0}_{\mathbf{V}_{r},\mathbf{W},i} \ar[u]^{\iota_{\mathbf{V},r}} \ar[d]_{\phi_{\mathbf{V}_{r},i}}	& & & \\
		\txt{$\dot{\mathbf{E}}_{\mathbf{V}_{r},\mathbf{W},i}$\\$\times$\\ $\mathbf{Gr}(\nu_{i}-r,\tilde{\nu}_{i})$} & Y_{1,r} \ar[l]_-{y_{1,r}} \ar[r]^{y_{2,r}}& Y_{2,r} \ar[r]^-{y_{3,r}}& \txt{$\dot{\mathbf{E}}_{\mathbf{V}'',\mathbf{W},i}$\\$\times$\\ $\mathbf{Gr}(\nu_{i}'',\tilde{\nu}''_{i})$} \\
		\txt{$\dot{\mathbf{E}}_{\mathbf{V}_{r},\mathbf{W},i}$\\$\times$\\ $\mathbf{Fl}(\nu_{i}-1-r+a,\nu_{i}-r,\tilde{\nu}_{i})$} \ar[u]^-{q_{1,r,a}} \ar[d]_-{q_{2,r,a}}	  & Z_{1,r,a} \ar[l]_-{z_{1,r,a}} \ar[r]^{z_{2,r,a}} \ar[u]_{w_{1,a}} \ar[d]^{w'_{1,r-a}} & Z_{2,r,a} \ar[r]^-{z_{3,r,a}} \ar[u]_{w_{2,a}} \ar[d]^{w'_{2,r-a}} & \txt{$\dot{\mathbf{E}}_{\mathbf{V}'',\mathbf{W},i}$\\$\times$\\ $\mathbf{Fl}(\nu_{i}''-1,\nu_{i}'',\tilde{\nu}''_{i})$} \ar[u]_-{q_{1}} \ar[d]^-{q_{2}}\\
		\txt{$\dot{\mathbf{E}}_{\mathbf{V}'''_{r-a},\mathbf{W},i}$\\$\times$\\ $\mathbf{Gr}(\nu_{i}-1-r+a,\tilde{\nu}_{i})$} & Y'_{1,r-a} \ar[l]_-{y'_{1,r-a}} \ar[r]^{y'_{2,r-a}}& Y'_{2,r-a} \ar[r]^-{y'_{3,r-a}}& \txt{$\dot{\mathbf{E}}_{\mathbf{V}',\mathbf{W},i}$\\$\times$\\ $\mathbf{Gr}(\nu_{i}''-1,\tilde{\nu}''_{i})$} 	\\
		\mathbf{E}^{0}_{\mathbf{V}'''_{r-a},\mathbf{W},i} \ar[u]^-{\phi_{\mathbf{V}'''_{r-a},i}}	 \ar[d]_-{\iota_{\mathbf{V}''',r-a}} & & & \\
		\mathbf{E}^{r-a}_{\mathbf{V}''',\mathbf{W},i} \ar[d]_-{j_{\mathbf{V}''',r-a}}
		& X'_{1,r-a}  \ar[l]^-{x'_{1,r-a}} \ar[uu]_{v'_{1,r-a}} \ar[r]_{x'_{2,r-a}}  \ar[d]^-{j'_{2}u'_{1,r-a}} & X'_{2,r-a}  \ar[uu]_{v'_{2,r-a}} \ar[r]_-{x'_{3,r-a}} \ar[d]^-{j'_{3}u'_{2,r-a}} &\mathbf{E}^{0}_{\mathbf{V}',\mathbf{W},i} \ar[uu]_{\phi_{\mathbf{V}',i}} \ar[d]^{j_{\mathbf{V}',0}} \\
		\mathbf{E}_{\mathbf{V}''',\mathbf{W},\hat{\Omega}} & \mathbf{E}^{'}_{\mathbf{V}',\mathbf{W},\hat{\Omega}}\ar[l]_{p'_{1}} \ar[r]^{p'_{2}}& \mathbf{E}^{''}_{\mathbf{V}',\mathbf{W},\hat{\Omega}} \ar[r]^{p'_{3}}& \mathbf{E}_{\mathbf{V}',\mathbf{W},\hat{\Omega}}.
	}
	\]

\end{proof}

The following lemma is an analogy of Lemma \ref{rkey'} and Proposition \ref{rt'} for  objects which are not semisimple.
\begin{lemma}\label{key0}
	For graded space $\mathbf{V}$ and each $1\leqslant r \leqslant \nu_{j}$, we fix a decomposition $\mathbf{V}=\mathbf{V}_{r}\oplus \mathbf{S}_{r}$ of $\mathbf{V}$ such that $|\mathbf{S}_{r}|=rj$. For any object $A$ in $\mathcal{N}_{\mathbf{V},j}$,  we can find some object  $B_{r}$ in $\mathcal{D}^{b}_{G_{\mathbf{V}}}(\mathbf{E}_{\mathbf{V}_{r},\mathbf{W},\hat{\Omega}})$  for each $1\leqslant r \leqslant \nu_{j}$ such that those $\mathcal{F}^{(r),\vee}_{j}(B_{r})$ generate $A$ via distinguished triangles. More precisely, $A$ in $\mathcal{N}_{\mathbf{V},j}$ if and only if $A$ can be generated by finite objects of the form $\mathcal{F}^{(n),\vee}_{j}(B)$.
\end{lemma}
\begin{proof}
	By Remark \ref{remarkFD}, $\mathcal{F}^{(n),\vee}_{j}$ commutes with Fourier-Deligne transforms, thus we can assume $j$ is a source in $\hat{\Omega}$.
	Let $F_{r}$ be the closed subset of $\mathbf{E}_{\mathbf{V},\mathbf{W},\hat{\Omega}}$ consisting of $x$ such that $\mathbf{S}_{r}$ is $x$-stable, $Q_{r}$ be the stabilizer of $\mathbf{S}_{r}$ in $G_{\mathbf{V}}$ and $U_{r}$ be its unipotent radical. Notice that $E''_{\hat{\Omega}}\cong G_{\mathbf{V}} \times_{Q_{r}} F_{r}$, $E'_{\hat{\Omega}}\cong G_{\mathbf{V}} \times_{U_{r}} F_{r}$,  $F_{r} \cong \mathbf{E}_{\mathbf{V}_{r},\mathbf{W},\hat{\Omega}}$ and $p_{3}$ induces $ G_{\mathbf{V}} \times_{Q_{r}} \mathbf{E}^{0}_{\mathbf{V}_{r},\mathbf{W},j} \cong \mathbf{E}^{r}_{\mathbf{V},\mathbf{W},j}$. 
	Consider the following commutative diagram
	\[
	\xymatrix{
		G_{\mathbf{V}} \times_{U_{r}} F_{r} \ar@<0.5ex>[d]^{p_{1}} \ar[dr]_{p_{2}} &    \\
		F_{r}  \ar@<0.5ex>[u]^{i_{r}} \ar[r]^{\iota_{r}} &  G_{\mathbf{V}} \times_{Q_{r}} F_{r}   ,
	}
	\]
	where $i_{r},\iota_{r}$ are the natural embeddings. Then the induction functor $\mathcal{F}^{(r),\vee}_{j}$ is isomorphic to $(p_{3})_{!}(\iota_{r})_{_{\clubsuit}}$ up to shifts. In fact, $(p_{3})_{!}(\iota_{r})_{_{\clubsuit}}$ is exactly the definition of the induction functor  in \cite[10.3.5]{Achar}. Consider the following commutative diagram for $\mathcal{F}^{(r),\vee}_{j}$
	\[
	\xymatrix{
		F_{r} \ar[r]^{\iota_{r}} & G_{\mathbf{V}} \times_{Q_{r}} F_{r} \ar[r]^{p_{3}}&  \mathbf{E}_{\mathbf{V},\mathbf{W},\hat{\Omega}} \\
		\mathbf{E}^{0}_{\mathbf{V}_{r},\mathbf{W},j} \ar[u]^{i_{1,r}} \ar[r]^-{\iota'_{r}} &  G_{\mathbf{V}} \times_{Q_{r}} \mathbf{E}^{0}_{\mathbf{V}_{r},\mathbf{W},j}  \ar[r]^-{\cong} \ar[u]_{i_{2,r}}  & \mathbf{E}^{r}_{\mathbf{V},\mathbf{W},j} \ar[u]_{i_{3,r}},
	}
	\]
	where $i_{1},i_{2},i_{3}$ are the obvious embeddings and $\iota'_{r}$ is the restriction of $\iota_{r}$. Notice that the restriction functor is defined by $(p_{3}\iota_{r})^{\ast}$ up to shifts, then for any object $C$, we have the following isomorphism up to shifts
	$$\mathcal{F}^{(r),\vee}_{j} (i_{1,r})_{!}(i_{1,r})^{\ast} \mathbf{Res}^{\mathbf{V}\oplus\mathbf{W}}_{\mathbf{V}_{r}\oplus \mathbf{W},\mathbf{S}_{r}}(C) \cong (i_{3,r})_{!}(i_{3,r})^{\ast}(C). $$
	(In fact, similar isomorphisms for semisimple objects have been also used by Lusztig in the proof of  \cite[Proposition 6.6]{MR1088333}.)
	
	Let $B_{1}= \mathcal{F}^{(1),\vee}_{j} (i_{1,1})_{!}(i_{1,1})^{\ast} \mathbf{Res}^{\mathbf{V}\oplus\mathbf{W}}_{\mathbf{V}_{1}\oplus \mathbf{W},\mathbf{S}_{1}}(A) $. Since the support of $A$ is contained in $\mathbf{E}^{\geqslant 1}_{\mathbf{V},\mathbf{W},j}$,  there is a canonical distinguished triangle induced by $\mathbf{E}^{\geqslant 1}_{\mathbf{V},\mathbf{W},j}=\mathbf{E}^{ 1}_{\mathbf{V},\mathbf{W},j} \cup \mathbf{E}^{\geqslant 2}_{\mathbf{V},\mathbf{W},j}$  
	$$ B_{1} \rightarrow A \rightarrow B_{\geqslant 1} \xrightarrow{+1},$$
	where $B_{\geqslant 1}$ is an object whose support is contained  in $\mathbf{E}^{\geqslant 2}_{\mathbf{V},\mathbf{W},j}$. 
	
	By a similar argument as in the proof of Lemma \ref{commute4}, we can inductively define $B_{r}$ be the object $ \mathcal{F}^{(r),\vee}_{j} (i_{1,r})_{!}(i_{1,r})^{\ast} \mathbf{Res}^{\mathbf{V}\oplus\mathbf{W}}_{\mathbf{V}_{r}\oplus \mathbf{W},\mathbf{S}_{r}}(B_{\geqslant r-1}) $ and define $B_{\geqslant r}$ be the mapping cone of the  canonical morphism $B_{r} \rightarrow B_{\geqslant r-1}$  induced by the partition $\mathbf{E}^{\geqslant r}_{\mathbf{V},\mathbf{W},j}=\mathbf{E}^{ r}_{\mathbf{V},\mathbf{W},j} \cup \mathbf{E}^{\geqslant r+1}_{\mathbf{V},\mathbf{W},j}$. When $r > \nu_{j}$, $\mathbf{E}^{\geqslant r}_{\mathbf{V},\mathbf{W},j}$  is empty and  $B_{\geqslant r}$ is zero.   These $B_{r},1\leqslant r \leqslant \nu_{j}$ are exactly what we need. 
\end{proof}

\begin{proposition}\label{keypro}
	The functor  $\mathcal{E}^{(n)}_{i}:\mathcal{D}^{b}_{G_{\mathbf{V}}}(\mathbf{E}_{\mathbf{V},\mathbf{W},\hat{\Omega} })/\mathcal{N}_{\mathbf{V},i} \rightarrow \mathcal{D}^{b}_{G_{\mathbf{V}'}}(\mathbf{E}_{\mathbf{V}',\mathbf{W},\hat{\Omega}})/\mathcal{N}_{\mathbf{V}',i}$ sends objects of  $\mathcal{N}_{\mathbf{V},j}$ to objects of $\mathcal{N}_{\mathbf{V}',j}$ up to isomorphisms in localization at $i$. In particular,  it induces a functor  $\mathcal{E}^{(n)}_{i}:\mathcal{D}^{b}_{G_{\mathbf{V}}}(\mathbf{E}_{\mathbf{V},\mathbf{W},\hat{\Omega} })/\mathcal{N}_{\mathbf{V}} \rightarrow \mathcal{D}^{b}_{G_{\mathbf{V}'}}(\mathbf{E}_{\mathbf{V}',\mathbf{W},\hat{\Omega}})/\mathcal{N}_{\mathbf{V}'}$ between global localizations.
\end{proposition}

\begin{proof}
	Firstly, we prove the statement for $n=1$. Since $\mathcal{N}_{\mathbf{V},j}$ is a thick subcategory generated by objects of the form $\mathcal{F}^{(m),\vee}_{j}(B)$, we only need to check that,  $\mathcal{E}_{i}\mathcal{F}^{(m),\vee}_{j}(B)$ is isomorphic to an object of  $\mathcal{N}_{\mathbf{V}',j}$ in $\mathcal{D}^{b}_{G_{\mathbf{V}'}}(\mathbf{E}_{\mathbf{V}',\mathbf{W},\hat{\Omega}})/\mathcal{N}_{\mathbf{V}',i}$.  However, $\mathcal{E}_{i}\mathcal{F}^{(m),\vee}_{j}(B)$ is generated by those $\mathcal{F}^{(m),\vee}_{j}\mathcal{E}_{i,r,a}(B)$. By Remark \ref{remarkFD}, $\mathcal{F}^{(m),\vee}_{j}$ commutes with Fourier-Deligne transforms,  $\mathcal{F}_{\hat{\Omega}^{i},\hat{\Omega}^{j}}\mathcal{F}^{(m),\vee}_{j}\mathcal{E}_{i,r,a}(B) \cong \mathcal{F}^{(m),\vee}_{j}\mathcal{F}_{\hat{\Omega}^{i},\hat{\Omega}^{j}}\mathcal{E}_{i,r,a}(B)$ belongs to $\mathcal{N}_{\mathbf{V}',j}$. Therefore $\mathcal{F}^{(m),\vee}_{j}\mathcal{E}_{i,r,a}(B)$ belongs to $\mathcal{N}_{\mathbf{V}',j}$ and the statement holds for $n=1$. By Lemma \ref{Lemma 16}, we can inductively prove that $\mathcal{E}^{(n)}_{i}(A)$ is a direct summand of $(\mathcal{E}_{i})^{n}(A)$, hence the statement holds for general $n$. 
\end{proof}
In fact, one can similarly define $\mathcal{E}^{(n)}_{i,r,a}$ for $0 \leqslant a \leqslant n$ and prove that those $\mathcal{F}^{(m),\vee}_{j}\mathcal{E}^{(n)}_{i,r,a}(A)$ generate $\mathcal{E}_{i}^{(n)}\mathcal{F}^{(m),\vee}_{j}(A)$ as what we have done in the proof of Lemma \ref{commute4}. Then the corollary above holds for general $n$ without using Lemma \ref{Lemma 16}.

\begin{corollary}\label{corollary tilde E}
	The functor  $\mathcal{E}^{(n)}_{i}:\mathcal{D}^{b}_{G_{\mathbf{V}}}(\mathbf{E}_{\mathbf{V},\mathbf{W},\hat{\Omega} })/\mathcal{N}_{\mathbf{V}} \rightarrow \mathcal{D}^{b}_{G_{\mathbf{V}'}}(\mathbf{E}_{\mathbf{V}',\mathbf{W},\hat{\Omega}})/\mathcal{N}_{\mathbf{V}'}$ restricts to a functor $\mathcal{E}^{(n)}_{i}:  \mathcal{Q}_{\mathbf{V},\mathbf{W}}/\mathcal{N}_{\mathbf{V}}\rightarrow \mathcal{Q}_{\mathbf{V}',\mathbf{W}}/\mathcal{N}_{\mathbf{V}'}$.
\end{corollary}
\begin{proof}
	It suffices to prove that $\mathcal{E}^{(n)}_{i}(L_{\boldsymbol{\nu}\boldsymbol{d}} ) $ is isomorphic to a direct sum of some $L_{\boldsymbol{\nu}'\boldsymbol{d}}$. We argue by induction on the length of $\boldsymbol{\nu}$ and $n$. 
	
	Without loss of generality, we can replace the flag type $\boldsymbol{\nu}=( i_{1}^{a_{1}},i_{2}^{a_{2}},\cdots ,i_{k}^{a_{k}})$ by $$(i_{1},\cdots,i_{1},i_{2},\cdots,i_{2},\cdots,i_{k},\cdots,i_{k})$$ such that each $i_{l}$ appears repeatedly for $a_{l}$ times for $1\leqslant l\leqslant k$, then $L_{\boldsymbol{\nu}\boldsymbol{d}}=\mathcal{F}_{i_{1}} L_{\boldsymbol{\nu}'\boldsymbol{d}}$ for $\boldsymbol{\nu}'= (i_{1},i_{1},\cdots,i_{k})$ such that $i_{1}$ appears for $a_{1}-1$ times and the other $i_{l}$ appear repeatedly for $a_{l}$ times for $1<l\leqslant k$. 
	Assume $|\mathbf{V}'|=|\mathbf{V}|-ni,|\mathbf{V}''|=|\mathbf{V}|-i_{1}$ and $|\mathbf{V}'''|=|\mathbf{V}|-ni-i_{1}$, then up to shifts,
	\begin{equation*}
		\begin{split}
			\mathcal{E}^{(n)}_{i} L_{\boldsymbol{\nu}\boldsymbol{d}}  
			\cong  \mathcal{E}^{(n)}_{i}\mathcal{F}_{i_{1}} L_{\boldsymbol{\nu}'\boldsymbol{d}}.
		\end{split}
	\end{equation*}
	
	If $i_{1} \neq i$, by Lemma \ref{lemma c2},  $\mathcal{E}^{(n)}_{i}	L_{\boldsymbol{\nu}\boldsymbol{d}}  \cong   \mathcal{F}_{i_{1}} \mathcal{E}^{(n)}_{i}L_{\boldsymbol{\nu}'\boldsymbol{d}}$. Otherwise, by Lemma \ref{lemma c1}, $\mathcal{E}^{(n)}_{i}	 L_{\boldsymbol{\nu}\boldsymbol{d}}$ is either isomorphic to the direct sum of  $\mathcal{F}_{i_{1}} \mathcal{E}^{(n)}_{i} L_{\boldsymbol{\nu}'\boldsymbol{d}}$ and some shifts of $\mathcal{E}^{(n-1)}_{i} L_{\boldsymbol{\nu}'\boldsymbol{d}}$ or isomorphic to a direct summand of $\mathcal{F}_{i_{1}} \mathcal{E}^{(n)}_{i} L_{\boldsymbol{\nu}'\boldsymbol{d}}$. By inductive hypothesis and the fact that $\mathcal{F}_{i_{1}}$ sends an object of $\mathcal{Q}_{\mathbf{V}''',\mathbf{W}}$ to an object of $\mathcal{Q}_{\mathbf{V}',\mathbf{W}}$, we finish the proof.
\end{proof}

\begin{corollary}\label{commute3}
	For $i \neq j$ and graded spaces such that $|\mathbf{V}'|+ni=|\mathbf{V}|+j$,  there is an isomorphism of functors between global localizations  $\mathcal{Q}_{\mathbf{V},\mathbf{W}}/\mathcal{N}_{\mathbf{V}} \rightarrow \mathcal{Q}_{\mathbf{V}',\mathbf{W}}/\mathcal{N}_{\mathbf{V}'}$
	\begin{equation*}
		\mathcal{E}^{(n)}_{i}\mathcal{F}_{j}\cong \mathcal{F}_{j}\mathcal{E}^{(n)}_{i}.
	\end{equation*}
	For $i=j$ and graded spaces such that $|\mathbf{V}'|+(n-1)i=|\mathbf{V}|$, let $N=2\nu_{i}-\tilde{\nu}_{i}-n+1$, then there is an isomorphism of functors between global localizations $\mathcal{Q}_{\mathbf{V},\mathbf{W}}/\mathcal{N}_{\mathbf{V}} \rightarrow \mathcal{Q}_{\mathbf{V}',\mathbf{W}}/\mathcal{N}_{\mathbf{V}'}$
	\begin{equation*}
		\mathcal{E}^{(n)}_{i}\mathcal{F}_{i} \oplus \bigoplus\limits_{0\leqslant m \leqslant N-1} \mathcal{E}^{(n-1)}_{i} [N-1-2m] \cong \mathcal{F}_{i}\mathcal{E}^{(n)}_{i} \oplus \bigoplus\limits_{0\leqslant m \leqslant -N-1} \mathcal{E}^{(n-1)}_{i} [-2m-N-1].
	\end{equation*}
	More precisely, in this case we have
	\begin{align*}
		&{\mathcal{E}}^{(n)}_{i}{\mathcal{F}}_{i}\cong {\mathcal{F}}_{i}{\mathcal{E}}^{(n)}_{i}, &\textrm{if}\ N=0;\\
		&{\mathcal{E}}^{(n)}_{i}{\mathcal{F}}_{i}\oplus\bigoplus_{0\leqslant m\leqslant N-1}\mathcal{E}^{(n-1)}_{i} [N-1-2m]\cong {\mathcal{F}}_{i}{\mathcal{E}}^{(n)}_{i},&\textrm{if}\  N\geqslant 1;\\
		&{\mathcal{E}}^{(n)}_{i}\mathcal{F}_{i}\cong {\mathcal{F}}_{i}{\mathcal{E}}^{(n)}_{i}\oplus \bigoplus\limits_{0\leqslant m \leqslant -N-1} \mathcal{E}^{(n-1)}_{i} [-2m-N-1], &\textrm{if}\ N\leqslant -1.
	\end{align*}
\end{corollary}

\subsection{The integrable highest weight modules}

In this subsection, we fix an orientation $\Omega$.
\begin{definition}	
	For $n\in \mathbb{N}$ and $i\in I$,  define functors between $\mathcal{L}_{\mathbf{V}}(\Lambda)=\mathcal{Q}_{\mathbf{V},\mathbf{W}}/\mathcal{N}_{\mathbf{V}}$ as follows.
	
	(1) For graded spaces $\mathbf{V}$ and $\mathbf{V}'$ such that $|\mathbf{V}|=|\mathbf{V}'|+ni$, define $$E^{(n)}_{i}:\mathcal{L}_{\mathbf{V}}(\Lambda) \rightarrow \mathcal{L}_{\mathbf{V}'}(\Lambda)$$ 
	\begin{equation*}
		E^{(n)}_{i}=\mathcal{F}_{\hat{\Omega}^{i},\hat{\Omega}}\mathcal{E}^{(n)}_{i}\mathcal{F}_{\hat{\Omega},\hat{\Omega}^{i}}.
	\end{equation*}
	
	(2) For graded spaces $\mathbf{V}$ and $\mathbf{V}''$ such that $|\mathbf{V}|+ni=|\mathbf{V}''|$, define $$F^{(n)}_{i}:\mathcal{L}_{\mathbf{V}}(\Lambda) \rightarrow \mathcal{L}_{\mathbf{V}''}(\Lambda)$$ 
	\begin{equation*}
		F^{(n)}_{i}=\mathcal{F}_i^{(n)}.
	\end{equation*}
	
	(3) Define $K_{i}: \mathcal{L}_{\mathbf{V}}(\Lambda) \rightarrow \mathcal{L}_{\mathbf{V}}(\Lambda)$ via
	\begin{equation*}
		K_{i}=\textrm{Id}\ [\tilde{\nu}_{i}-2\nu_{i}].
	\end{equation*}
	In particular, we denote by $E_{i}=E^{(1)}_{i}$ and $F_{i}=F^{(1)}_{i}$. Note that $K_{i}$ is invertible, and we define $K^{-}_{i}=\textrm{Id}\ [2\nu_{i}-\tilde{\nu}_{i}] $ to be its inverse.
\end{definition}

\begin{remark}
	Even though we need to choose an orientation $\Omega^{i}$ to define $\mathcal{E}^{(n)}_{i}$ for each $i$, it's easy to see that the definition of functor $E^{(n)}_{i}, i \in I$ does not rely on the choice of $\hat{\Omega}^{i}$ and  $\hat{\Omega}$. 
\end{remark}

By definitions and Corollary \ref{commute3}, we obtain the following proposition.
\begin{proposition}\label{relation1}
	Let $C=(c_{i,j})_{ i,j \in I}$ the Cartan matrix of $Q$, then functors $E_{i}$, $F_{i}$ and $K_{i},i\in I$ satisfy the following relations
	\begin{equation*} 
		K_{i}K_{j}=K_{j}K_{i},
	\end{equation*}
	\begin{equation*}
		E_{i}K_{j}=K_{j}E_{i}[-c_{j,i}],
	\end{equation*}
	\begin{equation*}
		F_{i}K_{j}=K_{j}F_{i}[c_{i,j}],
	\end{equation*}
	\begin{equation*}
		E_{i}F_{j}=F_{j}E_{i}\ \textrm{for}\ i \neq j,
	\end{equation*}
	\begin{equation*}
		E_{i}F_{i} \oplus \bigoplus\limits_{0\leqslant m \leqslant N-1} Id[N-1-2m] \cong F_{i}E_{i} \oplus \bigoplus\limits_{0\leqslant m \leqslant -N-1} Id[-2m-N-1]: \mathcal{L}_{\mathbf{V}}(\Lambda) \rightarrow \mathcal{L}_{\mathbf{V}}(\Lambda),
	\end{equation*}
	as endofunctors of $\mathcal{L}(\Lambda)= \coprod \limits_{\mathbf{V}} \mathcal{L}_{\mathbf{V}}(\Lambda)$, where $N=-\tilde{\nu}_{i}+2\nu_{i}$.
\end{proposition}

By Theorem \ref{Lusztig1}, we also have the following proposition.
\begin{proposition}\label{relation2}
	The functors $F_{i}$ for $i\in I$ satisfy the following relations
	\begin{equation*}
		\bigoplus\limits_{0\leqslant m \leqslant 1+ a_{i,j},m\ is\ odd}F^{(m)}_{i}F_{j}F^{(1+a_{i,j}-m)}_{i}\cong 	\bigoplus\limits_{0\leqslant m \leqslant 1+ a_{i,j},m\ is \ even}F^{(m)}_{i}F_{j}F^{(1+a_{i,j}-m)}_{i},
	\end{equation*}
	\begin{equation*}
		\bigoplus \limits_{0 \leqslant m < n } F^{(n)}_{i}[n-1-2m] \cong F^{(n-1)}_{i}F_{i}\cong F_{i} F^{(n-1)}_{i}\ \textrm{for}\ n \geqslant 2,
	\end{equation*}
	as endofunctors of $\mathcal{L}(\Lambda)= \coprod \limits_{\mathbf{V}} \mathcal{L}_{\mathbf{V}}(\Lambda)$, where $a_{i,j}=a_{i,j}$ is the number of arrows between $i$ and $j$ when $i\neq j$ and $a_{i,i}=-2$ for any $i \in I$.
\end{proposition}
\begin{proof}
	Notice that for $\boldsymbol{\nu}=(i^{a_{1}}_{1}, \cdots i^{a_{k}}_{k})$, by definition we have $$ F^{(a_{1})}_{i_{1}} F^{(a_{2})}_{i_{2}}  \cdots  F^{(a_{k})}_{i_{k}} =\mathbf{Ind}^{\mathbf{V}\oplus\mathbf{W}}_{\mathbf{V}',\mathbf{V}''\oplus \mathbf{W}}(L_{\boldsymbol{\nu}}\boxtimes -).$$ Then the proposition follows from Theorem \ref{Lusztig1}.
\end{proof}

\begin{proposition}\label{relation3}
	The functors  $E_{i}$ for $i\in I$ satisfy the following relations
	\begin{equation*}
		\bigoplus\limits_{0\leqslant m \leqslant 1+ a_{i,j},m~odd}E^{(m)}_{i}E_{j}E^{(1+a_{i,j}-m)}_{i} \cong 	\bigoplus\limits_{0\leqslant m \leqslant 1+ a_{i,j},m~even}E^{(m)}_{i}E_{j}E^{(1+a_{i,j}-m)}_{i},
	\end{equation*}
	\begin{equation*}
		\bigoplus \limits_{0 \leqslant m < n } E^{(n)}_{i}[n-1-2m] \cong E^{(n-1)}_{i}E_{i}\cong E_{i}E^{(n-1)}_{i}, n \geqslant 2,
	\end{equation*}
	as endofunctors of $\mathcal{L}(\Lambda)= \coprod \limits_{\mathbf{V}} \mathcal{L}_{\mathbf{V}}(\Lambda)$.
\end{proposition}
\begin{proof}
	The second isomorphism follows from Lemma \ref{Lemma 16}. We only need to prove the first one. From the definition of $\mathcal{E}^{(n)}_{i}$ and the description of $\mathcal{F}^{(n)}_{i}$ in equation (\ref{equation 2}), we can see that the functor $$\mathcal{E}^{(n)}_{i}: \mathcal{D}^{b}_{G_{\mathbf{V}}}(\mathbf{E}_{\mathbf{V},\mathbf{W},\hat{\Omega}})/\mathcal{N}_{\mathbf{V},i}\rightarrow \mathcal{D}^{b}_{G_{\mathbf{V}'}}(\mathbf{E}_{\mathbf{V}',\mathbf{W},\hat{\Omega}})/\mathcal{N}_{\mathbf{V}',i}$$ and $$\mathcal{F}^{(n)}_{i}: \mathcal{D}^{b}_{G_{\mathbf{V}'}}(\mathbf{E}_{\mathbf{V}',\mathbf{W},\hat{\Omega}})/\mathcal{N}_{\mathbf{V}',i}\rightarrow \mathcal{D}^{b}_{G_{\mathbf{V}}}(\mathbf{E}_{\mathbf{V},\mathbf{W},\hat{\Omega}})/\mathcal{N}_{\mathbf{V},i}$$ are compositions of some quasi-inverse equivalences and the pull-back or push-forward of the  following  morphisms
	\begin{center}
		$\dot{\mathbf{E}}_{\mathbf{V},\mathbf{W},i} \times \mathbf{Gr}(\nu_{i},\tilde{\nu}_{i}) \xleftarrow{q_{1}} \dot{\mathbf{E}}_{\mathbf{V},\mathbf{W},i} \times \mathbf{Fl}(\nu_{i}-n,\nu_{i},\tilde{\nu}_{i}) \xrightarrow{q_{2}} \dot{\mathbf{E}}_{\mathbf{V}',\mathbf{W},i} \times \mathbf{Gr}(\nu_{i}-n,\tilde{\nu}_{i}) .$
	\end{center}
	Since $q_{1},q_{2}$ are proper and smooth with relative dimension $n(\nu_{i}-n)$ and $n(\tilde{\nu}_{i}-\nu_{i}-n)$ respectively, we can see that $\mathcal{F}^{(n)}_{i}$ has a right adjoint $\mathcal{E}^{(n)}_{i}[n(2\nu_{i}-\tilde{\nu}_{i}-n)]$. Since the localization preserves the adjointness, the Serre relation of  $E^{(n)}_{i}$ follows from that of   $F^{(n)}_{i}$. 
\end{proof}

\begin{proposition}
	The functors $F^{(n)}_{i},E^{(n)}_{i},K_{i}$ for $n\in \mathbb{N}, i \in I$ and Verdier duality functor $\mathbf{D}$ satisfy the following relations
	\begin{equation*}
		F^{(n)}_{i}\mathbf{D}\cong\mathbf{D}F^{(n)}_{i},
	\end{equation*}
	\begin{equation*}
		E^{(n)}_{i}\mathbf{D}\cong\mathbf{D}E^{(n)}_{i},
	\end{equation*}
	\begin{equation*}
		K_{i}\mathbf{D}\cong\mathbf{D}(K_{i})^{-1}.
	\end{equation*}
\end{proposition}
\begin{proof}
	The first relation holds, since the induction functors commute with the  Verdier duality functor. The last relation can be easily checked by definition. We only prove the second relation for $n=1$, the other case can be proved by a similar argument. Notice that $E_{i}$ can be written as 
	\begin{equation*}
		E_{i}=(j_{\mathbf{V}',i})_{!} ((\phi_{\mathbf{V}',i})^{\ast}[(\nu'_{i})^{2}]) (q_{2})_{!}((q_{1})^{\ast}[\nu_{i}-1]) ((\phi_{\mathbf{V},i})_{\flat}[-\nu^{2}_{i}])(j_{\mathbf{V},i})^{\ast}.
	\end{equation*}
	Since $q_{2}$ is proper, $(q_{2})_{!}$ commutes with $\mathbf{D}$. Note that $$(j_{\mathbf{V}',i})_{!} \cong (j_{\mathbf{V}',i})_{\ast}:\mathcal{D}^{b}_{G_{\mathbf{V}}}(\mathbf{E}^{0}_{\mathbf{V},\mathbf{W},i}) \rightarrow \mathcal{D}^{b}_{G_{\mathbf{V}}}(\mathbf{E}_{\mathbf{V},\mathbf{W},\hat{\Omega}}) /\mathcal{N}_{\mathbf{V},i},$$ we can see that $(j_{\mathbf{V}',i})_{!}$ also commutes with $\mathbf{D}$. Since each functor in the expression commutes with the Verdier duality functor, so does $E_{i}$. 
\end{proof}
\begin{definition}
	Define $\mathcal{K}_{0}(\Lambda)=\mathcal{K}_{0}(\mathcal{L}(\Lambda))$ to be the Grothendieck group of $\mathcal{L}(\Lambda)$, which can be endowed with an $\mathcal{A}$-module structure. More precisely, $\mathcal{K}_{0}(\Lambda)$ is the $\mathcal{A}$-module spanned by objects $[L]$ in $\mathcal{L}(\Lambda)$ subject to relations
	\begin{equation*}
		[X \oplus Y]=[X]+[Y],
	\end{equation*}
	\begin{equation*}
		[X[1]]=v[X].
	\end{equation*}
	Similarly, we denote the Grothendieck group of $\mathcal{L}_{\mathbf{V}}(\Lambda)$ by $\mathcal{K}_{0,|\mathbf{V}|}(\Lambda)$.
	
	The functors $E^{(n)}_{i},F^{(n)}_{i},K^{\pm}_{i}$ for $n\in \mathbb{N},i \in I$ induces $\mathcal{A}$-linear operators on $\mathcal{K}_{0}(\Lambda)$, and we still denote these operators by $E^{(n)}_{i},F^{(n)}_{i},K^{\pm}_{i}$ for $n\in \mathbb{N},i \in I$, respectively.
\end{definition}

\begin{theorem}\label{thm1}
	The linear operators induced by functors $E^{(n)}_{i},F^{(n)}_{i},K^{\pm}_{i}$ for $n\in \mathbb{N}$ and $i \in I$ defines a $_{\mathcal{A}}\mathbf{U}$-module structure on $\mathcal{K}_{0}(\Lambda)$ which is isomorphic to the integrable highest weight $_{\mathcal{A}}\mathbf{U}$-module ${_{\mathcal{A}}L(\Lambda)}$ via the canonical isomorphism
	\begin{equation*}
		\varsigma^{\Lambda}:\mathcal{K}_{0}(\Lambda) \rightarrow {_{\mathcal{A}}L(\Lambda)}
	\end{equation*}
	such that $\varsigma^{\Lambda}$ sends the image of constant sheaf $[L_0]=[\overline{\mathbb{Q}}_{l}]$ on
	$\mathbf{E}_{0,\mathbf{W},\hat{\Omega}}\cong \{pt\}$
	to the highest weight vector $v_{\Lambda}\in {_{\mathcal{A}}L(\Lambda)}$.
	Moreover, the set
	$\bigcup\limits_{\mathbf{V}}\{\varsigma^{\Lambda}([L])|L$ is a nonzero simple perverse sheaf in $\mathcal{L}_{\mathbf{V}}(\Lambda)\}$ form a bar-invariant $\mathcal{A}$-basis of ${_{\mathcal{A}}L(\Lambda)}$, which is exactly the canonical basis of ${_{\mathcal{A}}L(\Lambda)}$ constructed by Lusztig.
\end{theorem}

\begin{proof}
	By Proposition \ref{relation1}, \ref{relation2} and \ref{relation3}, $\mathcal{K}_{0}(\Lambda)$ is a $\mathbf{U}$-module. By Lemma \ref{lkey'} and Theorem \ref{Lusztig1}, for any simple perverse sheaf $L$, its image $[L]\in  \mathcal{K}_{0}(\Lambda)$ can be written as a $\mathcal{A}$-linear combination of some $[L_{\boldsymbol{\nu}\boldsymbol{d}}]= [F^{(a_{1})}_{i_{1}} F^{(a_{2})}_{i_{2}}  \cdots  F^{(a_{k})}_{i_{k}}L_{0}]$, hence $\mathcal{K}_{0}(\Lambda)$ is a highest weight module, where $[L_0]$ is the highest weight vector.
	
	It remains to prove that $\mathcal{K}_{0}(\Lambda)$ is integrable. Let $L$ be a simple perverse sheaf in $\mathcal{Q}_{\mathbf{V}}$. On the one hand, for $N> \nu_{i}$, we have $(E_{i})^{(N)}([L])=0$. On the other hand, note that if $\nu_{i}-\sum\limits_{h\in \Omega,h'=i}\nu_{h''}-d_{i}>0$, then $\mathbf{E}^{\geqslant 1 }_{\mathbf{V},\mathbf{W},i}=\mathbf{E}_{\mathbf{V},\mathbf{W},\hat{\Omega}^{i}}$, so any objects of $\mathcal{Q}_{\mathbf{V},\mathbf{W}}$ belong to $\mathcal{N}_{\mathbf{V},i}$ and $\mathcal{L}_{\mathbf{V}}(\Lambda)=0$. For large enough $N$,  we have $(F_{i})^{(N)}([L])\in \mathcal{L}_{\mathbf{V}'}(\Lambda)$, where $\mathbf{V}'$ satisfies the $\nu'_{i}-\sum\limits_{h'=i,h \in \Omega}\nu'_{h''}-d_{i}>0$, and so $(F_{i})^{(N)}([L])=0$.
	
	It is clear that nonzero simple perverse sheaves form a bar-invariant $\mathcal{A}$-basis of ${_{\mathcal{A}}L(\Lambda)}$. We only need to compare this basis with the canonical basis. 
	
	In fact, recall that the left ideal $ {_{\mathcal{A}}\mathbf{U}}^{-} f_{i}^{(\langle \Lambda,\alpha_{i}^{\vee} \rangle+1 )} $ is exactly the $\mathcal{A}$-module spanned by Lusztig's sheaves supported in $\mathbf{E}^{\geqslant d_{i}+1}_{\mathbf{V},\Omega^{i}}$  via the identification $\mathcal{K}=\bigoplus\limits_{\mathbf{V}}K_0(\mathcal{Q}_{\mathbf{V}})\cong  {_{\mathcal{A}}\mathbf{U}}^{-}$, by Proposition \ref{rt'}. Since $(\pi_{\mathbf{W}})^{-1}(\mathbf{E}^{\geqslant d_{i}+1}_{\mathbf{V},\Omega^{i}} )  \subseteq \mathbf{E}^{\geqslant 1}_{\mathbf{V},\mathbf{W},\hat{\Omega}^{i}}$,  $\textrm{supp}(A) \subseteq \mathbf{E}^{\geqslant d_{i}+1}_{\mathbf{V},\Omega^{i}}$ implies that $(\pi_{\mathbf{W}})^{\ast}(A)$  is contained in $\mathcal{N}_{\mathbf{V},i}$ for any $A$. Otherwise, if  $\textrm{supp}(A) \cap \mathbf{E}^{\leqslant d_{i}}_{\mathbf{V},\Omega^{i}} \neq \emptyset$, then we have $\textrm{supp}((\pi_{\mathbf{W}})^{\ast}(A)) \cap \mathbf{E}^{0}_{\mathbf{V},\mathbf{W},\hat{\Omega}^{i}} \neq \emptyset$. In conclusion, a Lusztig's simple perverse sheaf $L \in \mathcal{P}_{\mathbf{V}}$ induces a zero object $(\pi_{\mathbf{W}})^{\ast}(L)$ in $\mathcal{Q}_{\mathbf{V},\mathbf{W}}/\mathcal{N}_{\mathbf{V}}$ if and only if its image is contained in the left ideal $ {_{\mathcal{A}}\mathbf{U}}^{-} f_{i}^{(\langle \Lambda,\alpha_{i}^{\vee} \rangle+1 )} $ for some $i$. Let $\mathcal{I}$ be the left ideal $\sum\limits_{ i \in I } {_{\mathcal{A}}\mathbf{U}}^{-} f_{i}^{(\langle \Lambda,\alpha_{i}^{\vee} \rangle+1 )} $ and  $\tilde{\pi}$ be the canonical projection ${_{\mathcal{A}}\mathbf{U}}^{-} \rightarrow {_{\mathcal{A}}\mathbf{U}}^{-}/\mathcal{I} \cong {_{\mathcal{A}}L(\Lambda)}$, then we have the following commutative diagram and recover Lusztig's construction.
	\[
	\xymatrix{
		\mathcal{K} \ar[d]^{\varsigma} \ar[r]^{\pi}
		&  \mathcal{K}_{0}(\Lambda) \ar[d]^{\varsigma^{\Lambda}}
		\\	
		{_{\mathcal{A}}\mathbf{U}}^{-} \ar[r]^{\tilde{\pi}}
		& {_{\mathcal{A}}L(\Lambda)},
	}
	\]
	where $\pi:\mathcal{K}\rightarrow \mathcal{K}_{0}(\Lambda)$ is the $\mathcal{A}$-linear map induced by the composition of functors $$\mathcal{Q}_{\mathbf{V}}\xrightarrow{(\pi_{\mathbf{W}})^{\ast}[\sum\limits_{ i \in I }\nu_{i}d_{i}]} \mathcal{Q}_{\mathbf{V},\mathbf{W}}  \xrightarrow{\textrm{natural functor}} \mathcal{Q}_{\mathbf{V},\mathbf{W}}/\mathcal{N}_{\mathbf{V}}.$$  In particular, the set $\bigcup\limits_{\mathbf{V}}\{\varsigma^{\Lambda}([L])|L$ is a nonzero simple perverse sheaf in $\mathcal{L}_{\mathbf{V}}(\Lambda)\}$ is exactly identified with the subset of $\mathcal{P}_{\mathbf{V}}$ consisting of $L$ such that $\tilde{\pi} (\varsigma ([L])) \neq 0$, which is exactly the canonical basis of $_{\mathcal{A}}L(\Lambda)$  constructed by Lusztig.
\end{proof}

\begin{remark}
	Notice that the nonzero simple objects in $\mathcal{Q}_{\mathbf{V},\mathbf{W}}/\mathcal{N}_{\mathbf{V}}$ are exactly the simple perverse sheaves in $\mathcal{Q}_{\mathbf{V},\mathbf{W}}$ but not in  $\mathcal{N}_{\mathbf{V}}$.  We denote the set of these simple perverse sheaves by $\mathcal{P}_{\mathbf{V}} \backslash \mathcal{N}_{\mathbf{V}}$, then one can see that these simple objects form a basis of  ${_{\mathcal{A}}L(\Lambda)}$ and 
	\begin{equation*}
		|\mathcal{P}_{\mathbf{V}}\backslash \mathcal{N}_{\mathbf{V}}|={\rm{dim}}_{\mathbb{Q}(v)} L_{|\mathbf{V}|}(\Lambda).
	\end{equation*}
\end{remark}

Following \cite{MR3200442}, the Euler form of localization defines a bilinear form on  $\mathcal{K}_{0}(\Lambda)$ as the following:
\begin{definition}
	Define a bilinear form $(-,-)^{\Lambda}$ on $\mathcal{K}_{0}(\Lambda)$ by:
	\begin{equation*}
		([A],[B])^{\Lambda}=\sum\limits_{n \geq 0}{\rm{dim}} {\rm{Ext}}^{n}_{\mathcal{D}^{b}_{G_{\mathbf{V}}}(\mathbf{E}_{\mathbf{V},\mathbf{W},\hat{\Omega}})/ \mathcal{N}_{\mathbf{V}} }(\mathbf{D}A,B)v^{-n},
	\end{equation*}
	for any   objects $A,B$ of $\mathcal{L}_{\mathbf{V}}(\Lambda)$. Otherwise,	if $A$ is an object of $\mathcal{L}_{\mathbf{V}}(\Lambda)$ and $B$ is an object of $\mathcal{L}_{\mathbf{V}'}(\Lambda)$ such that $|\mathbf{V}| \neq |\mathbf{V}'|$, define
	\begin{equation*}
		([A],[B])^{\Lambda}=0.
	\end{equation*}
\end{definition}

\begin{proposition}
	For any $i \in I$, the bilinear form $(-,-)^{\Lambda}$ is contravariant with respect to $F_{i}$ and $E_{i}$. More precisely, for any objects $A,B$,
	\begin{equation*}
		( [ F_{i}A],[B] )^{\Lambda}=([A],v[K^{-1}_{i}E_{i}B])^{\Lambda}.
	\end{equation*}
	Moreover, for simple perverse sheaves $L$ and $L'$,\\
	(1)If $[L] \neq [L']$, then  $([L],[L'])^{\Lambda} \in v^{-1} \mathbb{Z}[[v^{-1}]] \cap \mathbb{Q}(v)$ ;\\
	(2)If $L$ is a nonzero simple perverse sheaf in $\mathcal{L}_{\mathbf{V}}(\Lambda)$, then $([L],[L])^{\Lambda} \in 1+ v^{-1}\mathbb{Z}[[v^{-1}]]\cap \mathbb{Q}(v)$ .
\end{proposition}

\begin{proof}	
	The proof is similar to the proof of \cite[Theorem 4.13]{MR3200442}. Notice that Lusztig's simple perverse sheaves are self-dual, the almost orthogonality follows from  the perverse $t$-structure of $\mathcal{D}^{b}_{G_{\mathbf{V}}}(\mathbf{E}_{\mathbf{V},\mathbf{W},\hat{\Omega}})/ \mathcal{N}_{\mathbf{V}}$. Recall that in the proof of \ref{relation3}, we have proved that  $ F_{i}$ is left adjoint to  $E_{i}K_{i}[-1]$. Hence for any objects $A,B$,
	\begin{equation*}
		\begin{split}
			&{\rm{dim}} {\rm{Ext}}^{j}_{ \mathcal{D}^{b}_{G_{\mathbf{V}}}(\mathbf{E}_{\mathbf{V},\mathbf{W},\hat{\Omega}})/ \mathcal{N}_{\mathbf{V}} } (\mathbf{D}F_{i}A,B)\\
			=&
			{\rm{dim}} {\rm{Ext}}^{j}_{ \mathcal{D}^{b}_{G_{\mathbf{V}}}(\mathbf{E}_{\mathbf{V},\mathbf{W},\hat{\Omega}})/ \mathcal{N}_{\mathbf{V}} } (F_{i}\mathbf{D}A,B)\\
			=&  {\rm{dim}} {\rm{Ext}}^{j}_{ \mathcal{D}^{b}_{G_{\mathbf{V}'}}(\mathbf{E}_{\mathbf{V}',\mathbf{W},\hat{\Omega}})/ \mathcal{N}_{\mathbf{V}'}} (\mathbf{D}A, E_{i}K^{-1}_{i}B[-1]) \\
			=&  {\rm{dim}} {\rm{Ext}}^{j}_{ \mathcal{D}^{b}_{G_{\mathbf{V}'}}(\mathbf{E}_{\mathbf{V}',\mathbf{W},\hat{\Omega}})/ \mathcal{N}_{\mathbf{V}'}} (\mathbf{D}A, K^{-1}_{i}E_{i}B[1]) .
		\end{split}
	\end{equation*}
	Then by definition of $(-,-)^{\Lambda}$, we finish the proof. 
\end{proof}

\subsection{Compare the functor $E_{i}$ with derivation functors}

In this section, we construct a split exact sequence which is analogy to the exact sequence in \cite[Theorem 4.7]{MR2995184}. We fix a vertex $i\in I$ and an orientation $\Omega$ such that $i$ is a source. 

\begin{definition}
	
	(1) Define $\hat{\mathcal{Q}}_{\mathbf{V},\mathbf{W}}$ to be the full subcategory of $\mathcal{D}_{\mathbf{G}_{\mathbf{V}}}(\mathbf{E}_{\mathbf{V},\mathbf{W},\hat{\Omega}})$ whose objects are direct sums of Lusztig's sheaves in $\mathcal{Q}_{\mathbf{V},\mathbf{W}} $ such that multiplicity of each simple summand $A[n]$ is finite for any simple object $A$ and $n \in \mathbb{Z}$.
	
	(2) Define the localization $\hat{\mathcal{Q}}_{\mathbf{V},\mathbf{W}}/\mathcal{N}_{\mathbf{V},i}$ to be the full subcategory of $\mathcal{D}_{\mathbf{G}_{\mathbf{V}}}(\mathbf{E}_{\mathbf{V},\mathbf{W},\hat{\Omega}})/\mathcal{N}_{\mathbf{V},i}$ consisting of objects in $\hat{\mathcal{Q}}_{\mathbf{V},\mathbf{W}}$, and define the global localization $\hat{\mathcal{Q}}_{\mathbf{V},\mathbf{W}}/\mathcal{N}_{\mathbf{V}}$ to be the full subcategory of $\mathcal{D}_{\mathbf{G}_{\mathbf{V}}}(\mathbf{E}_{\mathbf{V},\mathbf{W},\hat{\Omega}})/\mathcal{N}_{\mathbf{V}}$  consisting of objects in $\hat{\mathcal{Q}}_{\mathbf{V},\mathbf{W}}$.
\end{definition}

Let $\hat{\mathcal{K}}_{\mathbf{V}}$ be the Grothendieck group of
$\hat{
	\mathcal{Q}}_{\mathbf{V},\mathbf{W}}$ and $\hat{\mathcal{K}}=\bigoplus\limits_{\mathbf{V}} \hat{\mathcal{K}}_{\mathbf{V}}$. They have  $\mathbb{Z}[[v,v^{-1}]]$-module structures and the set of simple objects forms a bar-invariant $\mathbb{Z}[[v,v^{-1}]]$-basis of $\hat{\mathcal{K}}$. 

Let  ${\rm{H}}^{\ast}(\mathbb{P}^{\infty})$ be the cohomology of the infinite projective variety, then following the proof of \cite[Lemma 12.3.6]{MR1227098},  $${\rm{H}}^{\ast}(\mathbb{P}^{\infty})\cong \bigoplus\limits_{m \geqslant 0  } \overline{\mathbb{Q}}_{l} [-2m].$$

\begin{definition}
	For graded spaces $\mathbf{V},\mathbf{V}',\mathbf{V}''$ such that $|\mathbf{V}|=|\mathbf{V}'|+i,|\mathbf{V}''|=i$ and $\mathbf{V}=\mathbf{V}'\oplus \mathbf{V}''$. 
	
	(1) define the functor $\hat{\mathcal{R}}^{\Lambda}_{i}: \mathcal{Q}_{\mathbf{V},\mathbf{W}} \rightarrow \hat{\mathcal{Q}}_{\mathbf{V}',\mathbf{W}}/\mathcal{N}_{\mathbf{V}'}$ by
	\begin{equation*}
		\hat{\mathcal{R}}^{\Lambda}_{i}(-)=\mathbf{Res}^{\mathbf{V}\oplus \mathbf{W}}_{\mathbf{V}'\oplus\mathbf{W},\mathbf{V}''}(-)\otimes ({\rm{H}}^{\ast}(\mathbb{P}^{\infty})) [-1].
	\end{equation*}

	(2) Define the functor ${_{i}\hat{\mathcal{R}}^{\Lambda}}: \mathcal{Q}_{\mathbf{V},\mathbf{W}} \rightarrow \hat{\mathcal{Q}}_{\mathbf{V}',\mathbf{W}}/\mathcal{N}_{\mathbf{V}'}$ by
	\begin{equation*}
		{_{i}\hat{\mathcal{R}}^{\Lambda}}(-)=\mathbf{Res}^{\mathbf{V}\oplus\mathbf{W}}_{\mathbf{V}'',\mathbf{V}'\oplus\mathbf{W}}(-)\otimes ({\rm{H}}^{\ast}(\mathbb{P}^{\infty})) [(i,|\mathbf{V}'\oplus \mathbf{W}|)-1].
	\end{equation*}
\end{definition}

\begin{lemma}\label{lemma29}
	For a graded space  $\mathbf{V}$ and $\boldsymbol{\nu} \in \mathcal{S}_{|\mathbf{V}|}$, ${_{i}\hat{\mathcal{R}}^{\Lambda}}(L_{\boldsymbol{\nu}\boldsymbol{d}})$ is a direct summand of $\hat{\mathcal{R}}^{\Lambda}_{i}(L_{\boldsymbol{\nu}\boldsymbol{d}})$ and its complement $C(L_{\boldsymbol{\nu}\boldsymbol{d}})$ is a finite direct sum of some shifted $[L_{\boldsymbol{\omega}\boldsymbol{d}}]$, where $\boldsymbol{\omega} \in \mathcal{S}_{\mathbf{V}'}$ runs over  flag types which can be obtained from $\boldsymbol{\nu}$ by removing $i$ from some $\nu^{k}=i^{a_{k}}$. In particular, $C(L_{\boldsymbol{\nu}\boldsymbol{d}})$ belongs to $\mathcal{Q}_{\mathbf{V}',\mathbf{W}}/\mathcal{N}_{\mathbf{V}'}$.
\end{lemma}

\begin{proof}	
	Consider the following morphisms in the definition of  the restriction functors $\mathbf{Res}^{\mathbf{V}\oplus\mathbf{W}}_{\mathbf{V}'\oplus\mathbf{W},\mathbf{V}''}$ and $\mathbf{Res}^{\mathbf{V}\oplus\mathbf{W}}_{\mathbf{V}'',\mathbf{V}'\oplus\mathbf{W}}$ respectively
	\begin{equation*}
		\mathbf{E}_{\mathbf{V}',\mathbf{W},\hat{\Omega}} \xleftarrow{\kappa'_{\Omega}} F' \xrightarrow{ \iota'_{\Omega} }  \mathbf{E}_{\mathbf{V},\mathbf{W},\hat{\Omega}};
	\end{equation*}
	\begin{equation*}
		\mathbf{E}_{\mathbf{V}',\mathbf{W},\hat{\Omega}} \xleftarrow{\kappa_{\Omega}} F \xrightarrow{ \iota_{\Omega} }  \mathbf{E}_{\mathbf{V},\mathbf{W},\hat{\Omega}}.
	\end{equation*}
	It is well known that  the restriction functor is a hyperbolic localization functor, that is, there is a $k^*$-action on $\mathbf{E}_{\mathbf{V},\mathbf{W},\hat{\Omega}}$ such that the following diagrams commute
	\[
	\xymatrix{
		(\mathbf{E}_{\mathbf{V},\mathbf{W},\hat{\Omega}})^{k^*} \ar[d]_{\cong} 
		&  {\mathbf{E}^{+}_{\mathbf{V},\mathbf{W},\hat{\Omega}}} \ar[l]_{\pi^{+}} \ar[d]^{\cong} \ar[r]^{g^+} &\mathbf{E}_{\mathbf{V},\mathbf{W},\hat{\Omega}} \ar@{=}[d] 
		\\	
		{\mathbf{E}_{\mathbf{V}',\mathbf{W},\hat{\Omega}}} 
		& F \ar[l]_{\kappa_{\Omega}} \ar[r]^{\iota_{\Omega}} &\mathbf{E}_{\mathbf{V},\mathbf{W},\hat{\Omega}} ,
	}
	\]
	
	\[
	\xymatrix{
		{(\mathbf{E}_{\mathbf{V},\mathbf{W},\hat{\Omega}}})^{k^*} \ar[d]_{\cong} \ar[d]_{\cong} 
		&  {\mathbf{E}^{-}_{\mathbf{V},\mathbf{W},\hat{\Omega}}} \ar[l]_{\pi^{-}} \ar[d]^{\cong} \ar[r]^{g^{-}} &\mathbf{E}_{\mathbf{V},\mathbf{W},\hat{\Omega}} \ar@{=}[d]
		\\
		{\mathbf{E}_{\mathbf{V}',\mathbf{W},\hat{\Omega}}} 
		& F' \ar[l]_{\kappa'_{\Omega}} \ar[r]^{\iota'_{\Omega}} &\mathbf{E}_{\mathbf{V},\mathbf{W},\hat{\Omega}},
	}
	\]
	
	see \cite{Barden} and \cite[Proposition 2.10]{MR4524567} for details. By formula (1) and Theorem 1 in \cite{Barden}, we have 
	\begin{align*}
		\mathbf{Res}^{\mathbf{V}\oplus\mathbf{W}}_{\mathbf{V}'\oplus\mathbf{W},\mathbf{V}''} 
		\cong &(\kappa'_{\Omega})_{!}(\iota'_{\Omega})^{\ast} [-\langle|\mathbf{V}'\oplus \mathbf{W}|,i \rangle]
		\cong (\pi^{-})_{!}(g^{-})^{\ast}[-\langle|\mathbf{V}'\oplus \mathbf{W}|,i\rangle]\\
		\cong &(\pi^{+})_{\ast}(g^{+})^{!}[-\langle|\mathbf{V}'\oplus \mathbf{W}|,i\rangle]
		\cong  (\kappa_{\Omega})_{\ast}(\iota_{\Omega})^{!}[-\langle|\mathbf{V}'\oplus \mathbf{W}|,i\rangle].
	\end{align*}
	
	Since $i$ is a source,  $\iota_{\Omega}$ is the identity,  and so  
	\begin{align*}
		\hat{\mathcal{R}}^{\Lambda}_{i} \cong  \bigoplus\limits_{n\geqslant 0}(\kappa_{\Omega})_{\ast} [-\langle|\mathbf{V}'\oplus \mathbf{W}|,i\rangle-1-2n].
	\end{align*}
	By definition,
	\begin{align*}
		{_{i}\hat{\mathcal{R}}^{\Lambda}} \cong & \bigoplus\limits_{n\geqslant 0}(\kappa_{\Omega})_{!} [(i,|\mathbf{V}'\oplus \mathbf{W}|)-\langle i,|\mathbf{V}'\oplus \mathbf{W}|\rangle-1-2n]\\
		\cong & \bigoplus\limits_{n\geqslant 0}(\kappa_{\Omega})_{!} [\langle|\mathbf{V}'\oplus \mathbf{W}|,i\rangle-1-2n].
	\end{align*}

	By Proposition \ref{indres formula}, $\mathbf{Res}^{\mathbf{V}\oplus\mathbf{W}}_{\mathbf{V}'\oplus\mathbf{W},\mathbf{V}''}(L_{\boldsymbol{\nu}\boldsymbol{d}})$ is a direct sum of some $L_{\boldsymbol{\omega}\boldsymbol{d}}$ such that $\boldsymbol{\omega} \in \mathcal{S}_{\mathbf{V}'}$ is a flag type which can be obtained from $\boldsymbol{\nu}$ by removing $i$ from $\nu^{k}$ for some $k$. More precisely, $\boldsymbol{\omega}=(\nu^{1},\nu^{2},\cdots, \nu^{k-1},i^{a_{k}-1},\nu^{k+1},\cdots, \nu^{m} )$ for some $k$ such that $\nu^{k}=i^{a_{k}}, a_{k}>0$. Consider the following commutative diagram in the proof of \cite[Proposition 4.2]{MR1088333},
	\[
	\xymatrix{
		\mathcal{F}_{\boldsymbol{\nu}\boldsymbol{d},\hat{\Omega}} (\boldsymbol{\omega}\boldsymbol{d}) \ar[d]_{ \alpha_{\boldsymbol{\omega}} } \ar[r]
		& \mathcal{F}_{\boldsymbol{\nu}\boldsymbol{d},\hat{\Omega}} \ar[d]^{\kappa_{\Omega} \pi_{\boldsymbol{\nu}\boldsymbol{d},\hat{\Omega}}}
		\\	
		\mathcal{F}_{\boldsymbol{\omega}\boldsymbol{d},\hat{\Omega}} \ar[r]^{\pi_{\boldsymbol{\omega}\boldsymbol{d},\hat{\Omega}}}
		& \mathbf{E}_{\mathbf{V}',\mathbf{W},\hat{\Omega}}
	}
	\]
	where $\mathcal{F}_{\boldsymbol{\nu}\boldsymbol{d},\hat{\Omega}} (\boldsymbol{\omega})$ is the locally closed subset of the flag variety $\mathcal{F}_{\boldsymbol{\nu}\boldsymbol{d},\hat{\Omega}}$ consisting of $(x,f)$ such that $x \in \mathbf{E}_{\mathbf{V},\mathbf{W},\hat{\Omega}}$ and $f$ is a flag of $\mathbf{V}\oplus \mathbf{W}$ which  induces a flag of $\mathbf{V}'\oplus \mathbf{W}$ such that its type is $\boldsymbol{\omega}\boldsymbol{d}$. By \cite[Lemma 4.4]{MR1088333},  $\alpha_{\boldsymbol{\omega}} $ is a vector bundle, we denote its rank by $f_{\boldsymbol{\omega}}$. Then we have 
	$$(\kappa_{\Omega})_{!} (L_{\boldsymbol{\nu}\boldsymbol{d}})=\bigoplus\limits_{\boldsymbol{\omega}} L_{\boldsymbol{\omega}\boldsymbol{d}}[\dim \mathcal{F}_{\boldsymbol{\nu}\boldsymbol{d},\hat{\Omega}}-\dim \mathcal{F}_{\boldsymbol{\omega}\boldsymbol{d},\hat{\Omega}} -2 f_{\boldsymbol{\omega}} ],$$  
	$$(\kappa_{\Omega})_{\ast} (L_{\boldsymbol{\nu}\boldsymbol{d}})=\bigoplus\limits_{\boldsymbol{\omega}} L_{\boldsymbol{\omega}\boldsymbol{d}}[\dim \mathcal{F}_{\boldsymbol{\nu}\boldsymbol{d},\hat{\Omega}}-\dim \mathcal{F}_{\boldsymbol{\omega}\boldsymbol{d},\hat{\Omega}} ],$$ 
	where the direct sum follows from similar arguments as in the proof of \cite[Lemma 4.7]{MR1088333} or \cite[Corollary 3.7]{MR4524567}. Hence 
	$$\hat{\mathcal{R}}^{\Lambda}_{i}(L_{\boldsymbol{\nu}\boldsymbol{d}})=\bigoplus\limits_{\boldsymbol{\omega}} \bigoplus\limits_{n \geqslant 0}L_{\boldsymbol{\omega}\boldsymbol{d}}[\dim \mathcal{F}_{\boldsymbol{\nu}\boldsymbol{d},\hat{\Omega}}-\dim \mathcal{F}_{\boldsymbol{\omega}\boldsymbol{d},\hat{\Omega}}-\langle|\mathbf{V}'\oplus \mathbf{W}|,i\rangle-1-2n ], $$
	and
	$${_{i}\hat{\mathcal{R}}^{\Lambda}}(L_{\boldsymbol{\nu}\boldsymbol{d}})=\bigoplus\limits_{\boldsymbol{\omega}} \bigoplus\limits_{n \geqslant 0}L_{\boldsymbol{\omega}\boldsymbol{d}}[\dim \mathcal{F}_{\boldsymbol{\nu}\boldsymbol{d},\hat{\Omega}}-\dim \mathcal{F}_{\boldsymbol{\omega}\boldsymbol{d},\hat{\Omega}}-2f_{\boldsymbol{\omega}} +\langle|\mathbf{V}'\oplus \mathbf{W}|,i\rangle-1-2n ]. $$
	
	If $\langle|\mathbf{V}'\oplus \mathbf{W}|,i\rangle \leqslant f_{\boldsymbol{\omega}} $, then $-2f_{\boldsymbol{\omega}} +\langle|\mathbf{V}'\oplus \mathbf{W}|,i\rangle \leqslant  -\langle|\mathbf{V}'\oplus \mathbf{W}|,i\rangle$. For such $\boldsymbol{\omega}$, it is easy to see that $$\bigoplus\limits_{n \geqslant 0}L_{\boldsymbol{\omega}\boldsymbol{d}}[\dim \mathcal{F}_{\boldsymbol{\nu}\boldsymbol{d},\hat{\Omega}}-\dim \mathcal{F}_{\boldsymbol{\omega}\boldsymbol{d},\hat{\Omega}}-2f_{\boldsymbol{\omega}} +\langle|\mathbf{V}'\oplus \mathbf{W}|,i\rangle-1-2n ]$$ is a direct summand of $$\bigoplus\limits_{n \geqslant 0}L_{\boldsymbol{\omega}\boldsymbol{d}}[\dim \mathcal{F}_{\boldsymbol{\nu}\boldsymbol{d},\hat{\Omega}}-\dim \mathcal{F}_{\boldsymbol{\omega}\boldsymbol{d},\hat{\Omega}}-\langle|\mathbf{V}'\oplus \mathbf{W}|,i\rangle-1-2n ]$$ and its complement is a finite direct sum of some shifted $[L_{\boldsymbol{\omega}\boldsymbol{d}}]$.
	
	Otherwise, if $\langle|\mathbf{V}'\oplus \mathbf{W}|,i\rangle > f_{\boldsymbol{\omega}} $, we claim that $L_{\boldsymbol{\omega}\boldsymbol{d}} $ is contained in $\mathcal{N}_{\mathbf{V}',i}$, then the lemma is proved. In fact, with the notations above, we assume that $$\boldsymbol{\omega}=(\nu^{1},\nu^{2},\cdots, \nu^{k-1},i^{a_{k}-1},\nu^{k+1},\cdots, \nu^{m} )$$ and take $\boldsymbol{\omega}'= (i^{a_{k}-1},\nu^{k+1},\cdots, \nu^{m} )$. Then $L_{\boldsymbol{\omega}'\boldsymbol{d}}$ is contained in $\mathcal{Q}_{\mathbf{V}'',\mathbf{W}}$ for some $\mathbf{V}''$. Since $i$ is a source in $\hat{\Omega}$, $\langle|\mathbf{V}'\oplus \mathbf{W}|,i\rangle =\nu'_{i} $. By Proposition 4.2(b) in \cite{MR1088333}, we can see that 
	$f_{\boldsymbol{\omega}}=\sum\limits_{l < k}\nu^{l}_{i} + \sum\limits_{l>k,h\in \hat{\Omega},h'=i} \nu^{l}_{h''},$ then $0 < \langle|\mathbf{V}'\oplus \mathbf{W}|,i\rangle - f_{\boldsymbol{\omega}}= \nu''_{i} - \sum\limits_{h \in \hat{\Omega}, h'=i} \nu''_{h''}. $ In particular, $\mathbf{E}^{\geqslant 1}_{\mathbf{V}'',\mathbf{W},i}= \mathbf{E}_{\mathbf{V}'',\mathbf{W},\hat{\Omega}}$ and $L_{\boldsymbol{\omega}'\boldsymbol{d}}$ belongs to $\mathcal{N}_{\mathbf{V}'',i}$. Notice that $L_{\boldsymbol{\omega}\boldsymbol{d}} \cong F_{i_{1}}^{(a_{1})}\cdots F_{i_{k-1}}^{(a_{k-1})}L_{\boldsymbol{\omega}'\boldsymbol{d}} $, it is also contained in $\mathcal{N}_{\mathbf{V}',i}$ and we finish the proof.

\end{proof}

\begin{theorem}\label{exact}
	For any  $L$ in $\mathcal{Q}_{\mathbf{V},\mathbf{W}}$, there is a split exact sequence in $\hat{\mathcal{Q}}_{\mathbf{V}',\mathbf{W}}/ \mathcal{N}_{\mathbf{V}'}$ 
	\begin{equation*}
		0 \rightarrow {_{i}\hat{\mathcal{R}}^{\Lambda}}(L)  \rightarrow {\hat{\mathcal{R}}^{\Lambda}_{i}}(L)  \rightarrow E_{i}(L) \rightarrow 0.
	\end{equation*}
\end{theorem}

\begin{proof}
	By Lemma \ref{lemma29}, we can define an $\mathcal{A}$-linear operator $\tilde{E}_{i}$ by sending $[L_{\boldsymbol{\nu}\boldsymbol{d}}]$ to $[C(L_{\boldsymbol{\nu}\boldsymbol{d}})] $ for any flag type $\boldsymbol{\nu}$. By definition, we have $\tilde{E}_{i}([L_{\boldsymbol{\nu}\boldsymbol{d}}]) = [{\hat{\mathcal{R}}^{\Lambda}_{i}}(L_{\boldsymbol{\nu}\boldsymbol{d}})  ]-[{_{i}\hat{\mathcal{R}}^{\Lambda}}(L_{\boldsymbol{\nu}\boldsymbol{d}})]. $
	From \cite{MR1088333} and Corollary  3.1 and 3.3 in \cite{zhao2022derivation}, we can see that
	$${\hat{\mathcal{R}}^{\Lambda}_{i}}F_{j} \cong (F_{j}{\hat{\mathcal{R}}^{\Lambda}_{i}}) \oplus (\delta_{i,j} \rm{Id} \otimes H^{\ast}(\mathbb{P}^{\infty})[-(i,|\mathbf{V}'\oplus\mathbf{W}|)-1] ),$$
	$${_{i}\hat{\mathcal{R}}^{\Lambda}}F_{j} \cong ({F_{j}} {_{i}\hat{\mathcal{R}}^{\Lambda}}) \oplus (\delta_{i,j} \rm{Id} \otimes H^{\ast}(\mathbb{P}^{\infty})[ (i,|\mathbf{V}'\oplus\mathbf{W}|)-1] ).$$
	Hence, we can check that the linear operator $\tilde{E}_{i}=[{\hat{\mathcal{R}}^{\Lambda}_{i}}]-[{_{i}\hat{\mathcal{R}}^{\Lambda}}] :\mathcal{K} \rightarrow \mathcal{K}_{0}(\Lambda)$ satisfies the relations 
	$$ \tilde{E}_{i}F_{j}= F_{j}\tilde{E}_{i}, i\neq j; $$
	$$	\tilde{E}_{i}F_{i} + \sum \limits_{0\leqslant m \leqslant N-1} v^{N-1-2m} {\rm{Id}} =F_{i}\tilde{E}_{i}+ \sum \limits_{0\leqslant m \leqslant -N-1} v^{-2m-N-1}{\rm{Id}},$$
	where $N=2\nu'_{i}-\tilde{\nu}'_{i}.$ Since the operator $E_{i}$ also satisfies the same commutative relation with $F_{j}$, we can argue by induction to show that $$[\tilde{E}_{i}(L_{\boldsymbol{\nu}'\boldsymbol{d}})]=[E_{i}(L_{\boldsymbol{\nu}'\boldsymbol{d}})]$$ for any flag type  $\boldsymbol{\nu}'=(\nu'^{1},\nu'^{2}, \cdots \nu'^{k-1}, \nu'^{k})$ with each $\nu'^{l} \in I$. (It means that $\boldsymbol{\nu}'$ is a flag type of complete flags). Note that each $L_{\boldsymbol{\nu}'\boldsymbol{d}}$ can be written as $L_{\boldsymbol{\nu}\boldsymbol{d}}\otimes {\rm{H}}^{\ast}(X)$ for some projective variety $X$ (up to shifts), then $C(L_{\boldsymbol{\nu}\boldsymbol{d}})\otimes {\rm{H}}^{\ast}(X) \cong C(L_{\boldsymbol{\nu}'\boldsymbol{d}}) $ (up to shifts) and we must have $[\tilde{E}_{i}(L_{\boldsymbol{\nu}\boldsymbol{d}})]=[E_{i}(L_{\boldsymbol{\nu}\boldsymbol{d}})]$  for any flag type $\boldsymbol{\nu}$. In particular, $E_{i}(L_{\boldsymbol{\nu}\boldsymbol{d}}) \cong C(L_{\boldsymbol{\nu}\boldsymbol{d}})$, since both of them are bounded semisimple complexes (up to isomorphisms). Then by definition, for any $\boldsymbol{\nu}$, $$E_{i}(L_{\boldsymbol{\nu}\boldsymbol{d}}) \oplus {_{i}\hat{\mathcal{R}}^{\Lambda}}(L_{\boldsymbol{\nu}\boldsymbol{d}}) \cong {\hat{\mathcal{R}}^{\Lambda}_{i}}(L_{\boldsymbol{\nu}\boldsymbol{d}}). $$ 
	
	Notice that from Lemma \ref{lkey'} and Proposition \ref{lt'}, we can deduce the following fact: for any simple object $L$ in $\mathcal{Q}_{\mathbf{V},\mathbf{W}}$, there exist families of flag types $\boldsymbol{\omega}$,$\boldsymbol{\tau}$ and integers $N(\boldsymbol{\omega},n)$ and $N(\boldsymbol{\tau},n)$ such that $$ L \oplus \bigoplus\limits_{ \boldsymbol{\omega},n } L_{\boldsymbol{\omega}\boldsymbol{d}}^{\oplus N(\boldsymbol{\omega},n)}[n] \cong \bigoplus\limits_{ \boldsymbol{\tau},n } L_{\boldsymbol{\tau}\boldsymbol{d}}^{\oplus N(\boldsymbol{\tau},n)}[n],$$ hence for any $L$, $E_{i}(L) \oplus {_{i}\hat{\mathcal{R}}^{\Lambda}}(L) \cong {\hat{\mathcal{R}}^{\Lambda}_{i}}(L)$  and we finish the proof.
	
\end{proof}

\begin{remark}
	Recall that the derivation operators $\bar{r}_{i}$ and $_{i}\bar{r}$ can be realized by the restriction functor, see \cite{MR1227098} or \cite{zhao2022derivation}. If we formally write $\frac{1}{v-v^{-1}}= \sum \limits_{n\geqslant 0 } v^{-1-2n}$, then $\hat{\mathcal{R}}^{\Lambda}_{i}$ and $ {_{i}\hat{\mathcal{R}}^{\Lambda}}$ indeed realize the  linear operators  $\frac{ 1  }{v-v^{-1}}\bar{r}_{i}, \frac{ v^{(i,|\mathbf{V}'\oplus\mathbf{W}|)}}{v-v^{-1}} {_{i}\bar{r}}: {_{\mathcal{A}}\mathbf{U}^{-}  } \rightarrow \mathbf{U}^{-}  $ respectively. Therefore, Theorem \ref{exact} indeed categorify the equation
	$$E_{i}(x \cdot  v_{\Lambda} )=( v^{(i,|x|-i)- \langle \Lambda,\alpha^{\vee}_{i} \rangle}{_{i} \bar{r}}(x)\cdot  v_{\Lambda} - v^{\langle \Lambda,\alpha^{\vee}_{i} \rangle} \bar{r}_{i}(x)\cdot  v_{\Lambda} ) /(v^{-1}-v ).  $$
\end{remark}

\section{Compare with Nakajima's realization}
In this section, we recall Nakajima's realization of integrable highest weight modules via his quiver varieties in \cite{MR1302318} and \cite{MR1604167} and compare our construction (at $v \rightarrow 1$) with the realization of Nakajima. We omit  $\mathbb{G}_{m}$-action defined in \cite[Section 3.iv]{MR1604167} in order to simplify the notations.
\subsection{Nakajima quiver variety}
For a given quiver $Q=(I,H,\Omega)$, Nakajima considered the double version of framed quiver $\hat{Q}$, its set of vertices is $I\cup \hat{I}$ and its set of arrows is $\hat{H}=H \cup \{ i \rightarrow \hat{i},\hat{i} \rightarrow i|i\in I\}$.

For dimension vectors $\nu,\omega$, $I$-graded space $\mathbf{V}$ and $\hat{I}$-graded space $\mathbf{W}$ such that $|\mathbf{V}|=\nu$ and $|\mathbf{W}|=\omega$, we consider the moduli space
\begin{equation*}
	\begin{split}
		\mathbf{E}_{\mathbf{V},\mathbf{W}}&= \bigoplus\limits_{h \in \hat{H}} \mathbf{Hom}(\mathbf{V}_{h'},\mathbf{V}_{h''})\\
		&=   \bigoplus\limits_{h \in H} \mathbf{Hom}(\mathbf{V}_{h'},\mathbf{V}_{h''}) \oplus \bigoplus\limits_{i \in I} \mathbf{Hom}(\mathbf{V}_{i},\mathbf{W}_{\hat{i}}) \oplus \bigoplus\limits_{i \in I} \mathbf{Hom}(\mathbf{W}_{\hat{i}},\mathbf{V}_{i}).
	\end{split}
\end{equation*}

We denote an element of $\mathbf{E}_{\mathbf{V},\mathbf{W}}$ by $(B,i,j)$, where $B$ is an element of $\mathbf{E}_{\mathbf{V}}$, and $ i,j$ are elements of $ \bigoplus\limits_{i \in I} \mathbf{Hom}(\mathbf{V}_{i},\mathbf{W}_{\hat{i}})$ and $\bigoplus\limits_{i \in I} \mathbf{Hom}(\mathbf{W}_{\hat{i}},\mathbf{V}_{i})$ respectively. The moment map of $\mathbf{E}_{\mathbf{V},\mathbf{W}}$ is defined by
\begin{equation*}
	(\mu( (B,i,j) ))_{k \in I}= \sum\limits_{h \in H, h''=k} \epsilon(h)B_{h}B_{\bar{h}} +j_{k}i_{k} .
\end{equation*} 
The algebraic group
\begin{equation*}
	G=G_{\mathbf{V}}= \prod\limits_{i \in I} \mathbf{GL}(\mathbf{V}_{i})
\end{equation*} 
acts on $\mathbf{E}_{\mathbf{V},\mathbf{W}}$ by 
\begin{equation*}
	g \cdot (B,i,j) =(g^{-1}Bg,ig,g^{-1}j).
\end{equation*}

\begin{definition}
	Nakajima quiver variety $\mathfrak{m}_{0}(\nu,\omega)$ and $\mathfrak{m}(\nu,\omega)$ are defined as the affine and projective geometric quotients of $\mu^{-1}(0)$ respectively. More precisely, let $A(\mu^{-1}(0))$ be the coordinate ring of the affine variety $\mu^{-1}(0)$, the affine quiver variety $\mathfrak{m}_{0}(\nu,\omega)=\mathfrak{m}_{0}$ is defined by
	\begin{equation*}
		\mathfrak{m}_{0}(\nu,\omega)=\mu^{-1}(0)//G={\rm{Spec}} \ A(\mu^{-1}(0))^{G}.
	\end{equation*}
	Take character $\chi_{G}:G \rightarrow \mathbb{C}; g \mapsto \prod\limits_{k\in I} {\rm{det}} g_{k}^{-1}$ and set
	\begin{equation*}
		A(\mu^{-1}(0))^{G,\chi_{G}^{n}}=\{f\in A(\mu^{-1}(0))|f(g(B,i,j))=\chi_{G}(g)^{n}f((B,i,j))\}
	\end{equation*}
	then the projective quiver variety $\mathfrak{m}(\nu,\omega)=\mathfrak{m}$ is defined by 
	\begin{equation*}
		\mathfrak{m}(\nu,\omega)={\rm{Proj}}( \bigoplus \limits_{n \geqslant 0} 	A(\mu^{-1}(0))^{G,\chi_{G}^{n}} ).
	\end{equation*}
\end{definition}	

In order to describe $\mathfrak{m}(\nu,\omega)$, Nakajima introduced the following stable conditions.
\begin{lemma}
	Assume that $(B,i,j)\in \mu^{-1}(0)$, then $(B,i,j)$ is stable if and only if the following statement holds,
	
	(S)\ For any $B$-stable subspace $S \subset \mathbf{V}$ such that $S \subset {\rm{Ker}}i$, we have $S=0$.
\end{lemma}	

Let $\mu^{-1}(0)^{s}$ be the subset of stable points, then Nakajima proved the following lemma and corollary in \cite{MR1604167}.

\begin{lemma}
	If $(B,i,j) \in \mu^{-1}(0)^{s}$,then:\\
	(1) ${\rm{Stab}}_{G}((B,i,j))=\{1\}$;\\
	(2) The tangent map $d\mu$ is surjective at $(B,i,j)$. In particular $\mu^{-1}(0)^{s}$ is nonsingular and of dimension $\sum\limits_{k\in I} (2\nu_{k}\omega_{k}-\nu_{k}^{2}) +(\nu,\nu)$.
\end{lemma}

\begin{corollary}
	Nakajima quiver variety $\mathfrak{m}(\nu,\omega)$ is the geometric quotient of  $\mu^{-1}(0)^{s}$. In particular, $\mathfrak{m}(\nu,\omega)$ is nonsingular and of dimension $\sum\limits_{k\in I} (2\nu_{k}\omega_{k}-2\nu_{k}^{2}) +(\nu,\nu)$. Moreover, the closed points in $\mathfrak{m}(\nu,\omega)$ are bijective to orbits of $\mu^{-1}(0)^{s}$.
\end{corollary}

We denote the geometric point corresponding to the orbit of $(B,i,j)\in \mu^{-1}(0)^{s}$ by $[B,i,j]$. If the orbit of $(B,i,j)\in \mu^{-1}(0)$ is close, then $[B,i,j]$ is also a geometric point of $\mathfrak{m}_{0}(\nu,\omega)$. The geometric invariant theory provides a morphism $\pi:\mathfrak{m}(\nu,\omega)\rightarrow \mathfrak{m}_{0}(\nu,\omega)$:
\begin{equation*}
	\pi([B,i,j])=[B^{0},i^{0},j^{0}]
\end{equation*}
where $G(B^{0},i^{0},j^{0})$ is the unique close orbit contained in the closure of $G(B,i,j)$.

\begin{proposition}
	The algebraic variety $\pi^{-1}(0) \subset \mathfrak{m}(\nu,\omega)$ is Lagrangian and homotopic to $\mathfrak{m}(\nu,\omega)$. In particular, we denote $\pi^{-1}(0)$ by $\mathfrak{L}(\nu,\omega)$.
\end{proposition}

For dimension vectors $\nu^{1},\nu^{2}$ and graded spaces $\mathbf{V}^{1},\mathbf{V}^{2}$ such that $|\mathbf{V}^{i}|=\nu^{i},i=1,2$, we define the morphisms $\pi_{i}: \mathfrak{m}(\nu^{i},\omega) \rightarrow \mathfrak{m}_{0}(\nu^{1}+\nu^{2},\omega)$ to be the composition of projections $\pi: \mathfrak{m}(\nu^{i},\omega) \rightarrow \mathfrak{m}_{0}(\nu^{i},\omega)$ and embeddings $\mathfrak{m}_{0}(\nu^{i},\omega) \rightarrow \mathfrak{m}_{0}(\nu^{1}+\nu^{2},\omega)$.

\begin{definition}
	Let $Z(\nu^{1},\nu^{2},\omega)$  be the subvariety $\mathfrak{m}(\nu^{1},\omega)\times \mathfrak{m}(\nu^{2},\omega)$ defined by:
	\begin{equation*}
		Z(\nu^{1},\nu^{2},\omega)=\{(x_{1},x_{2})\in \mathfrak{m}(\nu^{1},\omega)\times \mathfrak{m}(\nu^{2},\omega) |\pi_{1}(x_{1})=\pi_{2}(x_{2})\}.
	\end{equation*}
\end{definition}

Let $\mathfrak{m}_{0}(\nu,\omega)^{reg}$ be the subset of $\mathfrak{m}_{0}(\nu,\omega)$ consisting of $[B,i,j]$ with trivial stabilizer ${\rm{Stab}(B,i,j)}$. Then we have the following result by Nakajima:
\begin{lemma}
	If $[B,i,j] \in \mathfrak{m}_{0}(\nu,\omega)^{reg}$, then $(B,i,j) \in \mu^{-1}(0)^{s}$. Moreover, $\pi$ induces an isomorphism between $\mathfrak{m}_{0}(\nu,\omega)^{reg}$ and $\pi^{-1}(\mathfrak{m}_{0}(\nu,\omega)^{reg})$.
\end{lemma}

\begin{definition}
	Let $Z^{reg}(\nu^{1},\nu^{2},\omega)$ be the complement of the closure of the set
	\begin{equation*}
		\{(x_{1},x_{2})| \pi_{1}(x_{1})=\pi_{2}(x_{2}) \notin \mathfrak{m}_{0}(\nu,\omega)^{reg}, \subset \mathfrak{m}(\nu^{1}+\nu^{2},\omega)^{reg} \}.
	\end{equation*}
\end{definition}
Nakajima proved that $Z^{reg}(\nu^{1},\nu^{2},\omega)$ is Lagrangian and its irreducible components are of dimension $\frac{1}{2}( {\rm{dim}} \mathfrak{m}(\nu^{1},\omega)+{\rm{dim}} \mathfrak{m}(\nu^{2},\omega) )$, $Z^{reg}(\nu^{1},\nu^{2},\omega)$ is open in  $Z(\nu^{1},\nu^{2},\omega)$.

For dimension vectors $\nu^{i},i=1,2,3$, we define the projections
\begin{equation*}
	p_{i,j}: \mathfrak{m}(\nu^{1},\omega)\times \mathfrak{m}(\nu^{2},\omega) \times \mathfrak{m}(\nu^{3},\omega) \rightarrow \mathfrak{m}(\nu^{i},\omega)\times \mathfrak{m}(\nu^{j},\omega),
\end{equation*}	
then we can define the convolution product of the Borel-Moore homology groups
\begin{equation*}
	{\rm{H}}_{\ast}( Z(\nu^{1},\nu^{2},\omega) ) \otimes 	{\rm{H}}_{\ast}( Z(\nu^{2},\nu^{3},\omega) ) \rightarrow 	{\rm{H}}_{\ast}( Z(\nu^{1},\nu^{3},\omega) )	
\end{equation*}	
\begin{equation*}
	(c_{1},c_{2}) \mapsto (p_{1,3})_{\ast}( p^{\ast}_{1,2}c_{1} \cap p^{\ast}_{2,3}c_{2})
\end{equation*}
where ${\rm{H}}_{\ast}( Z(\nu^{1},\nu^{2},\omega) )$ is the Borel-Moore homology group of the variety  $ Z(\nu^{1},\nu^{2},\omega)  $ and ${\rm{H}}_{top}( Z(\nu^{1},\nu^{2},\omega) )$ is the homology group of top degree. The fundamental classes of irreducible components  of $Z(\nu^{1},\nu^{2},\omega)$ form a basis of ${\rm{H}}_{top}( Z(\nu^{1},\nu^{2},\omega) )$. Under the convolution product, $ \bigoplus \limits_{\nu^{1},\nu^{2}}{\rm{H}}_{\ast}( Z(\nu^{1},\nu^{2},\omega) ) $ becomes a $\mathbb{Q}$-algebra and $ \bigoplus \limits_{\nu^{1},\nu^{2}}{\rm{H}}_{top}( Z(\nu^{1},\nu^{2},\omega) ) $ is a subalgebra.

For given $\nu^{0}$ and $x\in \mathfrak{m}_{0}(\nu^{0},\omega)^{reg} \subset \mathfrak{m}(\nu,\omega) $, let $\mathfrak{m}(\nu,\omega)_{x}$ be the fiber at $x$ of the morphism $\pi:\mathfrak{m}(\nu,\omega) \rightarrow \mathfrak{m}_{0}(\nu,\omega)$. In particular, for $x=0$, we have 
\begin{equation*}
	\mathfrak{m}(\nu,\omega)_{0}=\mathfrak{L}(\nu,\omega)
\end{equation*}
Under the action defined by convolution, $\bigoplus \limits_{\nu}{\rm{H}}_{\ast}( \mathfrak{m}(\nu,\omega)_{x}) $ becomes a $ \bigoplus \limits_{\nu^{1},\nu^{2}}{\rm{H}}_{\ast}( Z(\nu^{1},\nu^{2},\omega) ) $-module. Moreover, $\bigoplus \limits_{\nu}{\rm{H}}_{top}( \mathfrak{m}(\nu,\omega)_{x}) $ becomes a $ \bigoplus \limits_{\nu^{1},\nu^{2}}{\rm{H}}_{top}( Z(\nu^{1},\nu^{2},\omega) ) $-module.

Fix $k \in I$ and consider graded spaces $\mathbf{V}^{2}$ such that $|\mathbf{V}^{2}|=\nu$ and $\mathbf{V}^{1}$ such that $|\mathbf{V}^{1}|=\nu-k$. We set $\nu^{2}=\nu,\nu^{1}=\nu-k$.
Nakajima introduced the Hecke correspondence $\beta_{k}(\nu,\omega)$, which is a nonsingular Lagrangian subvariety of $\mathfrak{m}(\nu^{1},\omega)\times\mathfrak{m}(\nu^{2},\omega)$(See details in \cite[Section 5]{MR1604167}).

For an orientation, let $\mathbf{A}_{\Omega}$ be the matrix such that $(\mathbf{A}_{\Omega})_{i,j}$ is the number of $h \in \Omega$ such that $h'=i,h''=j$. Let $\mathbf{C}_{\Omega}$ be the matrix defined by $\mathbf{I}-\mathbf{A}_{\Omega}$. For the opposite orientation $\bar{\Omega}$,  the matrices $\mathbf{A}_{\bar{\Omega}}$ and $\mathbf{C}_{\bar{\Omega}}$ can be defined similarly. Then Nakajima  defined $E_{k}\in \prod \limits_{\nu^{1}}{\rm{H}}_{top}( Z(\nu^{1},\nu^{1}+k,\omega) )$ to be the following formal sum
$$E_{k}=(-1)^{\langle h_{k} , \mathbf{C}_{\bar{\Omega}}\nu \rangle} [\beta_{k}(\nu,\omega)],  $$
and defined $F_{k}\in \prod \limits_{\nu^{1}}{\rm{H}}_{top}( Z(\nu^{1},\nu^{1}-k,\omega) )$ to be the following formal sum
$$F_{k}=(-1)^{\langle h_{k} , \omega- \mathbf{C}_{\Omega}\nu \rangle} [sw(\beta_{k}(\nu,\omega))], $$ 
where $sw: \mathfrak{m}(\nu',\omega) \times \mathfrak{m}(\nu,\omega) \rightarrow \mathfrak{m}(\nu,\omega) \times \mathfrak{m}(\nu',\omega) $ is the exchange of factors. Notice that the sign twists in  \cite{MR1604167} is incorrect and one should use the sign twists in formulas of \cite[9.3.2]{nakajima2001quiver}.

Nakajima proved the following main theorem.
\begin{theorem}\cite[Theorem 10.2]{MR1604167} \label{Nmain}
	With the action by $E_{k},F_{k},k\in I$, the Borel-Moore homology group
	$\bigoplus \limits_{\nu}{\rm{H}}_{top}( \mathfrak{L}(\nu,\omega))$ becomes a (left) $\mathbf{U}(\mathfrak{g})$-module. Moreover, it is isomorphic to the highest weight module $L_{0}(\Lambda_{\omega})$ with highest weight $\Lambda_{\omega}$ and the highest weight vector $[\mathfrak{L}(0,\omega)]$, and we have a $\mathbf{U}(\mathfrak{g})$-linear isomorphism satisfying
	\begin{equation*}
		\varkappa^{\Lambda_{\omega}}:\bigoplus \limits_{\nu}{\rm{H}}_{top}( \mathfrak{L}(\nu,\omega)) \rightarrow L_{0}(\Lambda_{\omega})
	\end{equation*}
	\begin{equation*}
		[\mathfrak{L}(0,\omega)] \mapsto  v_{\Lambda_{\omega}}
	\end{equation*}
	where $\Lambda_{\omega} $ is the dominant weight such that $\langle \Lambda_{\omega}, \alpha_{i}^{\vee} \rangle=\omega_{i},i \in I$.
\end{theorem}

The quiver variety $\mathfrak{m}(\nu,\omega)$ admits a partition
$\bigcup\limits_{r\geqslant 0}\mathfrak{m}_{k,r}(\nu,\omega)=\mathfrak{m}(\nu,\omega)$, where
\begin{equation*}
	\mathfrak{m}_{k,r}(\nu,\omega)=\{ [B,i,j]\in \mathfrak{m}(\nu,\omega)|{\rm{codim}} (\sum\limits_{h \in H, h''=k} {\rm{Im}}B_{h} + {\rm{Im}}j_{k})=r \}.
\end{equation*}
We also set $\mathfrak{m}_{k,\leqslant r}(\nu,\omega)= \bigcup\limits_{s\leqslant r} \mathfrak{m}_{k,s}(\nu,\omega)$. There is a natural map defined by 
\begin{equation*}
	p:\mathfrak{m}_{k, r}(\nu,\omega) \rightarrow \mathfrak{m}_{k, 0}(\nu-rk,\omega);
\end{equation*}
\begin{equation*}
	[B,i,j]\mapsto [B,i,j]|_{\sum\limits_{h \in H, h''=k}{\rm{Im}}B_{h} + {\rm{Im}}j_{k} \oplus \bigoplus \limits_{i \neq k}\mathbf{V}_{i}  }.
\end{equation*}
By definition, one can check that each fiber of $p$ is isomorphic to a Grassmannian, hence $p$ is smooth with connected fiber.

For any irreducible component $X \subset \mathfrak{m}(\nu,\omega)_{x} $  and $k \in I$, there exists a unique integer $r$ such that $X \cap \mathfrak{m}_{k,r}(\nu,\omega)_{x}$ is dense  in $X$ and we define $r=t_{k}(X)$.

Since $p:\mathfrak{m}_{k, r}(\nu,\omega) \rightarrow \mathfrak{m}_{k, 0}(\nu-rk,\omega)$ is smooth with connected fibers, $p$ induces a bijection between the set of irreducible components of $\mathfrak{m}_{k, r}(\nu,\omega)_{x}$ and  the set of irreducible components  of $\mathfrak{m}_{k, 0}(\nu-rk,\omega)_{x}$. It also induces a bijection $	\varrho_{k,r}$ between sets $Irr \mathfrak{m}(\nu,\omega)_{x}$ of irreducible components:
\begin{equation*}
	\varrho_{k,r}:\{X \subset  \mathfrak{m}(\nu,\omega)_{x} |t_{k}(X)=r \} \rightarrow \{X' \subset  \mathfrak{m}(\nu-rk,\omega)_{x} |t_{k}(X')=0 \}
\end{equation*}
\begin{equation*}
	X \mapsto \overline{p(X \cap \mathfrak{m}_{k,r}(\nu,\omega) )}
\end{equation*}
where $\overline{X}$ is the closure of the set $X$. 

\begin{lemma}\cite[Lemma 10.1]{MR1604167}
	If $t_{k}(X)=r >0$ and $\varrho_{k,r}(X)= \varrho_{k,r-1}(X'') $,then we have
	\begin{equation*}
		F_{k}[X'']=\pm r [X] + \sum\limits_{t_{k}(X')>r } c_{X'}[X'].
	\end{equation*}
\end{lemma}
As a corollary, we have the following lemma, which is parallel to  Lemma \ref{lkey'}.
\begin{lemma}\label{Nkey}
	For irreducible components $X \subset \mathfrak{L}(\nu,\omega)$ such that $t_{k}(X)=r>0$ and $X''=\varrho_{k,r}(X) $, we have the following equation in $\bigoplus \limits_{\nu}{\rm{H}}_{top}( \mathfrak{L}(\nu,\omega))$:
	\begin{equation*}
		F_{k}^{(r)} ([X''])=\pm [X] + \sum\limits_{t_{k}(X')>r} c_{X'}[X']
	\end{equation*}
	where $F_{k}^{(r)}= \frac{F_{k}^{r}}{r!}$ and $c_{X'}$ are constants.
\end{lemma}
\begin{proof}
	We prove by induction on $r$. If $r=1$, the lemma  trivially holds. For general $r$, we have
	\begin{equation*}
		\begin{split}
			F_{k}^{(r)} ([X''])=&\frac{1}{r}F_{k}  F_{k}^{(r-1)}([X''])\\
			=&\frac{1}{r} F_{k}  (\pm [\tilde{X}] + \sum\limits_{t_{k}(X')>r-1} b_{X'}[X'])\\
			=&\pm [X] + \sum\limits_{t_{k}(X')>r} c_{X'}[X']
		\end{split}
	\end{equation*}
	where $\tilde{X}$ is the irreducible component such that $X''=\varrho_{k,r-1}(\tilde{X})$. The second equation holds by the inductive assumption and  the third equation holds by applying the above lemma to $\tilde{X}$ and $X'$ respectively.
\end{proof}

\subsection{The left graphs and the isomorphism $\Phi^{\Lambda}$}

We recall the left graph of the canonical basis defined by Lusztig. In this and the next section, we usually omit the notation of Fourier-Deligne transforms if there is no ambiguity. For example, we denote $t_{i}(\mathcal{F}_{\hat{\Omega},\hat{\Omega}_{i}}(L) )$ by $t_{i}(L)$ and denote  $\mathcal{F}_{\hat{\Omega}_{i},\hat{\Omega}}\pi_{i,p} \mathcal{F}_{\hat{\Omega},\hat{\Omega}_{i}}$ by $\pi_{i,p}$.
\begin{definition}
	We define the left graph $\mathcal{G}_{1}=(\mathcal{V}_{1},\mathcal{E}_{1}) $ to be the $I \times \mathbb{N}_{>0}$-colored graph consisting of the set of vertices  $\mathcal{V}_{1}$ and the set of arrows $\mathcal{E}_{1}$ as follows,
	\begin{equation*}
		\mathcal{V}_{1}=\{ [L]  \in \mathcal{K}| L \in \mathcal{P}_{\mathbf{V}} , |\mathbf{V}|  \in \mathbb{N}[I] \},
	\end{equation*}
	\begin{equation*}
		\mathcal{E}_{1}=\{ [L] \xrightarrow{(k,r)} [L']|k\in I, 0 < r \in \mathbb{N}, \pi_{k,r}(L')=L \text{ for some } |\mathbf{V}| \in \mathbb{N}[I] \}.
	\end{equation*}
\end{definition}

Notice that objects in $\mathcal{Q}_{\mathbf{V},\mathbf{W}}$ are also $G_{\mathbf{W}}$-equivariant, Lusztig's key Lemma \ref{lkey'} also work for objects in $\mathcal{Q}_{\mathbf{V},\mathbf{W}}$. Now we can introduce the left graph of $\mathcal{L}(\Lambda)$, which is isomorphic to a subgraph of $\mathcal{G}_{1}$.
\begin{definition}
	We define the left graph  of $\mathcal{L}(\Lambda)$ to be the $I \times \mathbb{N}_{>0}$-colored graph $\mathcal{G}_{1}(\Lambda)=(\mathcal{V}_{1}(\Lambda),\mathcal{E}_{1}(\Lambda)) $ , which consists of the following set of vertices  $\mathcal{V}_{1}(\Lambda)$ and the set of arrows $\mathcal{E}_{1}(\Lambda)$
	\begin{equation*}
		\mathcal{V}_{1}(\Lambda)=\{ [L]  \in \mathcal{K}_{0}(\Lambda)| L \in \mathcal{P}_{\mathbf{V}}, L \ncong 0 \textrm{ in } \mathcal{L}_{\mathbf{V}}(\Lambda), |\mathbf{V}|  \in \mathbb{N}[I] \},
	\end{equation*}
	\begin{equation*}
		\mathcal{E}_{1}(\Lambda)=\{ [L] \xrightarrow{(k,r)} [L']|[L],[L']\in \mathcal{V}_{1}(\Lambda)  ,k\in I, 0 < r \in \mathbb{N}, \pi_{k,r}(L')=L \text{ for some } \mathbf{V} \}.
	\end{equation*}
\end{definition}

We say a sequence $\underline{s}=((i_{1},n_{1}),(i_{2},n_{2}),\cdots ,(i_{l},n_{l}) )$ of $I \times \mathbb{N}_{>0}$ is a left admissible path of $L \in \mathcal{P}_{\mathbf{V}}$, if $\pi_{i_{1},n_{1}} \pi_{i_{2},n_{2}}\cdots \pi_{i_{l},n_{l}}(L_{0})=L$, where $L_{0}$ is the unique simple perverse sheaf in $\mathcal{P}_{\mathbf{V}}$ with $|\mathbf{V}|=0$. By Lemma 7.2 in \cite{MR1088333}, for any $L \in \mathcal{P}_{\mathbf{V}}$, there exists a left admissible path of $L$.

Consider the affine variety $\Lambda_{\mathbf{V},\mathbf{W}} \subset \mathbf{E}_{\mathbf{V},\mathbf{W}}$
\begin{equation*}
	\Lambda_{\mathbf{V},\mathbf{W}} =\{ (B,i,j) \in \mu^{-1}(0)|j=0,B \text{ is nilpotent } \}
\end{equation*}
In particular, for $\mathbf{W}=0$, we denote $\Lambda_{\mathbf{V},0}$ by $\Lambda_{\mathbf{V}}$.

Let $\Lambda_{\mathbf{V},i,p}$ be the close subset of $\Lambda_{\mathbf{V}}$ defined by
\begin{equation*} 
	\Lambda_{\mathbf{V},i,p}=\{x\in \Lambda_{\mathbf{V}}| {\rm{codim}}_{\mathbf{V}_{i}} ( {\rm{Im}} \sum\limits_{h \in H, h''=i} x_{h}) =p\},
\end{equation*} 
then for any irreducible component $Z$, there exists a unique $p$ such that $Z \cap \Lambda_{\mathbf{V},i, p}$ is dense in $Z$ and we define $t_{i}(Z)=p$. Similarly, we can consider the close subsets 
\begin{equation*} 
	\Lambda_{\mathbf{V},i}^{p}=\{x\in \Lambda_{\mathbf{V}}| {\rm{dim}}_{\mathbf{V}_{i}} ( {\rm{Ker}} \bigoplus\limits_{h \in H, h'=i} x_{h}) =p\}
\end{equation*} 
and define $t_{i}^{\ast}(Z)=p$ if $p$ is the unique integer such that $Z \cap \Lambda_{\mathbf{V},i}^{p}$ is dense in $Z$. 

Given $\nu'+\nu''=\nu \in \mathbb{N}[I] $ and graded vector spaces $\mathbf{V},\mathbf{V}',\mathbf{V}''$ with dimension vectors $\nu, \nu', \nu''$ respectively, let  $\Lambda'$ be the set which consists of $(x,\tilde{\mathbf{W}},\rho_{1},\rho_{2})$ where $x \in \Lambda_{\mathbf{V}}$, $\tilde{\mathbf{W}}$ is a $x$-stable subspace of $\mathbf{V}$ with dimension $\nu''$ and $\rho_{1}:\mathbf{V}/\tilde{\mathbf{W}} \simeq \mathbf{V}', \rho_{2}: \tilde{\mathbf{W}} \simeq \mathbf{V}''$ are linear isomorphisms, and let $\Lambda''$ be the set which consists of $(x,\tilde{\mathbf{W}})$ as above. Consider the morphisms
\begin{center}
	$\Lambda_{\mathbf{V}'} \times \Lambda_{\mathbf{V}''} \xleftarrow{p} \Lambda' \xrightarrow{r} \Lambda'' \xrightarrow{q} \Lambda_{\mathbf{V}}$
\end{center}
where $$p(x,\tilde{\mathbf{W}},\rho_{1},\rho_{2})=(\rho_{1,\ast}(\bar{x}|_{\mathbf{V}/\tilde{\mathbf{W}}}),\rho_{2,\ast}(x|_{\tilde{\mathbf{W}}})  ),$$ $$r(x,\tilde{\mathbf{W}},\rho_{1},\rho_{2}) =(x, \tilde{\mathbf{W}}),~q(x,\tilde{\mathbf{W}})=x.$$ Note that $r$ is a principle $G_{\mathbf{V}'} \times G_{\mathbf{V}''}$-bundle and $q$ is proper. 

We assume that $|\mathbf{V}''|+pi=|\mathbf{V}|$ and denote $qr: \Lambda' \rightarrow \Lambda_{\mathbf{V}}$ by $q'$, then $$p^{-1} (\Lambda_{\mathbf{V''},i,0})= (q'^{-1})(\Lambda_{\mathbf{V},i, p}).$$ We denote $p^{-1} (\Lambda_{\mathbf{V''},i,0})$ by $\Lambda'_{i, p}$ and denote $(q^{-1})(\Lambda_{\mathbf{V},i, p})$ by $\Lambda''_{i,p}$. Then we have the following commutative diagram

\[
\xymatrix{
	\Lambda_{\mathbf{V}'',i,0}\ar[d]^{i_{1}} & \Lambda'_{i,p}\ar[d]^{i_{2}}\ar[l]_{p} \ar[r]^{r}& \Lambda''_{i,p} \ar[d]^{i_{3}}\ar[r]^{q}& \Lambda_{\mathbf{V},i, p} \ar[d]^{i_{4}}\\
	\Lambda_{\mathbf{V}''} & \Lambda' \ar[l]_{p} \ar[r]^{r} & \Lambda''  \ar[r]^{q} & \Lambda_{\mathbf{V}}
}
\]
here $i_{1},i_{2},i_{3}$ and $i_{4}$ are the natural embeddings.

\begin{lemma}\cite[4.17]{MR3202708}
	(1)The morphism $q': \Lambda'_{i,p} \rightarrow \Lambda_{\mathbf{V},i, p}$ is a principle $G_{\mathbf{V''}} \times G_{\mathbf{V}'}$ -bundle.\\
	(2)The morphism $p: \Lambda'_{i, p} \rightarrow \Lambda_{\mathbf{V}'',i,0}$ is a smooth map whose fibers are connected of dimension $(\sum \limits_{i \in I} \nu_{i}^{2} )- p(\nu'',i )$.
\end{lemma}
\begin{corollary}
	There is a bijection $\eta_{i,p}:Irr \Lambda_{\mathbf{V}'',i,0} \rightarrow  Irr \Lambda_{\mathbf{V},i, p} $ from the set of irreducible components of $\Lambda_{\mathbf{V}'',i,0}$ to the set of irreducible components $\Lambda_{\mathbf{V},i, p}$ induced by $q'p^{-1}$. The map $\eta_{i,p}$ also induces a bijection ( still denoted by $\eta_{i,p}$) from $\{Z \in Irr\Lambda_{\mathbf{V}''}| t_{i}(Z)=0 \}$ to $\{ Z \in Irr \Lambda_{\mathbf{V}}|t_{i}(Z)=p \} $ defined by: $$\eta_{i,p}( \bar{Z})= \overline{\eta_{i,p}(Z)}$$ for $Z \in Irr \Lambda_{\mathbf{V}'',i,0}$. Here $\bar{Z}$ and $\overline{\eta_{i,p}(Z)}$ are the closure of $Z$ and $\eta_{i,p}(Z)$ respectively.
\end{corollary}

Now we can define the left graph $\mathcal{G}_{2}$ for $Irr \Lambda_{\mathbf{V}}$.
\begin{definition}
	We define the left graph $\mathcal{G}_{2}=(\mathcal{V}_{2},\mathcal{E}_{2}) $ to be the $I \times \mathbb{N}_{>0}$-colored graph consisting of the following set of vertices  $\mathcal{V}_{2}$ and the set of arrows $\mathcal{E}_{2}$ as follows,
	\begin{equation*}
		\mathcal{V}_{2}=\{ Z | Z \in Irr \Lambda_{\mathbf{V}} , |\mathbf{V}|  \in \mathbb{N}[I] \},
	\end{equation*}
	\begin{equation*}
		\mathcal{E}_{2}=\{ Z \xrightarrow{(k,r)} Z'|k\in I, 0 < r \in \mathbb{N}, \eta_{k,r}(Z')=Z \text{ for some } |\mathbf{V}| \in \mathbb{N}[I] \}.
	\end{equation*}
\end{definition}

Similarly, we say a sequence $\underline{s}=((i_{1},n_{1}),(i_{2},n_{2}),\cdots ,(i_{l},n_{l}) )$ of $I \times \mathbb{N}_{>0}$ is a left admissible path of $Z \in Irr \Lambda_{\mathbf{V}}$, if $\eta_{i_{1},n_{1}} \eta_{i_{2},n_{2}}\cdots \eta_{i_{l},n_{l}}(Z_{0})=Z$, where $Z_{0}$ is the unique irreducible component of $\Lambda_{0}$. By Corollary 1.6 in \cite{MR1758244}, for any  $Z \in Irr \Lambda_{\mathbf{V}}$, there exists a left admissible path of $Z$.
Then by \cite{MR1458969} or Theorem 7.1 in \cite{fang2022correspondence}, we have the following result:
\begin{proposition}
	There is a bijection $\Phi: \mathcal{V}_{1} \rightarrow \mathcal{V}_{2}$ such that $\Phi$ commutes with the arrows in $\mathcal{E}_{1}$ and $\mathcal{E}_{2}$ and $\Phi$ preserves the value $t_{i}$ and $t^{\ast}_{i}$.
	\begin{equation*}
		\Phi([L]) \xrightarrow{(k,r)} \Phi([L']) \iff [L] \xrightarrow{(k,r)} [L'],k \in I, 0< r\in \mathbb{N}; 
	\end{equation*}
	\begin{equation*}
		t_{i}(L)=t_{i}(\Phi([L])),t^{\ast}_{i}(L)=t^{\ast}_{i}(\Phi([L])), i\in I.
	\end{equation*}
	Moreover, if $\underline{s}=((i_{1},n_{1}),(i_{2},n_{2}),\cdots ,(i_{l},n_{l}) )$ is a left admissible path of $L$, then  $$\Phi([L])=\eta_{i_{1},n_{1}} \eta_{i_{2},n_{2}}\cdots \eta_{i_{l},n_{l}}(Z_{0}).$$
\end{proposition}

We can also define the left graph $\mathcal{G}_{0}(\Lambda_{\omega})$ for Nakajima's quiver variety, which is an analogy of  $\mathcal{G}_{1}(\Lambda)$.
\begin{definition}
	We define the left graph $\mathcal{G}_{0}(\Lambda_{\omega})=(\mathcal{V}_{0},\mathcal{E}_{0}) $ of Nakajima's quiver variety (associated with $\omega$) to be the $I \times \mathbb{N}_{>0}$-colored graph, which consists of the following set of vertices  $\mathcal{V}_{0}$ and the set of arrows $\mathcal{E}_{0}$
	\begin{equation*}
		\mathcal{V}_{0}=\{ X \in Irr \mathfrak{L}(\nu,\omega)| \nu \in \mathbb{N}[I] \},
	\end{equation*}
	\begin{equation*}
		\mathcal{E}_{0}=\{ X' \xrightarrow{(k,r)} X|k\in I, 0 < r \in \mathbb{N}, \varrho_{k,r}(X)=X' \text{ for some } \nu \in \mathbb{N}[I] \}.
	\end{equation*}
\end{definition}

\begin{lemma}\cite[Lemma 5.8]{MR1302318}
	(1) $\mathfrak{L}(\nu,\omega)$ is isomorphic to the geometric quotient of $\Lambda_{\mathbf{V},\mathbf{W}} \cap \mu^{-1}(0)^{s}$. In particular, the set $Irr \mathfrak{L}(\nu,\omega)$ is bijective to the set of $G_{\mathbf{V}}$-invariant Lagrangian irreducible components of $\Lambda_{\mathbf{V},\mathbf{W}}\cap \mu^{-1}(0)^{s}$.\\
	(2) The projection $\pi_{\mathbf{W}}:\Lambda_{\mathbf{V},\mathbf{W}} \rightarrow \Lambda_{\mathbf{V}};(B,i,0) \mapsto B$
	is a vector bundle. In particular, the set $Irr \Lambda_{\mathbf{V},\mathbf{W}}$ is bijective to the set $ Irr \Lambda_{\mathbf{V}}$. 
\end{lemma}
Note that the set of $G_{\mathbf{V}}$-invariant Lagrangian irreducible components of $\Lambda_{\mathbf{V},\mathbf{W}}\cap \mu^{-1}(0)^{s}$ can be naturally regarded as a subset of $Irr \Lambda_{\mathbf{V},\mathbf{W}}$. With the lemma above, consider an injective map $\Psi:Irr \mathfrak{L}(\nu,\omega) \rightarrow  Irr \Lambda_{\mathbf{V}}$ defined by composing the two bijections above. If $\Lambda=\Lambda_{\omega}$, then for some $\nu \in \mathbb{N}[I]$ and $X \in Irr \mathfrak{L}(\nu,\omega)$, we define
$\Phi^{\Lambda}(X)= (\Phi)^{-1} (\Psi(X)) $. By definition, $\Phi^{\Lambda}:\mathcal{V}_{0} \rightarrow \mathcal{V}_{1} $ is injective.

\begin{theorem}\label{thm2}
	The image of $\Phi^{\Lambda}$ is contained in $\mathcal{V}_{1}(\Lambda)$. Moreover, $\Phi^{\Lambda}: \mathcal{V}_{0} \rightarrow \mathcal{V}_{1}(\Lambda)$ is exactly an isomorphism of ($I \times \mathbb{N}_{>0}$-colored) left graphs,
	\begin{equation*}
		X' \xrightarrow{(k,r)} X \iff 	\Phi^{\Lambda}(X') \xrightarrow{(k,r)} \Phi^{\Lambda}(X) ,k \in I, 0< r\in \mathbb{N}.
	\end{equation*}
\end{theorem}

\begin{proof}
	We firstly prove that ${\rm{Im}}(\Phi^{\Lambda}) \subset \mathcal{V}_{1}(\Lambda)$. Assume that  $\Phi^{\Lambda}(X)=[L]$ and $L \in \mathcal{N}_{\mathbf{V}}$ for some $X$, then there exists some $i\in I$ such that $L$ in $\mathcal{N}_{\mathbf{V},i}$. Then by the proof of Theorem \ref{thm1},  a simple perverse sheaf $L$, viewed as an object in $\mathcal{Q}_{\mathbf{V}}$ via the inverse of $(\pi_{\mathbf{W}})^{\ast}$, satisfies $t^{\ast}_{i}(L)=r >d_{i}$, hence the irreducible component $Z=\Phi(L) \subset \Lambda_{\mathbf{V}}$ satisfies $t_{i}^{\ast}(Z)=r> d_{i}$. Let $Z'=\pi_{\mathbf{W}}^{-1}(Z)$ be the irreducible component of $\Lambda_{\mathbf{V},\mathbf{W}}$ corresponding to $Z$. By definition, $\pi_{\mathbf{W}}^{-1}(\Lambda_{\mathbf{V},i}^{r} \cap Z)$ is dense in $Z'$.
	We claim that $\pi_{\mathbf{W}}^{-1}(\Lambda_{\mathbf{V},i}^{r} \cap Z) \cap \mu^{-1}(0)^{s}$ is empty. In fact, for any $(B,i,0)\in  \pi_{\mathbf{W}}^{-1}(\Lambda_{\mathbf{V},i}^{r} \cap Z)$, we consider $S={\rm{Ker}} \bigoplus \limits_{h \in H, h'=i} B_{h} $. $S$ is a $B$-stable subspace, and $S \cap {\rm{Ker}} i \neq 0$. (Otherwise, $i|_{S}$ is an injective linear map from the $r$-dimensional space $S$ to the $d_{i}$-dimensional space $W_{\hat{i}}$). Hence $S \cap {\rm{Ker}} i \subset {\rm{Ker}} i $ fails the stability condition (S). However, $Z' \cap \mu^{-1}(0)^{s}$ is dense in $Z'$. We get a contradiction and have proved that ${\rm{Im}}(\Phi^{\Lambda}) \subset \mathcal{V}_{1}(\Lambda)$.
	
	Notice that by Theorem \ref{thm1}, we have
	\begin{equation*}
		|\mathcal{P}_{\mathbf{V}}\backslash \mathcal{N}_{\mathbf{V}}|={\rm{dim}}_{\mathbb{Q}(v)} L_{|\mathbf{V}|}(\Lambda)
	\end{equation*}
	where $\mathcal{P}_{\mathbf{V}}\backslash \mathcal{N}_{\mathbf{V}}$ is the set of nonzero simple perverse sheaves in $\mathcal{L}_{\mathbf{V}}(\Lambda)$. And by Theorem \ref{Nmain}, we have
	\begin{equation*}
		{\rm{dim}}_{\mathbb{Q}(v)} L_{|\mathbf{V}|}(\Lambda)={\rm{dim}}_{\mathbb{Q}} L_{0,|\mathbf{V}|}(\Lambda) = |Irr\mathfrak{L}(\nu,\omega)|
	\end{equation*} 
	hence $\Phi^{\Lambda}: \mathcal{V}_{0} \rightarrow \mathcal{V}_{1}(\Lambda)$ is a bijection. Notice that $\Psi$ commutes with $\eta^{-1}_{k,r}$ and $\varrho_{k,r}$ by definition. In fact, both  $\eta^{-1}_{k,r}$ and $\varrho_{k,r}$ can be obtained by restricting $B$ to $\sum\limits_{h \in H, h''=k}{\rm{Im}}B_{h} \oplus \bigoplus \limits_{i \neq k}\mathbf{V}_{i} $ for generic $B$. Hence  $\Phi^{\Lambda}$ commutes with the arrows. More precisely, we have
	\begin{equation*}
		X' \xrightarrow{(k,r)} X \iff 	\Phi^{\Lambda}(X') \xrightarrow{(k,r)} \Phi^{\Lambda}(X) ,k \in I, 0< r\in \mathbb{N},
	\end{equation*}
	as desired.
\end{proof}
\nocite{MR1942245}

\subsection{Monomial bases from left graphs}
Let $\mathcal{S}$ be the set of sequences of $I \times \mathbb{N}_{>0}$.
\begin{definition}
	For a given order $(i_{1}\prec i_{2} \prec \cdots \prec
	i_{n})$ of $I$, we inductively define a map $\underline{s}^{\prec}=\underline{s}: \mathcal{V}_{1}(\Lambda) \rightarrow \mathcal{S}$ as follows: (1) If $L$ is the unique simple perverse sheaf on $\mathbf{E}_{\mathbf{V},\mathbf{W},\hat{\Omega}},|\mathbf{V}|=pi,p \leqslant d_{i} $,  define $\underline{s}([L])=((i,p))$. In particular $\underline{s}([L_{0}])=\emptyset $. (2) For the other $L \in \mathcal{V}_{1}(\Lambda)$, there exists a unique $r \in \mathbb{N}$ such that $t_{i_{r}}(L)=n>0$ but $t_{i_{s}}(L)=0$ holds for any $r<s$, take $K$ such that $L \xrightarrow{(i_{r},n)} K$ and  define $\underline{s}([L])=((i_{r},n),\underline{s}([K]))$.
\end{definition}

\begin{lemma}
	The map $\underline{s}^{\prec}=\underline{s}: \mathcal{V}_{1}(\Lambda) \rightarrow \mathcal{S}$ is well-defined and injective.
\end{lemma}
\begin{proof}
	It suffices to show that the inductive definition can be continued.
	
	If $t_{i_{r}}(L)=n>0$ in $\mathcal{V}_{1}(\Lambda)$, but $K$ such that $L \xrightarrow{(i_{r},n)} K$ belongs to $\mathcal{N}_{\mathbf{V}'}$. Then there exists $j$ such that $s_{j}^{\ast}(K)> 0$ and $K$ is a direct summand of $L_{(\boldsymbol{\nu},j^{d})},d>0$. By definition of $\pi_{i_{r},n}$, we know that $L$ is a  direct summand of $L_{(i_{r}^{n},\boldsymbol{\nu},j^{d})}$ and belongs to $\mathcal{N}_{\mathbf{V},j}$, a contradiction.
\end{proof}

Similarly, we can inductively define $\underline{s}':\mathcal{V}_{0} \rightarrow \mathcal{S}$. By definition, $\underline{s}([L])$ is a left admissible path of $L$ for any $L \in \mathcal{P}_{\mathbf{V}}\backslash \mathcal{N}_{\mathbf{V}}$ and $\underline{s}'(X)$ is a left admissible path of $X$ for any $X\in Irr \mathfrak{L}(\nu,\omega)$.

\begin{definition}
	Given two sequences $$\underline{s}_{1}=((i_{n_{1}},m_{1} ), (i_{n_{2}},m_{2}),\cdots,(i_{n_{k}},m_{k}) ),\ \underline{s}_{2}=((i_{n'_{1}},m'_{1} ), (i_{n'_{2}},m'_{2}),\cdots,(i_{n'_{l}},m'_{l}) )  \in \mathcal{S}$$ such that $\sum \limits_{1 \leqslant s \leqslant k}m_{s} i_{n_{s}}=\sum\limits_{1 \leqslant s \leqslant l}m'_{s}i_{n'_{s}},$ we say $\underline{s}_{2}\prec \underline{s}_{1}$ if there exists $r\in \mathbb{N}$ such that $(i_{n_{t}},m_{t})=(i_{n'_{t}},m'_{t})$ for $1 \leqslant t< r$ and $(i_{n'_{r}},m'_{r})\prec (i_{n_{r}},m_{r})$ with respect to the lexicographical order. 
\end{definition}
Recall that $\mathcal{S}_{|\mathbf{V}|}$ is naturally bijective to  the subset of $\mathcal{S}$ which consists of the sequences $\underline{s}=((i_{n_{1}},m_{1} ), (i_{n_{2}},m_{2}),\cdots,(i_{n_{k}},m_{k}) )$ such that $\sum \limits_{1 \leqslant s \leqslant k}m_{s} i_{n_{s}}=|\mathbf{V}|$, then $(\mathcal{S}_{|\mathbf{V}|},\prec)$ becomes a partially ordered set. Regraded as  subsets of $\mathcal{S}$ via $\underline{s}$ and $\underline{s}'$, $\mathcal{V}_{1}(\Lambda)$ and  $\mathcal{V}_{0}$ also become  partially ordered sets. We still denote their partial order by $\prec$.

For any sequence
\begin{equation*}
	\underline{s}=((i_{n_{1}},m_{1} ), (i_{n_{2}},m_{2}),\cdots,(i_{n_{k}},m_{k}))  \in \mathcal{S},
\end{equation*}
we define
\begin{equation*}
	m^{\Lambda}_{\underline{s}}=F_{i_{n_{1}}}^{(m_{1})}   F_{i_{n_{2}}}^{(m_{2})}   \cdots   F_{i_{n_{k}}}^{(m_{k})} [\mathfrak{L}(0,\omega)] \in \bigoplus \limits_{\nu}{\rm{H}}_{top}( \mathfrak{L}(\nu,\omega)) 
\end{equation*} 
and 
\begin{equation*}
	M^{\Lambda}_{\underline{s}}=F_{i_{n_{1}}}^{(m_{1})} F_{i_{n_{2}}}^{(m_{2})} \cdots F_{i_{n_{k}}}^{(m_{k})}[L_{0}] \in \mathcal{K}_{0}(\Lambda).
\end{equation*}

\begin{proposition}
	For fixed $|\mathbf{V}|=\nu$, we set
	\begin{equation*}
		\mathbf{M}^{\Lambda}_{\mathbf{V}}= \{M^{\Lambda}_{\underline{s}([L])}|[L] \in \mathcal{V}_{1}(\Lambda) , L \in \mathcal{P}_{\mathbf{V}} \},
	\end{equation*}
	\begin{equation*}
		\mathbf{M'}^{\Lambda}_{\mathbf{V}}= \{m^{\Lambda}_{\underline{s}'(X)}|X \in Irr \mathfrak{L}(\nu,\omega)  \}
	\end{equation*}
	then we have \\
	(1) The set $\mathbf{M}^{\Lambda}_{\mathbf{V}}$ is an $\mathcal{A}$-basis of $\mathcal{K}_{0,\nu}(\Lambda)$.\\
	(2) The set $\mathbf{M'}^{\Lambda}_{\mathbf{V}}$ is a  $\mathbb{Q}$-basis of ${\rm{H}}_{top}( \mathfrak{L}(\nu,\omega))$.\\
	(3) The transition matrix from $\mathbf{B}^{\Lambda}_{1,\mathbf{V}}=\{[L]
	|L \in \mathcal{V}_{1}(\Lambda) ,L \in \mathcal{P}_{\mathbf{V}} \}$ to $\mathbf{M}^{\Lambda}_{\mathbf{V}}$ is upper triangular  (with respect to $\prec$) and with diagonal entries all equal to 1.\\
	(4) The transition matrix from $\mathbf{B}^{\Lambda}_{2,\mathbf{V}}=\{[X]|X \in Irr \mathfrak{L}(\nu,\omega)   \}$ to  $\mathbf{M'}^{\Lambda}_{\mathbf{V}}$ is upper triangular  (with respect to $\prec$) and with diagonal entries all equal to $\pm 1$.
\end{proposition}

\begin{proof}
	We claim that for any $[L] \in \mathcal{V}_{1}(\Lambda)$,
	\begin{equation*}
		M^{\Lambda}_{\underline{s}([L])}=[L]+\sum \limits_{\underline{s}([L'])\succ \underline{s}([L]) } c_{L,L'}[L']
	\end{equation*}
	where $c_{L,L'}$ are constants in $\mathcal{A}$. We argue by induction on the length $k$ of $$\underline{s}([L])=((i_{n_{1}},m_{1} ), (i_{n_{2}},m_{2}),\cdots,(i_{n_{k}},m_{k})).$$ If $k=1$ and $\underline{s}([L])=((i_{n_{1}},m_{1} ))$, then $\nu= m_{1}i_{n_{1}}$ and 
	\begin{equation*}
		[L_{i_{n_{1}}^{m_{1}} }] =F_{i_{n_{1}}}^{(m_{1})}[L_{0}]=M^{\Lambda}_{\underline{s}([L])}
	\end{equation*}
	holds trivially. If $k>1$, by Lemma \ref{lkey'}, we can take $L' \in \mathcal{P}_{\mathbf{V}'}$ ,$|\mathbf{V}|= |\mathbf{V'}|+m_{1}i_{n_{1}}$  such that $t_{i_{1}}(L')=0, \pi_{i_{1},m_{1}}(L')=L$, and then
	\begin{equation*}
		F_{i_{n_{1}}}^{(m_{1})}[L']=[L]+\sum\limits_{t_{i_{n_{1}}}(L'')>m_{1} } c_{L',L''}[L'']
	\end{equation*}
	By the inductive assumption, 
	\begin{equation*}
		M^{\Lambda}_{\underline{s}([L'])}=[L']+\sum \limits_{\underline{s}([L''])\succ \underline{s}([L']) }c_{L',L''}[L''].
	\end{equation*}
	Notice that $\underline{s}([L])=((i_{n_{1}},m_{1}), \underline{s}([L'])) $, we have  
	\begin{equation*}
		\begin{split}	
			M^{\Lambda}_{\underline{s}([L])}=& F_{i_{n_{1}}}^{(m_{1})} M^{\Lambda}_{\underline{s}([L'])} \\
			=& F_{i_{n_{1}}}^{(m_{1})} ([L']+\sum \limits_{\underline{s}([L''])\succ \underline{s}([L']) }c_{L',L''}[L''] ) \\
			=&[L]+\sum\limits_{t_{i_{n_{1}}}(L'')>m_{1} } c_{L',L''}[L'']+\sum \limits_{\underline{s}([L''])\succ \underline{s}([L']) }c_{L',L''}F_{i_{n_{1}}}^{(m_{1})}[L''].
		\end{split}
	\end{equation*}
	Notice that if $t_{i_{n_{1}}}(L'')=m>m_{1}$, then $\underline{s}([L''])$ either starts with $(i_{n_{1}},m)$ or starts  with $(i_{r},m')$ for some $r<n_{1}$, hence  
	\begin{center}
		$(\clubsuit)\ \ \underline{s}([L''])\succ \underline{s}([L])$ holds for $L''$ which satisfies $t_{i_{n_{1}}}(L'')=m>m_{1}$.
	\end{center}

	We only need to show
	\begin{equation*}
		F_{i_{n_{1}}}^{(m_{1})}[L'']= \sum \limits_{\underline{s}([L'''])\succ \underline{s}([L]) } d_{L'''}[L''']
	\end{equation*}
	for those $L''$ which satisfy $\underline{s}([L''])\succ \underline{s}([L'])$. For those $L''$ such that $t_{i_{n_{1}}}(L'')>0$, by Proposition \ref{lt'}, $[L'']$ belongs to the $\mathcal{A}$- submodule  
	$F_{i_{n_{1}}} \mathcal{K}_{0}(\Lambda)$,
	hence
	$F_{i_{n_{1}}}^{(m_{1})}  [L'']$ belongs to $F_{i_{n_{1}}}^{(m_{1}+1)} \mathcal{K}_{0}(\Lambda).$ By Proposition \ref{lt'}, $\{[L]|L\in \mathcal{P}, t_{i_{n_{1}}}(L) \geqslant m_{1}+1  \}$ form a basis of $[L_{i_{n_{1}}^{(m_{1}+1)}}] \mathcal{K}$ in  $\mathcal{K}$. Project this basis to $\mathcal{K}_{0}(\Lambda)$, we can see that the set $$\{[L]|L\in \mathcal{P}, L \notin \mathcal{N}, t_{i_{n_{1}}}(L) \geqslant m_{1}+1  \}$$ forms a basis of $F_{i_{n_{1}}}^{(m_{1}+1)} \mathcal{K}_{0}(\Lambda)$.
	\begin{equation*}
		F_{i_{n_{1}}}^{(m_{1})}  [L'']=\sum \limits_{t_{i_{n_{1}}}(L''')>m_{1} } e_{L'''}[L'''].
	\end{equation*}
	By $(\clubsuit)$, those $L'''$ satisfy $t_{i_{n_{1}}}(L''')=m>m_{1}$, hence  $\underline{s}([L'''])\succ \underline{s}([L])$. For those $L''$ such that $t_{i_{n_{1}}}(L'')=0$,
	\begin{equation*}
		F_{i_{n_{1}}}^{(m_{1})}  [L'']=[\tilde{L}] + \sum \limits_{t_{i_{n_{1}}}(L''')>m_{1} } e_{L'''}[L''']
	\end{equation*} 
	where $\tilde{L}$ is  a simple perverse sheaf such that $t_{i_{n_{1}}}(\tilde{L})=m_{1}$. By $(\clubsuit)$,  $\underline{s}([L'''])\succ \underline{s}([L])$. So we only need to show $\underline{s}([\tilde{L}])\succ \underline{s}([L])$. In fact, if $t_{i_{r}}(\tilde{L}) >0$ holds for some $r< n_{1}$, then $\underline{s}([\tilde{L}])$ starts with $(i_{r},m'),r<n_{1}$ and $\underline{s}([\tilde{L}])\succ \underline{s}([L])$ holds by definition. Otherwise, we have $\underline{s}([\tilde{L}])=((i_{n_{1}},m_{1}),\underline{s}([L''])$. Since $\underline{s}([L''])\succ \underline{s}([L'])$, 
	\begin{equation*}
		\underline{s}([\tilde{L}])=((i_{n_{1}},m_{1}),\underline{s}([L''])) \succ ((i_{n_{1}},m_{1}),\underline{s}([L']))=\underline{s}([L]).
	\end{equation*}
	In conclusion, we have proved our claim. Then (1) and (3) follow by basic linear algebra and our claim. Similarly, we can prove
	\begin{equation*}
		m^{\Lambda}_{\underline{s}'(X)}=\pm[X]+\sum \limits_{\underline{s}'(X')\succ \underline{s}'(X) } c_{X,X'}[X'] 
	\end{equation*}
	where $c_{X,X'}$ are constants in $\mathbb{Q}$. Then (2) and (4) hold. 
\end{proof}

Recall that there are isomorphisms $$\varsigma^{\Lambda}: \mathcal{K}_{0}({\Lambda}) \rightarrow {_{\mathcal{A}}L(\Lambda)},~ \varkappa^{\Lambda}:\bigoplus \limits_{\nu}{\rm{H}}_{top}( \mathfrak{L}(\nu,\omega)) \rightarrow L_{0}(\Lambda),$$ and we still denote the composition of $\varsigma^{\Lambda}: \mathcal{K}_{0}(\Lambda) \rightarrow {_{\mathcal{A}}L(\Lambda)}$ and the classical limits by $$\varsigma^{\Lambda}: \mathbb{Z} \otimes _{\mathcal{A}}\mathcal{K}_{0}(\Lambda) \rightarrow  {_{\mathbb{Z}}L_{0}(\Lambda)}.$$

\begin{definition}
	We define $sgn=sgn^{\succ}: \bigcup \limits_{\nu} Irr \mathfrak{L}(\nu,\omega) \rightarrow \{\pm 1\} $ to be the unique function which satisfies
	\begin{equation*}
		m^{\Lambda}_{\underline{s}'(X)}=sgn(X)[X]+\sum \limits_{\underline{s}'(X')\succ \underline{s}'(X) } c_{X,X'}[X'].
	\end{equation*}
\end{definition}	

For the integrable highest weight module $L_{0}(\Lambda)$ of the universal enveloping algebra, we define 
\begin{equation*}
	\tilde{\mathbf{B}}^{\Lambda}_{1,\mathbf{V}}=\{\varsigma^{\Lambda}([L]
	)|L \in \mathcal{V}_{1}(\Lambda) ,L \in \mathcal{P}_{\mathbf{V}} \};\tilde{\mathbf{B}}_{1}=\tilde{\mathbf{B}}^{\Lambda}_{1}=\bigcup \limits_{\mathbf{V}}\tilde{\mathbf{B}}^{\Lambda}_{1,\mathbf{V}}
\end{equation*}
\begin{equation*}
	\tilde{\mathbf{B}}^{\Lambda}_{2,\mathbf{V}}=\{ \varkappa^{\Lambda}(sgn(X)[X])|X \in Irr \mathfrak{L}(\nu,\omega) \};  \tilde{\mathbf{B}}_{2}=\tilde{\mathbf{B}}^{\Lambda}_{2}=\bigcup \limits_{\mathbf{V}}\tilde{\mathbf{B}}^{\Lambda}_{2,\mathbf{V}}
\end{equation*}

\begin{theorem}
	The transition matrix from $\tilde{\mathbf{B}}_{1}$ to $\tilde{\mathbf{B}}_{2}$ is upper triangular  (with respect to $\prec$) and with diagonal entries all equal to 1.
	More precisely, if $X \in \mathcal{V}_{0}$ and $[L]=\Phi^{\Lambda}(X)$, then we have 
	\begin{equation*}
		\varsigma^{\Lambda}([L])=\varkappa^{\Lambda}(sgn(X)[X]) +\sum\limits_{\underline{s}'(X')\succ \underline{s}'(X) } c_{X'}\varkappa^{\Lambda} (sgn(X')[X'])
	\end{equation*}
	where $c_{X'} \in \mathbb{Q}$ are constants.
\end{theorem} 
\begin{proof}
	Notice that $\underline{s}(\Phi^{\Lambda}(X))=\underline{s}'(X)$, we have 
	\begin{equation*}
		\begin{split}
			\varkappa^{\Lambda}(m^{\Lambda}_{\underline{s}'(X)})=&F_{i_{n_{1}}}^{(m_{1})}F_{i_{n_{2}}}^{(m_{2})}\cdots F_{i_{n_{k}}}^{(m_{k})} v_{\Lambda} \\
			= &	\varsigma^{\Lambda} ( M^{\Lambda}_{\underline{s}([L])})
		\end{split}
	\end{equation*}
	hence $\varkappa^{\Lambda}(\mathbf{M'}^{\Lambda})=\varsigma^{\Lambda}(\mathbf{M}^{\Lambda})$  is a basis of $L_{0}(\Lambda)$, denoted by $\tilde{\mathbf{M}}^{\Lambda}$. Let $P_{i},i=1,2$ be the transition matrices from $\tilde{\mathbf{M}}^{\Lambda}$ to $\tilde{\mathbf{B}}_{i},i=1,2$ respectively, then each  $P_{i}$ is upper triangular  (with respect to $\prec$) and with diagonal entries all equal to 1. Since $\Phi^{\Lambda}$ preserves the order $\prec$, the transition matrix $P_{2}P_{1}^{-1}$ from $\tilde{\mathbf{B}}_{1}$ to $\tilde{\mathbf{B}}_{2}$ is also upper triangular and with diagonal entries all equal to 1.
\end{proof}

\begin{remark}
	We can see that the theorem does not depend on the choice of the order of $I$. More precisely, if we choose another order $\tilde{\prec}$ of $I$, then $\tilde{\prec}$ induces an order of $\mathcal{S}$. We can also define $\underline{s}^{\tilde{\prec}}:\mathcal{V}_{1} \rightarrow \mathcal{S}$ and $\underline{s}'^{\tilde{\prec}}:\mathcal{V}_{2} \rightarrow \mathcal{S}$ in a similar way. Then with the notations above, we still have $\varsigma([L])=\varkappa(f_{Z}) +\sum\limits_{\underline{s}'^{\tilde{\prec}}(f_{Z'})\tilde{\succ} \underline{s}'^{\tilde{\prec}}(f_{Z}) } c_{Z'}\varkappa (f_{Z'})$. Moreover, the transition matrix $P_{2}P_{1}^{-1}$ does not depend on the choice of $\prec$ (up to a permutation). In particular, the function $sgn$ does not depend on the choice of $\prec$, either.
\end{remark}

\begin{definition}
	For $X,X' \in \mathcal{V}_{0}$, we define $[X] \preceq' [X']$ if and only if for any order $\prec$ of $I$ and the induced map $\underline{s}'^{\prec}: \mathcal{V}_{0} \rightarrow \mathcal{S}$, we always have $\underline{s}'^{\prec}([X]) \prec \underline{s}'^{\prec}([X'])$. 
	
	Similarly, given $[L],[K] \in \mathcal{V}_{1}(\Lambda)$, we say $[L] \preceq [K]$ if and only if for any order $\prec$ of $I$ and the induced map $\underline{s}^{\prec}: \mathcal{V}_{1} \rightarrow \mathcal{S}$, we have $\underline{s}^{\prec}([L]) \prec \underline{s}^{\prec}([K])$.  
\end{definition}

One can check that $X \preceq' X'$  if and only if $\Phi^{\Lambda}(X) \preceq \Phi^{\Lambda}(X')$.

\begin{corollary}
	The transition matrix from $\tilde{\mathbf{B}}_{1}$ to $\tilde{\mathbf{B}}_{2}$ is upper triangular  (with respect to $\preceq$ and $\preceq'$) and with diagonal entries all equal to 1.
	More precisely, if $X \in \mathcal{V}_{0}$ and $[L]=\Phi^{\Lambda}(X)$, then we have 
	\begin{equation*}
		\varsigma^{\Lambda}([L])=\varkappa^{\Lambda}(sgn(X)[X]) +\sum\limits_{X \preceq X' } c_{X'}\varkappa^{\Lambda} (sgn(X')[X'])
	\end{equation*}
	where $c_{X'} \in \mathbb{Q}$ are constants.
\end{corollary}

\end{document}